\documentclass[11pt]{amsart}

\usepackage{amsfonts, enumerate} %amsfonts %queste annotazioni mi dicono il parametro che potrei usare, ad esempio in questo caso equivale a dire che potrei togliere "enumerate"
\usepackage{amssymb}
\usepackage[final]{graphicx, graphics,stmaryrd} %graphicx

\usepackage[margin=3cm]{geometry} 
\usepackage[english]{babel}
\usepackage[utf8]{inputenc}
\usepackage{amsthm}
\usepackage{graphics}
\usepackage{amsmath}
\numberwithin{equation}{section} %numerazione delle equazioni 
\usepackage{amstext}

\usepackage[safe,extra]{tipa}
\usepackage{multirow}
\usepackage{enumitem}

\usepackage[colorlinks=true,urlcolor=blue,
citecolor=red,linkcolor=blue,linktocpage,pdfpagelabels,
bookmarksnumbered,bookmarksopen]{hyperref}
\usepackage[hyperpageref]{backref}

\usepackage{bm} %simboli BOLD in math mode
\usepackage{stackrel} %per scrivere cose impilate in math mode
\usepackage{tikz-cd} %diagrammi
\usepackage{mathtools} %%???

\newtheorem{teo}{Theorem}[section]
\newtheorem{coro}[teo]{Corollary}
\newtheorem{prop}[teo]{Proposition}
\newtheorem{lemma}[teo]{Lemma}

\theoremstyle{definition}
\newtheorem{defin}[teo]{Definition}
\newtheorem{rmk}[teo]{Remark}

\newcommand{\mfk}{\mathfrak}

\def\Oo{\mathcal O}
\def\R{\mathbb{R}}

\def\g{\mfk{g}}
\def\d{\mathrm{d}}

\newcommand{\al}{\alpha}
\newcommand{\be}{\beta}
\def\G{\mathbb{G}}

\newcommand{\m}{\mbox}

\newcommand{\cor}{\textit}

\newcommand{\fine}{\qed\newline}

\DeclareMathOperator{\supp}{supp}

\DeclareMathOperator{\In}{in}
\DeclareMathOperator{\Out}{out}

\def\din{\Omega^{\In}}
\def\dout{\Omega^{\Out}}

\makeatletter
\@namedef{subjclassname@2020}{\textup{2020} Mathematics Subject Classification}
\makeatother

\begin{document}
 	\title[Critical singular problems in Carnot groups]{Critical singular problems in Carnot groups}
 	
 	\author[S.\,Biagi]{Stefano Biagi}
 	\author[M.\,Galeotti]{Mattia Galeotti}
 	\author[E.\,Vecchi]{Eugenio Vecchi}
 	
 	\address[S.\,Biagi]{Dipartimento di Matematica
 		\newline\indent Politecnico di Milano \newline\indent
 		Via Bonardi 9, 20133 Milano, Italy}
 	\email{stefano.biagi@polimi.it}
 	
 	\address[M.\,Galeotti]{Dipartimento di Matematica
 		\newline\indent Università  di Bologna \newline\indent
 		Piazza di Porta San Donato 5, 40122 Bologna, Italy}
 	\email{mattia.galeotti4@unibo.it}
 	
 	\address[E.\,Vecchi]{Dipartimento di Matematica
 		\newline\indent Università  di Bologna \newline\indent
 		Piazza di Porta San Donato 5, 40122 Bologna, Italy}
 	\email{eugenio.vecchi2@unibo.it}
 	
 	\keywords{PDEs on Carnot groups, Singular equations, Critical equations}
 	
 	\subjclass[2020]{35R03, 35B33, 35B25, 35J70}
 	%35R03 PDEs on Heisenberg groups, Lie groups, Carnot groups, etc.
 	%35J75 Singular elliptic equations
 	%35J20 Variational methods for second-order elliptic equations
 	%35B33 Critical exponents in context of PDEs
 	%35J70 Degenerate elliptic equations
 	%35H20 Subelliptic equations
 	%35B25 Singular perturbations in context of PDEs
 	%35B09 Positive solutions to PDEs
 	
 	\date{\today}
 	
 	\thanks{S.B. and E.V. are
 		member of the {\em Gruppo Nazionale per
 			l'Analisi Ma\-te\-ma\-ti\-ca, la Probabilit\`a e le loro Applicazioni}
 		(GNAMPA) of the {\em Istituto Nazionale di Alta Matematica} (INdAM), and are
 		partially 
 		supported by the PRIN 2022 project 2022R537CS \emph{$NO^3$ - Nodal Optimization, NOnlinear elliptic equations, NOnlocal geometric problems, with a focus on regularity}, founded by the European Union - Next Generation EU and by the Indam-GNAMPA project CUP E5324001950001 - {\em Problemi singolari e degeneri: esistenza, unicità e analisi delle proprietà qualitative delle soluzioni}.\\
 		M.G. is supported by the National Recovery and Resilience Plan (NRRP), research funded by the European Union – NextGenerationEU.
 		CUP D93C22000930002, {\em A multiscale integrated approach to the study of the nervous system in healt and disease (MNESYS)}}

 \begin{abstract}
 	We consider a power-type mild singular perturbation of a Dirichlet semilinear critical problem settled
 	in an open and bounded set in a Carnot group. 
 	Here, the term critical has to be understood in the sense of the Sobolev embedding.
 	We aim to prove
 	the existence of two positive weak solutions: the first one is obtained
 	by means of the variational Perron's method, while  for the second
 	one we adapt a classical argument relying on proper estimates of 
 	a family of functions which mimic the role of the classical
 	Aubin-Talenti functions in the Euclidean setting.
 	
 	Our results fall in the framework of semilinear PDEs in Carnot group but,
 	as far as we know, are the first ones dealing with singular perturbations
 	of power-type.
 \end{abstract}

\maketitle

\section{Introduction}
Let $\mathbb{G}$ be a Carnot group and let $\Omega \subset \mathbb{G}$ be an bounded and connected
open set with smooth enough boundary $\partial \Omega$. Let $\gamma \in (0,1)$, let $2^{\star}_{Q}:=\tfrac{2Q}{Q-2}$ be the critical Sobolev exponent related to the Sobolev inequality in $\mathbb{G}$, and let $\lambda >0$. We consider the following singular Dirichlet problem
\begin{equation}\tag{{$\mathrm{P}$}}\label{eq:Main_Problem}
	\left\{\begin{array}{rl}
		-\Delta_{\mathbb{G}}u = \dfrac{\lambda}{u^{\gamma}} + u^{2^{\star}_{Q}-1} & \textrm{ in } \Omega,\\
		u>0 & \textrm{ in } \Omega,\\
		u=0 & \textrm{ on } \partial \Omega.
	\end{array}\right.
\end{equation}
Along the paper it will sometimes be useful to denote the above problem as \eqref{eq:Main_Problem}$_\lambda$ 
to make it clear the choice of the parameter.
We immediately state the main result of this paper. In what follows,
we refer to Definition \ref{def:weak_sub_super_sol} for the precise definition of \emph{weak solution}
of \eqref{eq:Main_Problem}$_{\lambda}$.
\begin{teo} \label{thm:main}
	Let $\Omega\subset\mathbb{G}$ be an open and bounded set
	with smooth enough boundary $\partial \Omega$, and let $\gamma\in (0,1)$. Then, there exists $\Lambda > 0$
	such that
	\begin{itemize}
		\item[a)] problem \eqref{eq:Main_Problem}$_{\lambda}$ admits at least two weak solutions
		for every $0<\lambda<\Lambda$;
		\item[b)]  problem \eqref{eq:Main_Problem}$_{\Lambda}$ admits at least one weak solution;
		\item[c)]
		problem \eqref{eq:Main_Problem}$_{\lambda}$ does not admit weak solutions
		for every $\lambda>\Lambda$.
	\end{itemize}
\end{teo}

The above theorem is the natural generalization to Carnot groups of classical results of Haitao \cite{Haitao} and Hirano, Saccon and Shioji \cite{Hirano}, where the authors considered critical perturbations of mild singular terms using sub and supersolution methods (\cite{Haitao}) or a Nehari manifold approach (\cite{Hirano}). We stress that similar results have been obtained with different leading operators, see e.g. \cite{GiacMukSre, KRS, BV}.
Before commenting on the proof of Theorem \ref{thm:main}, we want to give a brief account of the existing literature concerning partial differential equations (PDEs, for short) on Carnot groups,
with a focus on \emph{critical semilinear equations}. 
\vspace{0.1cm}

\noindent -\,\,\emph{PDEs on Carnot groups}. To begin with, it is worth mentioning that, if $\mathbb{G}$ is a Carnot group and if $\Delta_\mathbb{G}$ is a sub-Lapla\-cian on $\mathbb{G}$ 
(see Section \ref{sec:Prel} for the relevant definitions), then $\Delta_\mathbb{G}$ is a second-order differential operator with 
\emph{non-negative characteristic form}, and thus it falls in the class of 
\emph{degenerate-elliptic operators.}
These degenerate operators
have appeared in the literature since the early 1900s, due to their appearance
in  models of theoretical physics and of diffusion processes; however,
the lack of \emph{regularizing properties} (caused by the lack of ellipticity) 
creates several difficulties and prevents the application of several techniques.

In this perspective, a major difference between \emph{general}
degenerate-elliptic operators and the sub-Laplacians on Carnot groups
is that $\Delta_\mathbb{G}$
is a \emph{sum of squares} of vector fields satisfying 
the celebrated 
{\em H\"{o}rmander hypoellipticity condition}, see \cite{Hormander}.
Starting from this fact, and by taking benefit of the underlying geometry
attached to $\Delta_\mathbb{G}$ induced by $\mathbb{G}$, 
G.B. Folland, L.P. Rothschild and E.M. Stein developed in the 70's the 
\emph{singular integral theory in nilpotent Lie
groups} (see, e.g., Folland's survey
\cite{Folland_per_Stein} for a torough discussion of the contributions
of Stein in this context). More precisely, in 1975  Folland \cite{Folland} 
accomplished a functional
analytic study of sub-Laplacians on Carnot Lie groups and proved
the existence of an associated \emph{well-behaved global fundamental solution}.
One year later
Rothschild and Stein \cite{RotSte} proved their celebrated 
\emph{lifting theorem}
enlightening the fundamental role played by the sub-Laplacians in the theory of second order PDEs which are sum of squares of vector fields; this remarkable result paved the way for a deep study of PDEs on Carnot groups.
\vspace{0.1cm}

\noindent-\,\,\emph{Critical PDEs on Carnot groups}.
Firstly, when $\mathbb{G}= \mathbb{H}^{n}$ is the Heisenberg group and $\Omega = \mathbb{H}^n$, problem \eqref{eq:Main_Problem}$_0$ (i.e. with $\lambda =0$) coincides with the CR-Yamabe problem which has been deeply studied in a series of papers by Jerison and Lee \cite{JerisonLee, JerisonLee2, JerisonLee3}, in connection with the existence of extremals for the associated Sobolev inequality. Further results concerning the CR-Yamabe problem can be found in e.g. \cite{Gamara,GaYa,ChMaYa}. In \cite{JerisonLee2}, Jerison and Lee actually provided the explicit expression of such extremals, resembling the Euclidean ones by Aubin and Talenti. After the seminal paper by Br\'{e}zis and Nirenberg \cite{BN}, it  became clear that the explicit knowledge of these functions was a key tool to attack the study of critical PDEs both in bounded and unbounded domains, at least in Heisenberg groups.
 In this perspective, we refer e.g.~to \cite{GaLa, BiCa, LuWei, BrRiSe, BiPr, Ugu1, Ugu2, LaUg, CitUg, MaUg, FelliUgu, MaMaPi, PaPiTe} where several existence and non-existence results have been proved for the critical (or slightly sub-critical) equations and to
\cite{Citti, GaUg, MaMa} for perturbation results in the spirit of Br\'{e}zis and Nirenberg. We refer to e.g. \cite{MaMaTr, MolicaRepovs} and the references therein for the case of sign-changing solutions.

The situation may a priori change when considering structures different from Heisenberg groups. In \cite{GaroVa2, GaroVa} Garofalo and Vassilev exhibited a family of minimizers for the Sobolev inequality in groups of Iwasawa type. As far as we know, there are no other structures, nor Sobolev inequalities with $p\neq 2$, for which the minimizers are explicitly known. Nevertheless, it is worth to mention that the best constant in the Sobolev inequality is achieved in all Carnot groups, see \cite{GaroVa2}. Moreover, the asymptotic behaviour at infinity of the minimizers have been found also in the case $p\neq 2$, see \cite{Loiudice3}. Going back to critical PDEs, thanks to \cite{BoUg}, Loiudice in \cite{Loiudice1} obtained the sufficient asymptotic expansions of a family of functions naturally associated with the extremals, and this was enough to obtain a perturbation result à la Br\'{e}zis-Nirenberg in a general Carnot group $\mathbb{G}$. We briefly mention that the above mentioned papers are mainly interested in the existence, multiplicity or non-existence of positive solutions. Following this line of research, in \cite{Loiudice2, Loiudice4} Loiudice considered the case of singular perturbations of critical problems, where the term {\em singular} has to be considered in the sense of a suitable Hardy-type potential. Our interest in the present paper is to consider {\em mild singular perturbation} of the form $u^{-\gamma}$ with $\gamma \in (0,1)$ (hence the term {\em mild}). Since the seminal paper by Crandall, Rabinowitz and Tartar \cite{CRT}, the literature dealing with singular problems of this kind in the Euclidean setting, even for $\gamma \geq 1$, has seen a great amount of contributions: we refer to the recent survey \cite{OP3} and the references therein for a very detailed account. On the other hand, to the best of our knowledge, this seems to be the first contribution which considers singular power-type perturbations in the setting of Carnot groups.

Let us now briefly comment on the main result stated in Theorem \ref{thm:main}. The proof follows the argument performed in the Euclidean case by Haitao \cite{Haitao}, suitably adapted to the Carnot group setting, and it consists of several technical steps that we list here below:
\begin{itemize}
	\item we prove the existence of a first solution by means of a variational sub and supersolution scheme, adapting the approach of Struwe \cite{Struwe}. The subsolution is naturally provided by the unique solution of the purely singular problem, see Theorem \ref{thm:Singular_Problem}, while the supersolution is constructed in Lemma \ref{lem:2.3Haitao}.
	We notice that this scheme immediately provides a threshold $\Lambda$ (see \eqref{eq:DefinitionLambda}) for the non-existence; 
	\item we show that for $\lambda \in (0,\Lambda)$ the first solution obtained as described before is a local minimizer in the natural topology associated with problem \eqref{eq:Main_Problem}, see Lemma \ref{lem:25Haitao};
	\item we follow an argument originally due to Tarantello \cite{Tarantello} (see also \cite{BadTar}) to prove the existence of a second solution. Here we heavily employ the asymptotic expansions found in \cite{Loiudice1}, taking care of the new singular term, and we adapt the Euclidean estimates of \cite{BrezisNirenberg} exploiting the properties of the convolution in Carnot groups proved in \cite{FollandStein}.
\end{itemize}

\medskip

The paper is organized as follows: in Section \ref{sec:Prel} we recall the basic facts on Carnot groups needed in what follows, like the Folland-Stein spaces which provide the natural variational framework where problem \eqref{eq:Main_Problem} is set, strong maximum principle and weak Harnack inequality for the operator $-\Delta_{\mathbb{G}} + c$ which we were not able to find in the literature but is probably well known to experts in the field. We also provide the basic result for the purely singular problem whose importance has been already described. Section \ref{sec:First_Solution} is devoted to find the first solution while the second solution (for $\lambda \in (0,\Lambda)$) is found in Section \ref{sec:Second_Solution}.

\section{Preliminaries}\label{sec:Prel}
In this section we collect all the relevant notations, definitions and preliminaries needed in the rest of the paper.
\subsection{Carnot groups}
A Carnot group $\mathbb{G}=(\mathbb{R}^{N},\diamond)$ of step $k$ is a connected, simply connected Lie group whose finite dimensional Lie algebra $\mathfrak{g}$ of left-invariant (w.r.t. $\diamond$) vector fields admits a stratification of step $k$, namely there exist $k$ linear subspaces $\mathfrak{g}_1, \ldots, \mathfrak{g}_k$ such that
\begin{equation*}
	\mathfrak{g} = \mathfrak{g}_{1}\oplus \ldots \oplus \mathfrak{g}_{k}, \qquad [\mathfrak{g}_1,\mathfrak{g}_i]=\mathfrak{g}_{i+1}, \qquad \mathfrak{g}_{k}\neq \{0\}, \qquad \mathfrak{g}_i = \{0\} \textrm{ for all } i>k.
\end{equation*}
In particular, this implies that Carnot groups are a special instance of graded groups.\\
We call $\mathfrak{g}_1$ the horizontal layer. We denote by $X_1, \ldots, X_{N}$ a basis of left-invariant vector fields of $\mathfrak{g}$ such that the following holds:
\begin{itemize}
	\item $X_1, \ldots, X_{m_1}$ is a basis of $\mathfrak{g}_{1}$;
	\item for every $1<i\leq k$, $X_{m_{i-1}+1},\ldots, X_{m_{i}}$ is a basis of of $\mathfrak{g}_{i}$;
	\item $m_{0}=0$ and $n_i := m_{i}-m_{i-1} = \dim \mathfrak{g}_i$ for every $1\leq i\leq k$;
	\item $m_1 + \ldots + m_k = N$.
\end{itemize}
We notice that $N$ is the topological dimension of $\mathbb{G}$, but we can also define its {\em homogeneous dimension} $Q$ as follows
\begin{equation}\label{eq:Def_Q}
	Q:= \sum_{i=1}^{k}i \, n_{i}.
\end{equation}
We notice that $N \leq Q$ and that $N=Q$ if and only if $\mathbb{G}$ is the classical Euclidean group $(\mathbb{R}^{N},+)$. In particular, this is the only possible case whenever $Q \leq 3$.\\
Since the exponential map is a one-to-one diffeomorphism from $\mathfrak{g}$ to $\mathbb{G}$, any point $g\in \mathbb{G}$ can be uniquely written in exponential coordinates as
\begin{equation*}
	g = g_1 X_1 + \ldots + g_N X_N = (g_1, \ldots, g_N).
\end{equation*}
Being a graded group, every Carnot group $\mathbb{G}$ possess a family of anisotropic dilations $\delta_{\lambda}: \mathbb{G}\to \mathbb{G}$ defined as
\begin{equation}\label{eq:dilation}
	\delta_{\lambda}(g) = \left(\lambda^{\alpha_1}g_1, \ldots, \lambda^{\alpha_{N}}g_N \right), \quad \textrm{ for every } \lambda >0,
\end{equation}
\noindent where $\alpha_{j} = i$ if $m_{i-1}<j\leq m_{i}$.
We notice that $Q = \alpha_1 + \ldots + \alpha_{N}$.\\
An explicit expression of the group operation $\diamond$ can be then determined by means of the Campbell-Baker-Hausdorff formula, see e.g. \cite{BLU} for more details. The null element of $\diamond$ is the identity $0 = (0,\ldots,0)$ and the inverse of a certain $g \neq 0$ is usually denoted by $g^{-1}$.\\
The group operation $\diamond$ can also be used to define a further family of automorphisms of $\mathbb{G}$ known as {\em left translations} $\tau : \mathbb{G}\to \mathbb{G}$. More precisely, given a base point $h\in \mathbb{G}$, we define 
\begin{equation}\label{eq:left-translation}
	\tau_{h}(g):= h \diamond g.
\end{equation}

\medskip

It is possible to endow a Carnot group $\mathbb{G}$ with a richer structure. Firstly, it is possible to define a scalar product $\langle \cdot, \cdot \rangle_{\mathfrak{g}_{1}}$ such that the basis $\{X_1, \ldots, X_{m_1}\}$ of the horizontal layer becomes orthonormal, and this provides a sub-Riemannian structure over $\mathbb{G}$. We notice that one can also provide a purely Riemannian structure defining a scalar product on the Lie algebra $\mathfrak{g}$ making the entire basis $\{X_1,\ldots,X_N\}$ orthonormal.
The sub-Riemannian Carnot group $\mathbb{G}$ can also be endowed with an intrinsic metric structure by means of the so called {\em Carnot-Carath\'{e}odory} ($CC$ in short) distance. It is well known that with this distance, these spaces are not Riemannian at any scale, see e.g. \cite{Semmes}.

\medskip

From a more analytic point of view, we can define several differential operators modelled on the horizontal vector fields
$\{X_1,\ldots, X_{m_1}\}$. 
First, given a smooth horizontal vector field $V=v_1 X_1 + \ldots + v_{m_1}X_{m_1}$, we define its horizontal divergence as
\begin{equation}\label{eq:horizontal_divergence}
	\mathrm{div}_{\mathbb{G}}V := X_{1}v_1 + \ldots + X_{m_1}v_{m_1}.
\end{equation}
Moreover, given a smooth enough scalar-valued function $u:\mathbb{G}\to \mathbb{R}$, we can define
the horizontal gradient of $u$ as
\begin{equation}\label{eq:Horizontal_grad}
	\nabla_{\mathbb{G}}u := (X_{1}u, \ldots, X_{m_1}u),
\end{equation}
\noindent and the sub-Laplacian of $u$ as
\begin{equation}\label{eq:horizontal_lap}
	\Delta_{\mathbb{G}}u := \mathrm{div}_{\mathbb{G}}(\nabla_{\mathbb{G}}u) = X_1^2 u + \ldots + X_{m_1}^{2} u,
\end{equation}
We point out that both the divergence $\mathrm{div}_{\mathbb{G}}$ and the horizontal gradient $\nabla_{\mathbb{G}}$ (and \cor{a fortiori} the sub-Laplacian $\Delta_{\mathbb{G}}$)
are independent of the choice of the base on $\g_1$: in fact, they are intrisically
associated to the sub-Riemannian structure on $\mathbb{G}$.

We notice that both the horizontal gradient $\nabla_{\mathbb{G}}$ and the sub-Laplacian $\Delta_{\mathbb{G}}$ are left-invariant operators, i.e.
\begin{equation}
	\nabla_{\mathbb{G}}(u\circ \tau_{h}) = (\nabla_{\mathbb{G}}u)\circ \tau_{h} \quad \textrm{ and } \quad \Delta_{\mathbb{G}}(u \circ \tau_{h}) = (\Delta_{\mathbb{G}}u)\circ \tau_{h}, \quad \textrm{ for every } h \in \mathbb{G},
\end{equation}
\noindent and they are, respectively, homogeneous of degree one and two w.r.t. the family of dilations $\delta_{\lambda}$ defined in \eqref{eq:dilation}, namely
\begin{equation}
	\nabla_{\mathbb{G}}(u\circ \delta_{\lambda})= \lambda (\nabla_{\mathbb{G}}u)\circ \delta_{\lambda} \quad \textrm{ and } \quad \lambda^{2}(\Delta_{\mathbb{G}}u)\circ \delta_{\lambda}, \quad \textrm{ for every } \lambda >0.
\end{equation}
The Lebesgue measure $\mathcal{L}^{N}$ coincides with the Haar measure of $\mathbb{G}$ and hence is left-invariant and satisfies the following scaling property:
\begin{equation}
	\mathcal{L}^{N}(\delta_{\lambda}(E)) = \lambda^{Q} \mathcal{L}^{N}(E) \quad \textrm{ for every measurable set } E \subset \mathbb{G}.	
\end{equation}
Every integral in this manuscript has to be understood
with respect to the Haar measure, unless otherwise stated.

Moreover, the homogeneity of the $X_i$'s implies that the (formal)
adjoint of $X_i$ in the Lebesgue space $L^2(\mathbb{G})$ is precisely $-X_i$ (for $i\leq i\leq m_1$), that is,
\begin{equation} \label{eq:XiSA}
 \int_{\mathbb{G}}(X_i\varphi)\psi = -\int_{\mathbb{G}}\varphi(X_i\psi)\quad\text{for
 every $1\leq i\leq m_1$}.
\end{equation}
In particular, $-\Delta_\mathbb{G}$ is a \emph{self-adjoint operator}.
\medskip

Every Carnot group can be endowed with several homogeneous norms. A homogeneous (quasi)norm $\rho : \mathbb{G} \to \mathbb{R}$ is a non-negative function further satisfying the following properties:
\begin{itemize}
	\item $\rho(g)=0$ if and only if $g=0$;
	\item $\rho(\delta_{\lambda}(g)) = \lambda \, \rho(g)$ for every $g \in \mathbb{G}$ and for every $\lambda >0$;
	\item $\rho(h\diamond g) \leq C \left(\rho(h) + \rho(g)\right)$ for every $g,h \in \mathbb{G}$ and for some constant $C \geq 1$.
\end{itemize}
The importance of such objects is witnessed by a famous result of Folland: in \cite{Folland}, he showed that there exists a homogeneous norm $|\cdot|_{\mathbb{G}}$ on $\mathbb{G}$ and a positive constant $C_Q>0$, depending only on $Q$, such that the function
\begin{equation}\label{eq:def_Gamma}
	\Gamma_{h}(g):= \dfrac{C_Q}{|h^{-1}\diamond g|^{Q-2}_{\mathbb{G}}}, \quad \textrm{ with } Q \geq 3,
\end{equation}
\noindent is a fundamental solution of $-\Delta_{\mathbb{G}}$ with pole at $h \in \mathbb{G}$.
Moreover, homogeneous norms can be used to define distances, different from the $CC$-distance, as follows: 
\begin{equation*}
	d_{\rho}(g,h):= \rho(h^{-1}\diamond g).
\end{equation*}
In any case, all these norms (and the relative distances) are equivalent and they all induce on $\mathbb{G}$ the Euclidean topology. For our purposes, we prefer to work with the homogeneous norm $|\cdot|_{\mathbb{G}}$ (and the associated distance $d_{\mathbb{G}}$) which provides the fundamental solution defined in \eqref{eq:def_Gamma}. In particular, we will denote by 
\begin{equation*}
	B_{r}(g_0) := \{ g \in \mathbb{G}: d_{\mathbb{G}}(g,g_0) = |g_{0}^{-1}\diamond g|_{\mathbb{G}} <r\},
\end{equation*} 
\noindent the open ball of radius $r>0$ and center $g_{0}\in \mathbb{G}$.

\medskip

\subsection{Folland-Stein spaces and critical PDEs}\label{sec:fsandpdes}
Let $\mathcal{O} \subseteq \mathbb{G}$ be an open set.
For every $f \in C^{\infty}_{0}(\mathcal{O})$ there exists a positive constant $C_Q>0$ depending only on the homogeneous dimension $Q$ such that the following Sobolev inequality holds true
\begin{equation}\label{eq:Sobolev_Ineq}
	\|f\|^{2}_{L^{2_{Q}^{\star}}(\mathcal{O})} \leq C_{Q} \, \| |\nabla_{\mathbb{G}}f| \|^{2}_{L^{2}(\mathcal{O})},
\end{equation}
\noindent where 
\begin{equation}
	2_{Q}^{\star} := \dfrac{2Q}{Q-2},
\end{equation}
\noindent denotes the (sub-elliptic) critical Sobolev exponent, resembling the classical Euclidean 
one
$$2^{\star}=\frac{2N}{N-2}.$$
Thanks to \eqref{eq:Sobolev_Ineq}, 
$\| |\nabla_{\mathbb{G}}f| \|_{L^{2}(\Omega)}$ provides a norm on the space $C^{\infty}_{0}(\Omega)$.
We define the Folland-Stein space $S^{1}_{0}(\mathcal{O})$ as the completion of $C^{\infty}_{0}(\mathcal{O})$ w.r.t.\,the above norm, and we set
$$\|u\|_{S^{1}_{0}(\mathcal{O})} = \||\nabla_\mathbb{G}u|\|_{L^2(\mathcal{O})}
\quad\text{for every $u\in S^{1}_{0}(\mathcal{O})$}.$$
We explicitly observe that, owing to \eqref{eq:Sobolev_Ineq}, we have
\begin{equation} \label{eq:explicitS01def}
 S_0^1(\mathcal{O}) = \big\{u\in L^{2_{Q}^{\star}}(\mathcal{O}):\,\text{$X_iu\in L^2(\mathcal{O})$
 for all $1\leq i\leq m_1$}\big\},
\end{equation}
where $X_1u,\ldots,X_{m_1}u$ are meant in the sense of distributions, that is (see also \eqref{eq:XiSA}),
$$\int_{\mathcal{O}}(X_ju)\varphi = -\int_{\mathcal{O}}uX_j\varphi\quad\text{for all $\varphi\in C_0^\infty(\mathcal{O})$}.$$
In the particular case when $\mathcal{O}$ \emph{is bounded} (which the is case we are mainly
interested in, together with the case $\mathcal{O} = \mathbb{G}$),
the space $S_0^1(\mathcal{O})$ enjoys the following classical properties.
\begin{enumerate}
 \item $S_0^1(\mathcal{O})$ is endowed with a structure of real Hilbert space by the inner product
 $$\langle u,v\rangle_{S_0^{1}(\mathcal{O})} = \int_{\mathcal{O}}\langle \nabla_\mathbb{G}u,
 \nabla_\mathbb{G}v\rangle_{\mathfrak{g}_{1}}\qquad (u,v\in S_0^1(\mathcal{O})),$$
whose associated norm is precisely $\|\cdot\|_{S_0^1(\mathcal{O})}$.

 \item By density, the Sobolev inequality \eqref{eq:Sobolev_Ineq} actually holds for every function $u\in S_0^1(\mathcal{O})$. As a consequence, $S_0^1(\mathcal{O})$ is \emph{continuously embedded} into $L^p(\mathcal{O})$ for every $1\leq p\leq 2_Q^\star$.
 Furthermore, this embedding turns out to be \emph{compact} when $1\leq p < 2_Q^\star$. 
 \item If $u\in S_0^1(\mathcal{O})$, then $u_+ = \max\{u,0\},\,u_- = \max\{-u,0\}\in S_0^1(\mathcal{O})$, and
 $$\nabla_\mathbb{G}u_+ = \nabla_\mathbb{G}u\cdot\chi_{\{u > 0\}}
 \quad \textrm{ and } \quad  \nabla_\mathbb{G}u_- = \nabla_\mathbb{G}u\cdot\chi_{\{u < 0\}}.
 $$
 In particular, for every $c \in\mathbb{R}$, we derive
 \begin{equation}\label{eq:Grad_and_level_set}
 	\nabla_\mathbb{G}u = 0  \quad \textrm{a.e. on any level set } \{u = c\}.
 \end{equation}
 \end{enumerate}
We refer, e.g., to \cite{FSSC} for good approximation results in terms of smooth function, allowing to prove the above facts.
\begin{rmk} \label{rem:ConvergenceS01}
On account of the above properties of $S_0^1(\Omega)$, it is possible to prove
the following \emph{convergence result}, which will be repeatedly used in the sequel.

Assume that $\{u_k\}_k\subseteq S_0^1(\Omega)$ is a \emph{bounded sequence}. Since $S_0^1(\Omega)$ is a real \emph{Hilbert space}, and since we have already pointed out that the embedding 
  $$S_0^1(\Omega)\hookrightarrow L^p(\Omega)$$
  is compact for every $1\leq p< 2^\star_Q$, as in the classical case we deduce that there exists a function $u\in S_0^1(\Omega)$ such that (up to a subsequence, and as $k\to+\infty$)
  \vspace{0.1cm}
  
  \begin{itemize}
  	\item  $u_k\to u$ \emph{weakly in $S_0^1(\Omega)$};
  	\item $u_k\to u$ \emph{strongly} in $L^p(\Omega)$ for every $1\leq p< 2^\star_Q$;
  	\item $u_k\to u$ \emph{pointwise a.e.\,in $\Omega$}.
  \end{itemize}
  On the other hand, since the embedding $S_0^1(\Omega)\hookrightarrow L^{2^*Q}(\Omega)$ is \emph{continuous}, we have that $u\in L^{2^\star_Q}(\Omega)$, and  $\{u_k\}_k$ is \emph{bounded also in $L^{2^\star_Q}(\Omega)$}; in particular, setting 
  $$p_0 = \frac{2^\star_Q}{2^\star_Q-1}\in (1,2^\star_Q),$$
  we have that $\{v_k = u_k^{2^\star_Q-1}\}_k$ is bounded in $L^{p_0}(\Omega)$. As a consequence of this fact, and since we know that $u_k\to u$ po\-int\-wise a.e.\,in $\Omega$, we deduce that $\{v_k\}_k$ \emph{converges weakly to $u^{2^\star_Q-1}$}
  (as $k\to +\infty)$, and up to a subsequence) in $L^{p_0}(\Omega)$, that is, 
  $$\int_{\Omega}u_k^{2^\star_Q-1}\varphi\,dx \to \int_{\Omega}u^{2^\star_Q-1}\varphi\,dx\quad\text{for every $\varphi\in L^{p_0'}(\Omega)
  = L^{2^\star_Q}(\Omega)\supseteq S_0^1(\Omega)$}.$$
\end{rmk}

\medskip

In \cite{GaroVa} Garofalo and Vassilev adapted the celebrated {\em concentrantion-compactness principle} of Lions showing that the best Sobolev constant in \eqref{eq:Sobolev_Ineq} can be achieved and it is characterized as follows
\begin{equation}\label{eq:def_best_constant}
	S_{\mathbb{G}} := \inf_{f\in S^{1}_{0}(\mathbb{G})} \dfrac{\| |\nabla_{\mathbb{G}}f| \|^{2}_{L^{2}(\mathbb{G})}}{	\|f\|^{2}_{L^{2_{Q}^{\star}}(\mathbb{G})}}.
\end{equation}
As in the Euclidean case, this fact has immediate consequences at the level of the critical PDE
\begin{equation}\label{eq:Critical_PDE_in_G}
	-\Delta_{\mathbb{G}}u = u^{2^{\star}_{Q}-1} \quad \textrm{in } \mathbb{G}.
\end{equation}
We list here below the most important results concerning \eqref{eq:def_best_constant} and  \eqref{eq:Critical_PDE_in_G}.

\begin{itemize}
	\item Every minimizing sequence of \eqref{eq:def_best_constant} is relatively compact in $S^{1}_{0}(\mathbb{G})$, after possibly translating and dilating each of its elements. In particular, the minimum in \eqref{eq:def_best_constant} is achieved and \eqref{eq:Critical_PDE_in_G} admits a non-negative and non-trivial solution $U \in S^{1}_{0}(\mathbb{G})$, see \cite[Theorem 6.1]{GaroVa}.
	\item By Bony's maximum principle \cite{Bony}, every non-negative solution of \eqref{eq:Critical_PDE_in_G} is actually positive.
	\item If $T \in S^{1}_{0}(\mathbb{G})$ is a positive solution of \eqref{eq:Critical_PDE_in_G}, then there exists a positive constant $M_1>0$ such that 
	\begin{equation}\label{eq:Stima_Bonf_Ugu}
		T(g) \leq M_1 \, \min \{1, |g|_{\mathbb{G}}^{2-Q}\}, \quad \textrm{ for every } g \in \mathbb{G},
	\end{equation} 
	\noindent see \cite[Theorem 3.4]{BoUg}.
	\item If $T \in S^{1}_{0}(\mathbb{G})$ is a positive solution of \eqref{eq:Critical_PDE_in_G}, then there exists a positive constant $M_2 >0$ such that
	\begin{equation}\label{eq:Stima_Loiudice}
		T(g) \geq M_2 \, \dfrac{|B_{1}(0)|}{(1+|g|_{\mathbb{G}})^{Q-2}}, \quad \textrm{ for every } g \in \mathbb{G},
	\end{equation}
	\noindent see \cite[Lemma 3.2]{Loiudice1}.
\end{itemize}

 Let now $T \in S^{1}_{0}(\mathbb{G})$ be a minimizer of \eqref{eq:def_best_constant}. For every $\varepsilon >0$ define the rescaled function
\begin{equation}\label{eq:Minimi_riscalati}
	T_{\varepsilon}(g) := \varepsilon^{(2-Q)/2} T \left(\delta_{1/\varepsilon}(g)\right).
\end{equation}
Let further be $R>0$ such that $B_{R}(0)\subset \Omega$ and let $\varphi \in C^{\infty}_{0}(B_{R}(0))$ be a cut-off function such that $0\leq \varphi \leq 1$ and $\varphi \equiv 1$ in $B_{R/2}(0)$. Define the function 
\begin{equation}\label{eq:def_u_varepsilon}
	U_{\varepsilon}(g):= \varphi(g) T_{\varepsilon}(g), \quad g \in \mathbb{G}.
\end{equation}
The following holds:
\begin{itemize}
	\item Up to multiplicative constants, $T_{\varepsilon}$ is a solution of \eqref{eq:Critical_PDE_in_G}.
	\item Due to scaling invariance, we can set $T$ as to have
	\begin{equation}
		\| |\nabla_{\mathbb{G}}T_{\varepsilon}|\|^{2}_{L^{2}(\mathbb{G})} = \| T_{\varepsilon}\|^{2^{\star}_{Q}}_{L^{2_{Q}^{\star}}(\mathbb{G})} = S_{\mathbb{G}}^{Q/2}.
	\end{equation}
	\item The function $U_{\varepsilon}$ satisfies the following estimates as $\varepsilon \to 0^{+}$
	\begin{align}
		\| |\nabla_{\mathbb{G}}U_{\varepsilon}|\|^{2}_{L^{2}(\mathbb{G})} & =  S_{\mathbb{G}}^{Q/2} + O(\varepsilon^{Q-2}) \label{eq:Stima_grad_u_varepsilon}\\
		 \|U_{\varepsilon}\|^{2^{\star}_{Q}}_{L^{2_{Q}^{\star}}(\mathbb{G})}& = S_{\mathbb{G}}^{Q/2} + O(\varepsilon^{Q}) \label{eq:Stima_u_varepsilon}, 
	\end{align}
	\noindent see \cite[Lemma 3.3]{Loiudice1}.
\end{itemize}	

\subsection{Strong Maximum Principle and Harnack inequality} Now we have reviewed
the basic concepts concerning the Carnot groups setting, and before starting
our study of problem \eqref{eq:Main_Problem}$_\lambda$, for a future reference 
we explicitly
state here below a \emph{Weak Harnack inequality} for 
$$\text{$L = -\Delta_\mathbb{G}+c(x)$ (with $c\leq 0$)},$$
from which we will also derive a \emph{Strong Maximum Principle.}
These results will be fundamental throughout the rest of the paper.

\begin{prop}[Weak Harnack inequality for $-\Delta_\mathbb{G}+c$] \label{prop:WHarnack}
 Let $\mathcal{O}\subseteq\R^n$ be a bounded open set, and let 
 $c\in L^\infty(\mathcal{O}),\,c\geq 0.$
 Moreover, let $u\in S_0^1(\mathcal{O})$ be a \emph{weak supersolution}
 of the equation 
 \begin{equation} \label{eq:PDEOrderzeroHarnack}
  \text{$-\Delta_\mathbb{G}u+cu = 0$ in $\mathcal{O}$},
  \end{equation}
 that is
 \emph{(}see also the subsequent Definition \ref{def:weak_sub_super_sol}\emph{)},
 $$\int_\mathcal{O}\langle\nabla_\mathbb{G}u,\nabla_\mathbb{G}\varphi\rangle_{\mathfrak{g}_{1}}
 +\int_\mathcal{O} cu\varphi\geq 0\quad\text{
 for every $\varphi\in C_0^\infty(\mathcal{O}),\,\varphi\geq 0$}.$$
 If, in addiction, $u\geq 0$ a.e.\,in $\mathcal{O}$, there exist
 constants $c_0 > 0$ and $p_0\in (0,1)$,
 \emph{independent of~$u$}, such that,
 for every $g_0\in\mathcal{O},\,r > 0$ with $B_{4r}(g_0)\Subset\mathcal{O}$, we have
 \begin{equation} \label{eq:HarnackInequ}
  \Big(-\!\!\!\!\!\!\!\int_{B_{3r}(g_0)}|u|^{p_0}\Big)^{1/p_0}\leq c_0\,\inf_{B_{r}}u.
 \end{equation}
\end{prop}
\begin{proof}
 The proof of this proposition is contained in the proof of \cite[Theorem 4.1]{GutLan},
 where the Authors establish a \emph{full Harnack inequality}
 for the \emph{non-negative weak solutions $u$} of 
 \begin{equation} \label{eq:PDEGutLan} 
  \text{$Lu = 0$ in $\Omega$},
 \end{equation}
 and $L$ is a general $X$-elliptic operator (of which $-\Delta_\mathbb{G}+c(\cdot)$ is a particular
 case).
 Indeed,
 to prove the cited Harnack inequality for a given non-negative weak solution $u$
 of \eqref{eq:PDEGutLan}, the Authors in \cite{GutLan} use
 an adaptation of the Moser iteration technique, and they show that
 \begin{itemize}
  \item[i)] since $u$ is, in particular, a \emph{weak subsolution} of \eqref{eq:PDEGutLan}, then
  $$\sup_{B_r(g_0)}u \leq C_1\Big(-\!\!\!\!\!\!\!\int_{B_{2r}(g_0)}|u|^{s}\Big)^{1/s},$$
  for some constants $C_1 > 0,\,s > 0$ independent of $u$;
  \vspace*{0.05cm}
  
  \item[ii)] since $u$ is, in particular, a \emph{weak supersolution} of \eqref{eq:PDEGutLan}, then
  $$\inf_{B_r(g_0)}u \geq C_2\Big(-\!\!\!\!\!\!\!\int_{B_{3r}(g_0)}|u|^{-p_0}\Big)^{-1/p_0},$$
  for some constant $C_2 > 0$ and every $p_0$ small enough;
  \vspace*{0.05cm}
  
  \item[iii)] there exists a constant $C_3 > 0$, only depending on $p_0$, such that
  $$\Big(-\!\!\!\!\!\!\!\int_{B_{3r}(g_0)}|u|^{p_0}\Big)^{1/p_0}
  \leq C_3\Big(-\!\!\!\!\!\!\!\int_{B_{3r}(g_0)}|u|^{-p_0}\Big)^{-1/p_0},$$
 \end{itemize}
 provided that $B_{4r}(g_0)\Subset\mathcal{O}$. Thus, since the proofs of i)\,-\,iii) are mutually independent,
 if $u$ is just a weak supersolution of \eqref{eq:PDEOrderzeroHarnack} (which is a particular
 case of \eqref{eq:PDEGutLan}), the demonstration of assertion ii)
 in \cite{GutLan} gives exactly the desired estimate \eqref{eq:HarnackInequ}.
\end{proof}
From Proposition \ref{prop:WHarnack}, and using a classical covering argument, we obtain the following
\begin{coro} \label{cor:BrezisNirenbergpernoi}
 Let $\mathcal{O}\subseteq\mathbb{G}$ be a bounded open set,
 and let $c\in L^\infty(\mathcal{O}),\,\text{$c\geq 0$}$. Furthermore, let
 $u\in S_0^1(\mathcal{O})$ be a weak superso\-lu\-tion of 
equation \eqref{eq:PDEOrderzeroHarnack} satisfying
$$\text{$u > 0$ a.e.\,on every open ball $B = B_r(g_0)\Subset\Omega$}.$$
Then, for every open set $\mathcal{O}'\Subset\mathcal{O}$ there exists $C = C(\Oo',u) > 0$ such that
$$\text{$u\geq C(\Oo',u) > 0$ a.e.\,in $\Oo'$}.$$
\end{coro}
As anticipated, Proposition \ref{prop:WHarnack} implies the following \emph{Strong Maximum Principle}.
\begin{prop}[Strong Maximum Principle for $L$] \label{prop:SMP}
 Let $\mathcal{O}\subseteq \mathbb{G}$ be a \emph{bounded and connected} 
 open set, and let
 $u\in S^1_0(\mathcal{O}),\,u\geq 0$ be a \emph{weak supersolution} of \eqref{eq:PDEOrderzeroHarnack}. 
 Then,
 \begin{equation} \label{eq:SMPConclusion}
  \text{either $u \equiv 0$ or $u > 0$ a.e.\,in $\mathcal{O}$}.
 \end{equation}
\end{prop}
\begin{proof} 
 We assume that there exists a set $\mathcal{Z}\subseteq\mathcal{O}$ of \emph{positive Lebesgue measure} such that $u \equiv 0$ (pointwise) on $\mathcal{Z}$, and we prove that in this case we have $u\equiv 0$ a.e.\,on $
\mathcal{O}$.
\vspace{0.1cm}

To this end we first observe that, since $|\mathcal{Z}| > 0$, we can find $r_0 > 0$ such that
$$\mathcal{Z}_0 = \mathcal{Z}\cap \{g\in\mathcal{O}:\,d_\mathbb{G}(g,\partial\mathcal{O}) \geq r_0\}$$
 has \emph{positive Lebesgue measure}; moreover, since $K = \{g\in\mathcal{O}:\,d_\mathbb{G}(g,\partial\mathcal{O}) \geq r_0\}\subseteq\Omega$
 is  com\-pact (recall that $d_\mathbb{G}$ induces 
 the Euclidean topology), there exist $g_1,\ldots,g_p\in K$ and $r_1,\ldots,r_p > 0$ (for some $p\in\mathbb{N}$) such that $B_{4r_i}(g_i)\Subset\mathcal{O}$ (for $1\leq i\leq p$) and
 $$\mathcal{Z}_0 = \bigcup_{i = 1}^p (\mathcal{Z}_0\cap B_{r_i}(g_i)).$$
 As a consequence, there exists a $d_\mathbb{G}$-ball $B_0 = B_{r_{i_0}}(g_{i_0})$
 (for some $1\leq i_0\leq p$) such that
 \begin{equation} \label{eq:Z0capBpositiveMeas}
  |\mathcal{Z}_0\cap B_{r_{i_0}(g_{i_0})}| > 0.
 \end{equation}	
 In particular, since $u\geq 0$ a.e.\,in $\mathcal{O}$ and $u = 0$ on $\mathcal{Z}_0$, we have
 \begin{equation} \label{eq:infuzeroZ0}
 	\inf_{B_{r_{i_0}}(g_{i_0})} u = \inf_{\Omega}u = 0.
 \end{equation}
 With \eqref{eq:Z0capBpositiveMeas}-\eqref{eq:infuzeroZ0}
at hand, we then define the set
 $$S = \{g\in\mathcal{O}:\,\text{there exists a ball 
 $B = B_r(g)\subseteq\mathcal{O}$ 
 s.t.\,$u = 0$ a.e.\,on $B$}\},$$
 and we prove that $S = \mathcal{O}$ by a \emph{connection argument}:
  \vspace{0.1cm}
 
 -\,\,\emph{$S\neq\emptyset$}. Let $g_{i_0}\in\mathcal{O}$ be as in \eqref{eq:Z0capBpositiveMeas}-\eqref{eq:infuzeroZ0}. Since $u$ is a non-negative weak supersolution of equation \eqref{eq:PDEOrderzeroHarnack}, and since
 $B = B_{4r_{i_0}(g_{i_0})}\Subset\mathcal{O}$, we can apply the Weak Harnack Inequality in Proposition \ref{prop:WHarnack}: this gives the following estimate
 $$\Big(-\!\!\!\!\!\!\!\int_{B_{3r_{i_0}}(g_{i_0})}|u|^{p_0}\Big)^{1/p_0}\leq c\,\inf_{B_{r_{i_0}}(g_{i_0})}u = 0,
$$
from which we derive that $u = 0$ a.e.\,in $B_{3r_{i_0}}$. Hence, $g_{i_0}\in S$.
 \medskip

 -\,\,\emph{$S$ is open}. Assume that $g\in S$, and let $B = B_r(g)\subseteq \mathcal{O}$ be a 
 $d_\mathbb{G}$-ball such that  
 $$\text{$u = 0$ a.e.\,in $B$}.$$ 
 Given any $g'\in B$, if we choose $\rho > 0$ in such a way that $B_\rho(g')\subseteq B$, we obviously have
  $$\text{$u = 0$ a.e.\,on $B_{\rho}(g')\subseteq B$},$$
 and this proves that $g'\in S$.
 By the arbitrariness of $g'\in B_r(g)$, we  conclude that $B_r(g)\subseteq S$, and hence $S$ is open, as desired.

 \medskip

-\,\,\emph{$S$ is closed}. Assume that $(g_k)_k$ is a sequence of points in $S$ which converges (as $k\to+\infty$) to some $g\in\mathcal{O}$, and let $\rho > 0$
be such that $B_{4\rho}(g)\Subset\mathcal{O}$. Since $g_k\to g$ as $k\to+\infty$, there
exists some $k_0\in\mathbb{N}$ such that $g_{k_0}\in B_\rho(g)$; on the other hand, since
$g_{k_0}\in S$, it is possible to find a suitable $d_\mathbb{G}$-ball $B = B_r(g_{k_0})$
such that $B \subseteq B_\rho(g)$ and
$$\text{$u = 0$ a.e.\,in $B$}.$$
Hence, using once again the 
Weak
 Harnack Inequality in Proposition \ref{prop:WHarnack}, we get
$$\Big(-\!\!\!\!\!\!\!\int_{B_{3\rho}(g)}|u|^{p_0}\Big)^{1/p_0}\leq c\,\inf_{B_{\rho}(g)}u = 0,
$$
and thus $u = 0$ a.e.\,in $B_{3\rho}(g)$, that is, $g\in S$. Hence, $S$ is closed.
\medskip

Gathering the above facts, and recalling that $\mathcal{O}$ is connected, we then conclude that $\mathcal{O} = S$, and this obviously implies that
$u = 0$ a.e.\,in $\mathcal{O}$, as desired.
\end{proof}
\subsection{Weak sub/supersolutions of \texorpdfstring{$\eqref{eq:Main_Problem}_{\lambda}$}{Main Problem}}
We are now ready to properly set the definition of weak sub/supersolution of \eqref{eq:Main_Problem}$_\lambda$.
 However, we prefer to provide such definitions for the following slightly more general singular Dirichlet problem:
\begin{equation}\label{eq:General_Singular_Problem}
	\left\{\begin{array}{rl}
		-\Delta_{\mathbb{G}}u = \dfrac{\lambda}{u^{\gamma}} + f(x,u)& \textrm{ in } \Omega,\\
		u>0 & \textrm{ in } \Omega,\\
		u=0 & \textrm{ on } \partial \Omega,
	\end{array}\right.
\end{equation}
\noindent where $\gamma \in (0,1)$, $\lambda >0$ and $f: \Omega \times (0,+\infty)\to \mathbb{R}$ is a Carath\'{e}odory function satisfying the following critical growth condition: there exists a positive constant $K_{f}>0$ such that
\begin{equation}\label{eq:Growth_of_f}
	|f(x,t)|\leq K_{f} (1+t^{2^{\star}_{Q}-1})  \quad \textrm{ for a.e. } x \in \Omega \textrm{ and for  every } t>0.
\end{equation}
Here and throughout, $\Omega\subseteq\mathbb{G}$ is a fixed \emph{bounded and
connected open set} (as in \eqref{eq:Main_Problem}).
\begin{defin}\label{def:weak_sub_super_sol}
	Let $f: \Omega \times (0,+\infty)\to \mathbb{R}$ be a Carath\'{e}odory function satisfying the growth condition \eqref{eq:Growth_of_f}. We say that a function $u \in S^{1}_{0}(\Omega)$ is a weak subsolution (resp. supersolution) of \eqref{eq:General_Singular_Problem} if it satisfies the following properties:
	\begin{itemize}
		\item[(i)] $u>0$ in $\Omega$ and $u^{-\gamma} \in L^{1}_{\textrm{loc}}(\Omega)$.
		\item[(ii)] For every $0\leq \varphi \in C^{\infty}_{0}(\Omega)$, it holds that
		\begin{equation}\label{eq:weak_sub_super_sol}
			\int_{\Omega}\langle \nabla_{\mathbb{G}}u, \nabla_{\mathbb{G}}\varphi \rangle_{\mathfrak{g}_{1}} \leq (\textrm{resp. } \geq) \lambda \int_{\Omega}u^{-\gamma}\varphi + \int_{\Omega}f(x,u)\varphi.
		\end{equation}
	\end{itemize}
	Finally, we say that $u \in S^{1}_{0}(\Omega)$ is a weak solution of \eqref{eq:General_Singular_Problem} if it is both a weak subsolution and a weak supersolution of \eqref{eq:General_Singular_Problem} without the non-negativity condition on $\varphi$.
\end{defin}

 \begin{rmk} \label{rem:defweaksolPb}
	We now list here below some comments concerning the above Definition \ref{def:weak_sub_super_sol}. In what follows,
	we tacitly understand that $f:\Omega\to(0,+\infty)$ is a Carath\'eodory function satisfying
	the growth
	assumption \eqref{eq:Growth_of_f}.
	\medskip
	
	1)\,\,If $u\in S^{1}_{0}(\Omega)$ is a weak sub/supersolution of problem
	\eqref{eq:General_Singular_Problem}, all the integrals appearing in \eqref{eq:weak_sub_super_sol} 
	\emph{exist and are finite}.
	Indeed, by H\"older's inequality, for every 
	$\varphi\in C_0^\infty(\Omega),\,\text{$\varphi\geq 0$ in $\Omega$}$, 
	we obtain 
	\begin{equation} \label{eq:BrhouphiFinite}
				\left|\int_{\Omega}\langle \nabla_{\mathbb{G}}u, \nabla_{\mathbb{G}}\varphi \rangle_{\mathfrak{g}_{1}}\right|  \leq 
			\| u\|_{S_0^1(\Omega)}\|\varphi\|_{S_0^1(\Omega)}
	\end{equation}
	Moreover, using \eqref{eq:Growth_of_f} and the Sobolev inequality, we also get
	\begin{equation} \label{eq:termfFinite}
		\begin{split}
			\int_\Omega f(x,u)\,|\varphi|& \leq K_f\Big(\|\varphi\|_{L^1(\Omega)}
			+ \int_{\Omega}|u|^{2^{\star}_{Q}-1}|\varphi|\,dx\Big) \\
			& (\text{using H\"older's inequality}) \\
			& \leq  K_f\big(\|\varphi\|_{L^1(\Omega)}
			+\|u\|_{L^{2^{\star}_{Q}}(\Omega)}^{2^{\star}_{Q}-1}\cdot\|\varphi\|_{L^{2^{\star}_{Q}}(\Omega)}\big) \\
			& \leq K_f\big(\|\varphi\|_{L^1(\Omega)}
			+S_{\mathbb{G}}^{-2^{\star}_{Q}/2}\| u\|_{S_0^1(\Omega)}^{2^{\star}_{Q}-1}\cdot\| \varphi\|_{S_0^1(\Omega)}\big) \\
			& (\text{again by H\"older's inequality and Poincaré inequality}) \\
			& \leq C\| \varphi\|_{S_0^1(\Omega)}\big(1+\| u\|_{S_0^1(\Omega)}^{2^{\star}_{Q}-1}\big)<+\infty,
		\end{split}
	\end{equation}
	where $S_{\mathbb{G}} > 0$ is the best Sobolev constant in $\mathbb{G}$, and $C > 0$ depends on $Q,\,f,\,|\Omega|$.
	Fi\-nally, since \emph{we are assuming that $u^{-\gamma}\in L^1_{\mathrm{loc}}(\Omega)$}, we obviously
	have
	$$0\leq \int_\Omega u^{-\gamma}\varphi  \leq 
	\|\varphi\|_{L^\infty(\Omega)}\int_{\mathrm{supp}(\varphi)}u^{-\gamma}< +\infty. 
	$$
	We explicitly stress that this last estimate 
	(which is related with the \emph{singular term} $u^{-\gamma}$) is the unique estimate involving
	the $L^\infty$-norm of the test function $\varphi$; on the contrary, 
	estimates \eqref{eq:BrhouphiFinite}-\eqref{eq:termfFinite} only involve
	the \emph{$S_0^1$-norm} of $\varphi$.
	\vspace*{0.1cm}
	
	2)\,\,If $u\in S^{1}_{0}(\Omega)$ is a weak \emph{solution} of
	problem
	\eqref{eq:General_Singular_Problem}, and if $\varphi\in C_0^\infty(\Omega)$ is a {non-negative}
	test function (that is, $\varphi\geq 0$ in $\Omega$), by \eqref{eq:BrhouphiFinite}-\eqref{eq:termfFinite}
	we have
	\begin{align*}
		0\leq \int_\Omega u^{-\gamma}\varphi  & = 
		\frac{1}{\lambda}\Big(\int_{\Omega}\langle \nabla_{\mathbb{G}}u, \nabla_{\mathbb{G}}\varphi \rangle_{\mathfrak{g}_{1}}
		+ \int_\Omega f(x,u)\varphi \Big)
		\\
		& \leq C\| \varphi\|_{S_0^1(\Omega)}\big(1+\| u\|_{S_0^1(\Omega)}^{2^{\star}_{Q}-1}\big);
	\end{align*}
	from this, by using a standard \emph{density argument}, we can easily prove the following facts:
	\begin{itemize}
		\item[a)] $u^{-\gamma}\varphi\in L^1(\Omega)$
		\emph{for every $\varphi\in S^{1}_{0}(\Omega)$};
		\item[b)] identity \eqref{eq:weak_sub_super_sol}
		actually holds \emph{for every $\varphi\in S^{1}_{0}(\Omega)$}.
	\end{itemize}

	3)\,\,Let $u\in S^{1}_{0}(\Omega)$ be a weak \emph{solution} of problem \eqref{eq:General_Singular_Problem}.
	Since, in particular, we know that
	$u > 0$ a.e.\,in $\Omega$ (and $u\equiv 0$ a.e.\,in $\mathbb{G}\setminus\Omega$), it is quite easy
	to recognize that $u$ is a \emph{weak supersolution of the equation}
	$$-\Delta_{\mathbb{G}} u  = 0\quad\text{in $\Omega$},$$
	in the sense of Definition \ref{def:weak_sub_super_sol}.
	As a consequence, we are entitled to apply \cite{GutLan}, ensuring that
	\begin{equation*}
		\begin{gathered}
			\text{for every $\mathcal{O}\Subset\Omega$ there exists $c(\mathcal{O},u) > 0$ 
				s.t.\,$u\geq c(\mathcal{O},u) > 0$ a.e.\,in $\mathcal{O}$}.
		\end{gathered}
	\end{equation*}
\end{rmk}

\medskip

\subsection{Singular problem}\label{subsec:Singular}
The last preliminary result
we aim to present concerns the 
\emph{un\-per\-tur\-bed}, purely singular version of problem \eqref{eq:Main_Problem}$_\lambda$, that is,
\begin{equation}\label{eq:Singular_Problem}
	\left\{\begin{array}{rl}
				-\Delta_{\mathbb{G}}u = \dfrac{\lambda}{u^{\gamma}} & \textrm{ in } \Omega,\\
				u>0 & \textrm{ in } \Omega,\\
				u=0 & \textrm{ on } \partial \Omega,
	\end{array}\right.
\end{equation}
\noindent where $\gamma \in (0,1)$ and $\lambda >0$. We say that $u \in S^{1}_{0}(\Omega)$ is a weak solution of \eqref{eq:Singular_Problem} accordingly to Definition \ref{def:weak_sub_super_sol} with $f \equiv 0$.
In this context, we have the following theorem.

\begin{teo}\label{thm:Singular_Problem}
	Let $\Omega\subseteq\mathbb{G}$
	be a bounded open set. Moreover, let $\gamma \in (0,1)$ and $\lambda >0$. 
	
	Then, there exists a unique weak solution $w_{\lambda} \in S^{1}_{0}(\Omega) \cap L^{\infty}(\Omega)$ to \eqref{eq:Singular_Problem}. Moreover, $w_{\lambda}$ is a global minimizer in the $S^{1}_{0}(\Omega)$-topology
	of the functional 
		\begin{equation}\label{eq:def_Functional_Singular}
			J_{\lambda}(u):= \dfrac{1}{2}\int_{\Omega}|\nabla_{\mathbb{G}}u|^{2} 
			-\dfrac{\lambda}{1-\gamma} \int_{\Omega}|u|^{1-\gamma}.
		\end{equation}
		Finally, we also have that $J_{\lambda}(w_{\lambda})<0$.
   \end{teo}
		\begin{proof}
		The proof of this theorem is very similar
		to that of  \cite[Theorem 1.1]{SunWuLong} (for the  solvability of
		\eqref{eq:Singular_Problem}), and exploits a classical scheme of Stampacchia (for the
		global boundedness of $w_\lambda$); however, we present it
		here in detail for the sake of completeness.
		
			To ease the readability, we split the proof into five steps.
			\medskip
			
			\textsc{Step I).} In this first step we prove that $J_\lambda$ possesses a 
			\emph{global and strictly negative minimum}, which is attained by
			a function $w_\lambda\in S_0^1(\Omega)\setminus\{0\}$ 
			such that $w_\lambda \geq 0$ a.e.\,in $\Omega$.
			\vspace*{0.1cm}

			To this we first observe that, since $S_0^1(\Omega)$ is a Hilbert
			space and since $0<1-\gamma<1$, the functional $J_\lambda$ is weakly
			lower semicontinuous on $S_0^1(\Omega)$ (notice that $\||\nabla_\mathbb{G}u|\|_{L^2(\Omega)}$
			is precisely the \emph{norm} of $S_0^1(\Omega)$). Moreover, owing to the Sobolev inequality
			\eqref{eq:Sobolev_Ineq} and using H\"older's inequality, we have
			\begin{equation*}
				J_{\lambda}(u) \geq  \dfrac{\|u\|^{2}_{S^{1}_{0}(\Omega)}}{2} - 
				\dfrac{\lambda C}{1-\gamma}\|u\|^{1-\gamma}_{S^{1}_{0}(\Omega)}
			\end{equation*}
			(for some positive constant $C$ only depending on $Q$ and $\Omega$),
			and this proves that $J_\lambda$ is \emph{coercive} on $S_0^1(\Omega)$.
			Gathering these facts, we then derive that 
			$J_\lambda$
			attains a \emph{global minimum}, that is, there exists 
			a function $w_\lambda\in S_0^1(\Omega)$
			such that
			$$J_\lambda(w_\lambda) = \min_{u\in S_0^1(\Omega)}J_\lambda(u)=m_\lambda.$$
			We now observe that, if $\varphi\in C_0^\infty(\Omega)$ is a fixed non-vanishing 
			function, we have
			$$J_\lambda(t\varphi) = \frac{t^2}{2}\int_{\Omega}|\nabla_{\mathbb{G}}\varphi|^{2} 
			-\dfrac{\lambda t^{1-\gamma}}{1-\gamma} \int_{\Omega}|\varphi|^{1-\gamma}\to-\infty
			\quad\text{as $t\to0^+$}$$
			(recall that $0<1-\gamma<1$); as a consequence, we have
			 $m_\lambda < 0$, and thus $w_\lambda\not\equiv 0$.

			 Finally, taking into account that $J_\lambda(|u|)=J_\lambda(u)$ for every $u\in S_0^1(\Omega)$,
			 by possibly replacing the function $w_\lambda$ with $|w_\lambda|$, we can assume that
			 $w_\lambda \geq 0$ a.e.\,in $\Omega$.
			 \medskip
			 
			 \textsc{Step II).} In this second step we prove that the global, non-negative
			 minimizer $w_\lambda\in S_0^1(\Omega)$ 
			 obtained in the previous step is actually strictly positive (a.e.) in $\Omega$.
			 
			 To this end, we arbitrarily fix a \emph{non-negative} function $\varphi\in C_0^\infty(\Omega)$
			 and we notice that, since 
			 the function $w_\lambda$ is a global minimizer of $J_\lambda$, we have
			 \begin{align*}
			  0 & \leq \liminf_{t\to 0^+}\Big\{
			  \frac{1}{t}(J_\lambda(w_\lambda+t\varphi)-J_\lambda(w_\lambda))\Big\} 
			  \leq\int_\Omega\langle\nabla_\mathbb{G}w_\lambda,
			  \nabla_\mathbb{G}\varphi\rangle_{\mathfrak{g}_{1}},
			 \end{align*}
			 and thus $-\Delta_\mathbb{G}u\geq 0$ in the weak sense on $\Omega$;
			 from this, by using the Strong Maximum Principle in Proposition 
			 \ref{prop:SMP} (and since that $w_\lambda\not\equiv 0$), we deduce that 
			 $w_\lambda > 0$ a.e.\,in $\Omega$.
			 \medskip
			 
			 \textsc{Step III).} In this third step, we prove that the function  $w_\lambda\in S_0^1(\Omega)$
			 (which we have already proved to be a glo\-bal and strictly positive minimzer of
			 $J_\lambda)$ is a weak solution of 
			 \begin{equation} \label{eq:PDEpurelysingular}
			  -\Delta_\mathbb{G}w_\lambda = \frac{\lambda}{w_{\lambda}^\gamma}\quad\text{in $\Omega$},
			 \end{equation}
			 and hence of the whole problem \eqref{eq:Singular_Problem}.
			 
			 To this end we first observe that, since $w_\lambda$ is a global minimizer of $J_\lambda$, the point $t_0 = 0$ is a 
			 \emph{global minimum} for the (smooth) map $h:\mathbb{R}\to\mathbb{R}$ defined by
			 $$h(t) = J_\lambda\big((1+t)w_\lambda)\big)
			 = \frac{(1+t)^2}{2}\int_\Omega|\nabla_\mathbb{G}w_\lambda|^2-
			 \frac{\lambda(1+t)^{1-\gamma}}{1-\gamma}\int_\Omega|w_\lambda|^{1-\gamma};$$
			 as a consequence, we have
			 \begin{equation} \label{eq:operatorPartzero}
			 	0 = h'(0) = \int_\Omega|\nabla_\mathbb{G}w_\lambda|^2-
			 \lambda\int_\Omega|w_\lambda|^{1-\gamma}. 
			 \end{equation}
			 Moreover, again by the minimality of $J_\lambda(w_\lambda)$, we also get
			 \begin{align*}
			 	0&\leq J_\lambda(w_\lambda+t\varphi)
			 	-J_\lambda(w_\lambda) \\
			 	& \leq \int_\Omega\langle\nabla_\mathbb{G}w_\lambda,\nabla_\mathbb{G}\varphi\rangle_{\mathfrak{g}_{1}}
			 + \frac{t^2}{2}\int_\Omega |\nabla_\mathbb{G}\varphi|^2
			  - \frac{\lambda}{1-\gamma}\int_\Omega\big[
			 (w_\lambda+t\varphi)^{1-\gamma}-w_\lambda^{1-\gamma}\big]
			 \end{align*}
			 for every \emph{non-negative function $\varphi\in S_0^1(\Omega)$ and every $t > 0$};
			 thus, by letting $t\to 0^+$ with the aid of the Fatou Lemma
			 (and by the arbitrariness of $\varphi$),
			 we obtain 
			 \begin{equation} \label{eq:dausareperconcludereHaitao}
			 \int_\Omega\big(\langle\nabla_\mathbb{G}w_\lambda,\nabla_\mathbb{G}\varphi\rangle_{\mathfrak{g}_{1}}
			 - {\lambda} w_\lambda^{-\gamma}\varphi\big)\geq 0,\quad\text{for every
			 $\varphi\in S_0^1(\Omega),\,\varphi\geq 0$}.
			 \end{equation}
			 With estimates \eqref{eq:operatorPartzero}-\eqref{eq:dausareperconcludereHaitao} at hand,
			 we can easily complete the demonstration of this step. 
			 
			 Indeed, given any
			 $\psi\in S_0^1(\Omega)$, by applying
			 \eqref{eq:dausareperconcludereHaitao} to
			 $$\varphi = (w_\lambda+\varepsilon\psi)_+\in S_0^1(\Omega)$$
			 (and with $\varepsilon > 0$ arbitrarily fixed), and by using \eqref{eq:operatorPartzero}, we get
			 \begin{align*}
			  & 0 \leq \int_\Omega\big(\langle\nabla_\mathbb{G}w_\lambda,\nabla_\mathbb{G}\varphi\rangle_{\mathfrak{g}_{1}}
			 - {\lambda}w_\lambda^{-\gamma}\varphi\big)
			 \\
			 & \qquad = \int_{\{w_\lambda+\varepsilon\psi\geq 0\}}\big(\langle\nabla_\mathbb{G}w_\lambda,\nabla_\mathbb{G}(w_\lambda+\varepsilon\psi)\rangle_{\mathfrak{g}_{1}}
			 - {\lambda}w_\lambda^{-\gamma}(w_\lambda+\varepsilon\psi)\big) \\
			 & \qquad = \int_{\Omega}\big(\langle\nabla_\mathbb{G}w_\lambda,\nabla_\mathbb{G}(w_\lambda+\varepsilon\psi)\rangle_{\mathfrak{g}_{1}}
			 - {\lambda}w_\lambda^{-\gamma}(w_\lambda+\varepsilon\psi)\big)	 \\
			 & \qquad\qquad
			 - \int_{\{w_\lambda+\varepsilon\psi< 0\}}\big(\langle\nabla_\mathbb{G}w_\lambda,\nabla_\mathbb{G}(w_\lambda+\varepsilon\psi)\rangle_{\mathfrak{g}_{1}}
			 - {\lambda}w_\lambda^{-\gamma}(w_\lambda+\varepsilon\psi)\big)\\
			 & \qquad
			 = \Big(\int_\Omega|\nabla_\mathbb{G}w_\lambda|^2-
			 \lambda\int_\Omega|w_\lambda|^{1-\gamma}\Big)
			  +\varepsilon\int_\Omega\big(\langle\nabla_\mathbb{G}w_\lambda,\nabla_\mathbb{G}\psi\rangle_{\mathfrak{g}_{1}}
			 - {\lambda} w_\lambda^{-\gamma}\psi\big) \\
			 & \qquad\quad - \int_{\{w_\lambda+\varepsilon\psi< 0\}}\big(\langle\nabla_\mathbb{G}w_\lambda,\nabla_\mathbb{G}(w_\lambda+\varepsilon\psi)\rangle_{\mathfrak{g}_{1}}
			 - {\lambda}w_\lambda^{-\gamma}(w_\lambda+\varepsilon\psi)\big)
			 \\
			 & \qquad (\text{here we use \eqref{eq:operatorPartzero}}) \\
			 & \qquad = \varepsilon\int_\Omega\big(\langle\nabla_\mathbb{G}w_\lambda,\nabla_\mathbb{G}\psi\rangle_{\mathfrak{g}_{1}}
			 - {\lambda} w_\lambda^{-\gamma}\psi\big) \\
			 & \qquad\quad - \int_{\{w_\lambda+\varepsilon\psi< 0\}}\big(\langle\nabla_\mathbb{G}w_\lambda,\nabla_\mathbb{G}(w_\lambda+\varepsilon\psi)\rangle_{\mathfrak{g}_{1}}
			 - {\lambda}w_\lambda^{-\gamma}(w_\lambda+\varepsilon\psi)\big)
			 \\
			& \qquad \leq 
			\varepsilon\int_\Omega\big(\langle\nabla_\mathbb{G}w_\lambda,\nabla_\mathbb{G}\psi\rangle_{\mathfrak{g}_{1}}
			 - {\lambda} w_\lambda^{-\gamma}\psi\big) 
			 - \varepsilon\int_{\{w_\lambda+\varepsilon\psi< 0\}}\langle\nabla_\mathbb{G}w_\lambda,\nabla_\mathbb{G}\psi\rangle, 
			  \end{align*}
			 where we have used the fact that $w_\lambda > 0$ a.e.\,in $\Omega$. 
			 In view of this fact, we also recognize that the Lebesgue measure of the set $A_\varepsilon = \{w_\lambda+\varepsilon\psi< 0\}$
			 goes to $0$ as $\varepsilon\to 0^+$: indeed, since
			 $$\psi < -\frac{1}{\varepsilon}w_\lambda < 0\quad\text{on $A_\varepsilon$},$$
			 and since $\Omega\supseteq A_\varepsilon$ has finite measure, we have
			 $$\lim_{\varepsilon\to 0^+}|A_\varepsilon| = 
			 \Big|\bigcap_{\varepsilon > 0}A_\varepsilon\Big| = |\{w_\lambda\leq 0\}| = 0.$$
			 Hence, from the above estimate
			 (and by the arbitrariness of $\varepsilon > 0$), we obtain
			 \begin{align*}
			 	\int_\Omega\big(\langle\nabla_\mathbb{G}w_\lambda,\nabla_\mathbb{G}\psi\rangle_{\mathfrak{g}_{1}}
			 - {\lambda} w_\lambda^{-\gamma}\psi\big) 
			\geq \lim_{\varepsilon\to 0^+}\int_{\{w_\lambda+\varepsilon\psi< 0\}}\langle\nabla_\mathbb{G}w_\lambda,\nabla_\mathbb{G}\psi\rangle = 0.
			 \end{align*}
			 This, together with the arbitrariness of
			 $\psi \in S_0^1(\Omega)$, allows us to finally conclude that 
			 $w_\lambda$ is a weak solution of equation
			 \eqref{eq:PDEpurelysingular}, as desired.
			 			 \medskip
			 
			 \textsc{Step IV).} In this fourth step we show that $w_\lambda\in L^\infty(\Omega)$.
			 To this end we first observe that, since we \emph{already know} that
			 $w_\lambda$ is a weak solution
			 of \eqref{eq:Singular_Problem}, for every $k> 1$ we have
			 $$
			 \int_{\Omega}
			  \langle\nabla_\mathbb{G}w_\lambda,\nabla_\mathbb{G}[(w_\lambda-k)_+]\rangle_{\mathfrak{g}_{1}}
			  -\lambda\int_{\Omega}w_\lambda^{-\gamma}(w_\lambda-k)_+ = 0;$$
			  from this, by using the H\"older inequality and the weighted Young inequality, together
			  with the continuous embedding of $S_0^1(\Omega)$ into $L^{2_Q^\star}(\Omega)$, we get
	\begin{align*}
	 \int_{\Omega}
			  |\nabla_\mathbb{G}[(w_\lambda-k)_+]|^2 & \leq
			   \frac{\lambda}{k^\gamma}\int_{A_k}(w_\lambda-k)_+
		\leq \frac{\lambda}{k^\gamma}|A_k|^{1-1/2_Q^\star}
		\Big(\int_{\Omega}|(w_\lambda-k)_+|^{2_Q^\star}\Big)^{1/2_Q^\star} \\
		& \leq c|A_k|^{1-1/2_Q^\star}\Big(\int_{\Omega}|\nabla_\mathbb{G}[w_\lambda-k)_+]|^{2}\Big)^{1/2}
		\\
		& \leq
		c(\varepsilon) |A_k|^{2(1-1/2_Q^\star)}+\varepsilon\int_{\Omega}|\nabla_\mathbb{G}[w_\lambda-k)_+]|^{2},
	\end{align*}
	where $c(\varepsilon) > 0$ is a suitable constant depending on $\varepsilon$, and
	$$A_k = \{w_\lambda\geq k\}.$$
	From this, by choosing $\varepsilon > 0$ sufficiently small, we get
	\begin{equation} \label{eq:estimLinfDaUsare}
	 \int_{\Omega}
			  |\nabla_\mathbb{G}[(w_\lambda-k)_+]|^2
			  \leq c |A_k|^{2(1-1/2_Q^\star)}.
	\end{equation}
	With estimate \eqref{eq:estimLinfDaUsare} at hand, 
	we can easily complete the proof of this step (and of the  whole
	theorem). Indeed, owing to the cited \eqref{eq:estimLinfDaUsare}, for every
	$1< k < h$ we have
	\begin{align*}
	 (h-k)^{2}|A_h|^{2/2_Q^\star}
	 & \Big(\int_{A_h}(w_\lambda-k)_+^{2_Q^\star}\Big)^{2/2_Q^\star}
	 \leq \Big(\int_{A_k}(w_\lambda-k)_+^{2_Q^\star}\Big)^{2/2_Q^\star} \\
	 & (\text{here we use once again \eqref{eq:Sobolev_Ineq}}) \\
	 & \leq c\int_{\Omega}
			  |\nabla_\mathbb{G}[(w_\lambda-k)_+]|^2\leq 
			  c |A_k|^{2(1-1/2_Q^\star)},
	\end{align*}
	for some constant $c > 0$ independent of $k,h$; as a consequence, we get
	$$|A_h|\leq \frac{c}{(h-k)^{2_Q^\star}}|A_k|^{2_Q^\star-1}\quad\text{for all $1< k < h$}.$$
	Since, obviously, $q = 2_Q^\star-1 > 1$, we can apply \cite[Lemma B.1]{KS}: this gives
	$$|A_{d}| = |\{w_\lambda\geq d\}| = 0$$
	for some $d > 0$, and thus $w_\lambda\in L^\infty(\Omega)$, as desired.
	\medskip
	
	\textsc{Step V)}. In this last step, we prove that $w_\lambda$ is
	the \emph{unique} weak solution of problem \eqref{eq:Singular_Problem}. To this end, let us suppose that $z_\lambda\in S_0^1(\Omega)$ is another
	weak solution of the same problem; we then choose a smooth function $\theta\in C^\infty(\mathbb{R})$ satisfying the following properties
		\begin{itemize}
		 \item  $\theta(t) = 0$ for $t \leq 0$ and $\theta(t) =1$ for $t \geq 1$;
		 \item $\theta$ is non-decreasing on $\mathbb{R}$;
		 \end{itemize}
		and we define (for every $\varepsilon > 0$)
$$\theta_\varepsilon(t): = \theta\left(\frac{t}{\varepsilon}\right).$$ 
Now, since $\varphi_\varepsilon\in S_0^1(\Omega)$, and since
$w_\lambda,z_\lambda$ solve \eqref{eq:Singular_Problem}, we get	
\begin{align*}
	(\ast)\,\,&\int_{\Omega}\langle\nabla_{\mathbb{G}} w_{\lambda}, \nabla_{\mathbb{G}}(w_{\lambda}-z_\lambda)\rangle_{\mathfrak{g}_{1}} \, \theta'_{\varepsilon}(w_{\lambda}-z_\lambda)
			- \lambda \int_{\Omega}\dfrac{\theta_{\varepsilon}(w_{\lambda}-z_\lambda)}{w_{\lambda}^{\gamma}} = 0; \\[0.1cm]
	(\ast)\,\,& \int_{\Omega}\langle\nabla_{\mathbb{G}} z_\lambda, \nabla_{\mathbb{G}}(w_{\lambda}-z_\lambda)\rangle_{\mathfrak{g}_{1}}\, \theta'_{\varepsilon}(w_{\lambda}-z_\lambda) 
	- \lambda \int_{\Omega}\dfrac{\theta_{\varepsilon}(w_{\lambda}-z_\lambda)}{z_\lambda^{\gamma}} = 0.	
		\end{align*}
	As a consequence, by subtracting the above identities, we obtain
	\begin{equation}\label{eq:SubtractionUniqueness}
		\begin{aligned}
			0 &\geq  -\int_{\Omega}|\nabla(z_{\lambda}-w_\lambda)|_{\mathbb{G}}^2 \, \theta'_{\varepsilon}(w_{\lambda}-z_\lambda)
			 \\
			 & =  \lambda\int_\Omega\Big(\frac{1}{z_\lambda^\gamma}-\frac{1}{w_\lambda^\gamma}\Big)\theta_\varepsilon(w_\lambda-z_\lambda).
		\end{aligned}
	\end{equation}
	From this, by letting $\varepsilon \to 0^+$ we derive that	
	$$\lim_{\varepsilon\to 0^+}\int_{\Omega}\left( \dfrac{1}{z_\lambda^{\gamma}} - \dfrac{1}{w_{\lambda}^{\gamma}}\right) \theta_{\varepsilon}(w_{\lambda}-z_\lambda)
	= \int_{\{w_{\lambda}> z_\lambda\}}\left( \dfrac{1}{z_\lambda^{\gamma}}- \dfrac{1}{w_{\lambda}^{\gamma}}\right)\leq 0,$$
	 from which we derive that
	$|\{ w_{\lambda}> z_\lambda\}| = 0$, that is,
	$w_\lambda\leq z_\lambda$ a.e.\,on $\Omega$.	
	Finally, by interchan\-ging the role of $z_\lambda$ and $w_\lambda$,
	we conclude that
	$$\text{$w_\lambda\equiv z_\lambda$ a.e.\,in $\Omega$},$$
	and this proves the uniqueness of $w_\lambda$.
			\end{proof}

\section{Existence of the first solution}\label{sec:First_Solution}
To begin with, we define
\begin{equation}\label{eq:DefinitionLambda}
	\Lambda := \sup \{ \lambda >0: \eqref{eq:Main_Problem}_\lambda \textrm{ admits a weak solution}\}.
\end{equation}
We then turn to prove in this first part of the section the following facts:
\vspace*{0.1cm}

a)\,\,$\Lambda$ is well-defined and $\Lambda < +\infty$;
\vspace*{0.05cm}

b)\,\,problem \eqref{eq:Main_Problem}$_\lambda$ admits a weak solution for every $0<\lambda\leq \Lambda$.
\vspace*{0.1cm}

\noindent We begin by proving assertion a).
\begin{lemma} \label{lem:Lambdafinito}
	Let $\Lambda$ be as in \eqref{eq:DefinitionLambda}. Then $\Lambda \in (0,+\infty)$.
	\begin{proof}
		We consider the functional 
		\begin{equation}\label{eq:FunctionalIlambda}
			I_{\lambda}(u) := \dfrac{1}{2}\int_{\Omega}|\nabla_{\mathbb{G}}u|^{2} - \dfrac{\lambda}{1-\gamma}\int_{\Omega}|u|^{1-\gamma} - \dfrac{1}{2^{\star}_{Q}}\int_{\Omega}|u|^{2^{\star}_{Q}}, \quad u \in S^{1}_{0}(\Omega).
		\end{equation}
		First of all, by combining H\"older's and Sobolev's inequalities, we have
		\begin{equation} \label{eq:todeduce21Haitao}
			\begin{split}
				& \mathrm{a)}\,\,\int_\Omega |u|^{2^{\star}_{Q}}  \leq C\||\nabla_{\mathbb{G}} u|\|_{L^{2}(\Omega)}^{2^{\star}_{Q}}; \\
				& \mathrm{b)}\,\,\int_{\Omega}|u|^{1-\gamma}  \leq C
				\|u\|_{L^{2^{\star}_{Q}}(\Omega)}^{(1-\gamma)/2^{\star}_{Q}} \leq C\||\nabla_{\mathbb{G}} u|\|^{1-\gamma}_{L^2(\Omega)};
			\end{split}
		\end{equation}
		\noindent as a consequence, denoting by 
		$$B_{r}:= \left\{ u \in S^{1}_{0}(\Omega): \|u\|_{S^{1}_{0}(\Omega)}\leq r\right\},$$
		the above estimates (\ref{eq:todeduce21Haitao}-a) imply the existence of $r_0>0$ and $\delta_0>0$ such that 
		\begin{equation}\label{eq:2.1Haitao}
			\left\{ \begin{array}{rl}
				\tfrac{1}{2}\||\nabla_{\mathbb{G}} u|\|_{L^{2}(\Omega)}^2 - \tfrac{1}{2^{\star}_{Q}}\|u\|^{2^{\star}_{Q}}_{L^{2^{\star}_{Q}}(\Omega)}\geq 2 \delta_0 & \textrm{for all } u \in \partial B_{r_0},\\[0.2cm]
				\tfrac{1}{2}\||\nabla_{\mathbb{G}} u|\|_{L^{2}(\Omega)}^2 - \tfrac{1}{2^{\star}_{Q}}\|u\|^{2^{\star}_{Q}}_{L^{2^{\star}_{Q}}(\Omega)}\geq 0 & \textrm{for all } \in B_{r_0},
			\end{array}\right.
		\end{equation}
		\noindent hence, again by (\ref{eq:todeduce21Haitao}-b) we conclude that there exists $\lambda_{\star}>0$ 
		such that
		\begin{equation}\label{eq:2.2Haitao}
			\begin{split}
				I_{\lambda}\big|_{\partial B_{r_0}} & \geq 2\delta_0 - \frac{\lambda\,C}{1-\gamma}r_0^{1-\gamma}
				\geq \delta_0, \quad \textrm{for every } \lambda \in (0,\lambda_{\star}].
			\end{split}
		\end{equation}
		We now define $c_{\star}:= \inf_{B_{r_0}}I_{\lambda_{\star}}$, and we notice that $c_{\star}<0$. Indeed, for every $v\not\equiv 0$, it holds that
		\begin{equation*}
			I_{\lambda_{\star}}(tv) = t^2 \||\nabla_{\mathbb{G}}v|\|_{L^{2}(\Omega)}^2 - \dfrac{\lambda_{\star}}{1-\gamma}t^{1-\gamma}\int_{\Omega}|v|^{1-\gamma}  - \dfrac{t^{2^{\star}_{Q}}}{2^{\star}_{Q}}\int_{\Omega}|v|^{2^{\star}_{Q}},
		\end{equation*}
		\noindent which becomes negative for $t>0$ small enough because $0<1-\gamma<1$. The argument is now pretty standard. We first consider a minimizing sequence $\{u_j\}_j
		\subset B_{r_0}$ related to $c_{\star}$ and we know that there exists $u_{\star}$ such that, up to subsequences,
		\begin{itemize}
			\item $u_j \to u_{\star}$ as $j \to +\infty$ weakly in $S^{1}_{0}(\Omega)$;
			\item $u_j \to u_{\star}$ as $j \to +\infty$ strongly in $L^{r}(\Omega)$ for every $r \in [2,2^{\star}_{Q})$;
			\item $u_j \to u_{\star}$ as $j \to +\infty$ pointwise a.e. in $\Omega$.
		\end{itemize}
		Moreover, since $I_{\lambda}(|u|) = I_{\lambda}(u)$ for every $\lambda>0$, we may also assume that $u_j \geq 0$.
		Combining \eqref{eq:2.2Haitao} with $c_{\star}<0$, we realize that there exists a positive, and independent of $j$, constant $\varepsilon_0>0$  such that
		\begin{equation} \label{eq:boundrhoujtodeduce}
			\|u_j\|_{S^{1}_{0}(\Omega)} \leq r_0 - \varepsilon_0.
		\end{equation}
		Now, by combining the
		algebraic inequality $(a+b)^p\leq a^b+b^p$ (holding true for
		all $a,b\geq 0$ and $0<p<1$) with the
		H\"{o}lder inequality, as $j\to +\infty$ we have
		\begin{equation*}
			\begin{split}
				\int_{\Omega}u_j^{1-\gamma} & \leq \int_{\Omega}u_{\star}^{1-\gamma}+
				\int_{\Omega}|u_j - u_{\star}|^{1-\gamma}  
				\\
				& \leq \int_{\Omega}u_{\star}^{1-\gamma}+ 
				C \|u_j-u_{\star}\|^{1-\gamma}_{L^{2}(\Omega)} = \int_{\Omega}u_{\star}^{1-\gamma} + o(1),
			\end{split}
		\end{equation*}
		\noindent and similarly
		\begin{equation*}
			\int_{\Omega}u_{\star}^{1-\gamma} \leq \int_{\Omega}u_{j}^{1-\gamma}+\int_{\Omega}|u_j - u_{\star}|^{1-\gamma} = \int_{\Omega}u_{j}^{1-\gamma} + o(1),
		\end{equation*}
		\noindent which in turn implies
		\begin{equation}\label{eq:2.3Haitao}
			\int_{\Omega}u_{j}^{1-\gamma}  = \int_{\Omega}u_{\star}^{1-\gamma}  + o(1), \quad \textrm{as } j \to +\infty.
		\end{equation}
		By Brezis-Lieb lemma, see \cite{BrezisLieb}, it is now well known that
		\begin{equation}
			\|u_j\|^{2^{\star}_{Q}}_{L^{2^{\star}_{Q}}(\Omega)} = \|u_{\star}\|^{2^{\star}_{Q}}_{L^{2^{\star}_{Q}}(\Omega)} + \|u_j - u_{\star}\|^{2^{\star}_{Q}}_{L^{2^{\star}_{Q}}(\Omega)} +o(1), \quad \textrm{as } j \to +\infty,
		\end{equation}
		\noindent and 
		\begin{equation}\label{eq:2.5Haitao}
		\||\nabla_{\mathbb{G}}u_j|\|_{\mathbb{G}}^2 = \||\nabla_{\mathbb{G}}u_{\star}|\|_{\mathbb{G}}^2 + \||\nabla_{\mathbb{G}}(u_j -u_{\star})|\|_{\mathbb{G}}^2 + o(1), \quad \textrm{as } j \to +\infty.
		\end{equation}
		By combining \eqref{eq:2.5Haitao} and \eqref{eq:boundrhoujtodeduce}, 
		it follows that $u_{\star}\in B_{r_0}$ and that
		$u_j - u_{\star} \in B_{r_0}$ for big enough $j$, and this allows to use the second line of \eqref{eq:2.1Haitao} on $u_j - u_{\star}$. 
		Using now \eqref{eq:2.3Haitao}-\eqref{eq:2.5Haitao}, as $j\to +\infty$, we find
		\begin{equation*}
			\begin{aligned}
				c_{\star} &= I_{\lambda_{\star}}(u_j) + o(1)\\
				&= I_{\lambda_{\star}}(u_{\star}) + \dfrac{1}{2}\||\nabla_{\mathbb{G}}(u_j -u_{\star})|\|_{\mathbb{G}}^2- \dfrac{1}{2^{\star}_{Q}}\|u_j - u_{\star}\|^{2^{\star}_{Q}}_{L^{2^{\star}_{Q}}(\Omega)}+ o(1)\\
				&\geq I_{\lambda_{\star}}(u_{\star}) + o(1) \geq c_{\star} + o(1),
			\end{aligned}
		\end{equation*}
		\noindent which proves that $u_{\star} \geq 0$, $u_{\star}\not\equiv
		0$ is a local minimizer of $I_{\lambda_{\star}}$ in the $S^{1}_{0}(\Omega)$-topology.
		From this, by arguing exactly as in the proof
		of Theorem \ref{thm:Singular_Problem} (the unique difference being the presence of the critical term, see also \cite[Lemma 2.1]{Haitao}), we show that $u_{\star}$ is actually a weak solution of  \eqref{eq:Main_Problem}$_\lambda$, 
		and hence we get that 
		$$\Lambda \geq \lambda_{\star}>0.$$
		\indent Let us now prove that $\Lambda < +\infty$. Following \cite{Haitao}, we consider 
		the \emph{first 
			Dirichlet eigenvalue $\mu_1$ of the operator $-\Delta_{\mathbb{G}}$} in $\Omega$, namely 
		$$\mu_1 = \min\big\{\||\nabla_{\mathbb{G}}|\|_{L^{2}(\Omega)}^2:\,\text{$u\in S^{1}_{0}(\Omega)$ and $\|u\|^2_{L^2(\Omega)} = 1$}\big\},$$
		and we let $e_1\in S^{1}_{0}(\Omega)$ 
		be the principal eigenfunction associated with $\mu_1$, i.e.,
		\vspace*{0.1cm}
		
		a)\,\,$\|e_1\|_{L^2(\Omega)} = 1$ and $e_1 > 0$ a.e.\,in $\Omega$; 
		\vspace*{0.05cm}
		
		b)\,\,$\||\nabla_{\mathbb{G}}e_1|\|_{\mathbb{G}}^2 = \mu_1$.
		\vspace*{0.1cm}
		
		\noindent The proof of the existence of such a function $e_1$ follows by rather standard arguments.
		
		Now, assuming that there exists
		a weak solution $u\in S^{1}_{0}(\Omega)$ of problem
		\eqref{eq:Main_Problem}$_\lambda$ (for some $\lambda > 0$),
		and using this $e_1$ as a test function in identity \eqref{eq:weak_sub_super_sol}, we get
		\begin{equation*}
			\mu_1 \int_{\Omega}ue_1  = \int_{\Omega}\langle \nabla_{\mathbb{G}}u, \nabla_{\mathbb{G}}e_1\rangle_{\mathfrak{g}_{1}} = \int_{\Omega}(\lambda u^{-\gamma} + u^{2^{\star}_{Q}-1})e_1.
		\end{equation*}
		Setting $\overline{\Lambda}$ a constant such that
		$$\overline{\Lambda}  t^{-\gamma} + t^{2^{\star}_{Q}-1} >  \mu_1 t, \quad \textrm{for every } t>0,$$
		\noindent we find that $\lambda < \overline{\Lambda}$ and then $\Lambda < \overline{\Lambda}<+\infty$, as desired.
	\end{proof}
\end{lemma}

Now we have established Lemma \ref{lem:Lambdafinito}, we then turn to 
prove assertion b), namely
the existence of at least one weak solution of problem \eqref{eq:Main_Problem}$_\lambda$ for
every $0<\lambda\leq \Lambda$. 
%To this end, following \cite{Haitao}
%we first establish some  preliminary results.
\medskip

To begin with, we prove the following simple yet important technical lemma.
\begin{lemma} \label{lem:SMPsoprasotto}
	Let $w,u\in S^{1}_{0}(\Omega)$ be a \emph{weak subsolution} 
	\emph{[}resp.\,\emph{weak supersolution}\emph{]}
	and a \emph{weak solution} of pro\-blem \eqref{eq:Main_Problem}$_\lambda$, respectively.
	We assume that
	\begin{itemize}
		\item[a)] $w\leq u$ \emph{[}resp.\,$w\geq u$\emph{]} a.e.\,in $\Omega$;
		\item[b)] for every open set $\Oo\Subset\Omega$ there exists $C = C(\Oo,w) > 0$ such that
		$$\text{$w \geq C$ a.e.\,in $\Oo$}.$$
	\end{itemize}
	Then, either $w\equiv u$ or $w<u$ \emph{[}resp.\,$w > u$\emph{]} a.e.\,in $\Omega$.
\end{lemma}
\begin{proof}
	We limit ourselves to consider only the case when $w$ is a \emph{weak subsolution} of 
	pro\-blem \eqref{eq:Main_Problem}$_\lambda$, since the case when $w$ is a
	weak supersolution is analogous.
	
	To begin with, we arbitrarily fix a bounded open set $\Oo\Subset\Omega$ and we observe that, since $w$ 
	is a weak subsolution of problem \eqref{eq:Main_Problem}$_\lambda$ and since $u$
	is a weak solution of the same problem, we have the following computations:
	\begin{align*}
		-\Delta_{\mathbb{G}} (u-w) & \geq \lambda(u^{-\gamma}-w^{-\gamma})
		+(u^{2^{\star}_{Q}-1}-w^{2^{\star}_{Q}-1}) \\
		& (\text{since $w\leq u$, see assumption a)}) \\
		& \geq \lambda({u}^{-\gamma}-w^{-\gamma}) \\
		& (\text{by the Mean Value Theorem, for some $\theta\in(0,1)$}) \\
		& = -\gamma\lambda(\theta {u}+(1-\theta)w)^{-\gamma-1}(u-w) \\
		& \geq -\gamma\lambda w^{-\gamma-1}(u-w) \\
		& (\text{by assumption b)}) \\
		& \geq -\gamma\lambda C^{-\gamma-1}(u-w),
	\end{align*}
	in the weak sense on $\Oo$
	(here, $C > 0$ is a constant depending on $\Oo$ and on $w$). 
	
	As a consequence of this fact,
	and since $w\leq u$ a.e.\,in $\Omega$, we are then entitled
	to apply the Strong Maximum Principle in Proposition \ref{prop:SMP} to the
	function $v = u-w$ (with $c \equiv \gamma\lambda C^{-\gamma-1}t$),
	obtaining that
	$$\text{either $v\equiv 0$ or $v > 0$ a.e.\,in $\Oo$}.$$
	Due to the arbitrariness of $\Oo\Subset\Omega$, this completes the proof.
\end{proof}
\begin{rmk} \label{rem:assumptionbnonserve}
	We explicitly observe that, if $w\in S^{1}_{0}(\Omega)$ is a weak \emph{supersolution}
	of problem \eqref{eq:Main_Problem}$_\lambda$, it follows from Remark \ref{rem:defweaksolPb}-3) that
	assumption b) in Lemma \ref{lem:SMPsoprasotto} is \emph{always satisfied}. 
	Hence, if $u\in S^{1}_{0}(\Omega)$ is a weak \emph{solution} of \eqref{eq:Main_Problem}$_\lambda$, we get
	$$(\text{$u\leq w$ a.e.\,in $\Omega$})\,\,\Longrightarrow\,\,(
	\text{either $u\equiv w$ or $u < w$ a.e.\,in $\Omega$}).$$
\end{rmk}

We now turn to establish a crucial Perron-type lemma which extends \cite[Lemma 2.2]{Haitao} to the case of Carnot groups.
\begin{lemma}\label{lem:2.2Haitao}
	Let $\underline{u}, \overline{u}\in S^{1}_{0}(\Omega)$ be a weak subsolution and a weak supersolution, respectively, of problem \eqref{eq:Main_Problem}$_\lambda$. 
	We assume that 
	\begin{itemize}
		\item[a)] $\underline{u}(g) \leq \overline{u}(g)$ for a.e.\,$g\in \Omega$;
		\item[b)] for every open set $\Oo\Subset\Omega$ there exists
		$C = C(\Oo,\underline{u}) > 0$ such that
		$$\text{$\underline{u}\geq C$ a.e.\,in $\Oo$}.$$
	\end{itemize}
	Then, there exists a weak solution $u \in S^{1}_{0}(\Omega)$ of \eqref{eq:Main_Problem} such that 
	$$\text{$\underline{u}(g) \leq u(g) \leq \overline{u}(g)$ for a.e. $g \in \Omega$}.$$
	\begin{proof}
		We adapt to our setting the proof of \cite[Lemma 2.2]{Haitao}.
		We consider the set 
		\begin{equation*}
			M := \left\{ u \in S^{1}_{0}(\Omega): \underline{u} \leq u \leq \overline{u} \textrm{ a.e. in } \Omega \right\},
		\end{equation*}
		\noindent which is closed and convex.\\
		\indent \textsc{Step 1:} we claim that there exists a relative minimizer $u_{\lambda}$ of $I_{\lambda}$ on $M$.\\
		\noindent It is enough to show that $I_{\lambda}$ is w.l.s.c.\,on $M$. To this aim, let $u_j \in M$ be weakly convergent to $u$ in $S^{1}_{0}(\Omega)$. Without loss of generality, possibly passing to a subsequence, we may assume that $u_j \to u$ pointwise a.e. in $\Omega$, so that $u\in M$.
		Thanks to the continuous embedding of $S^{1}_{0}(\Omega)$, we have that
		\begin{equation*}
			\int_{\Omega}\overline{u}^{2^{\star}_{Q}}  < +\infty \quad \int_{\Omega}\overline{u}^{1-\gamma}  < +\infty,
		\end{equation*}
		\noindent where in the latter we used first H\"{o}lder inequality.
		Hence, by dominated convergence, we have that, as $j\to +\infty$,
		\begin{equation*}
			\|u_j\|^{2^{\star}_{Q}}_{L^{2^{\star}_{Q}}(\Omega)} \to \|u\|^{2^{\star}_{Q}}_{L^{2^{\star}_{Q}}(\Omega)} \quad \textrm{and} \quad \int_{\Omega}|u_j|^{1-\gamma}  \to \int_{\Omega}|u|^{1-\gamma}.
		\end{equation*}
		Therefore,
		$$\liminf_{j \to +\infty}I_{\lambda}(u_j) \geq I_{\lambda}(u),$$
		\noindent as desired.\\
		\indent {\sc Step 2:} we prove that $u_\lambda$ is a weak solution of \eqref{eq:Main_Problem}$_\lambda$. Just in this paragraph
		we will use more compact notation $u$ instead  of $u_\lambda$.\\
		\noindent We take $\varphi \in S^{1}_{0}(\Omega)$ and $\varepsilon>0$. We define the function 
		$$v_{\varepsilon}:= u + \varepsilon \varphi - \varphi^{\varepsilon} + \varphi{_\varepsilon} \in M,$$
		\noindent where 
		$$\varphi^{\varepsilon} := (u + \varepsilon \varphi - \overline{u})_{+} \qquad \varphi_{\varepsilon} := (u + \varepsilon \varphi - \underline{u})_{-}.$$
		Since $u + t(v_{\varepsilon}-u)\in M$ for $t\in (0,1)$, we have that
		\begin{equation}\label{eq:2.9Haitao}
			\begin{aligned}
				0& \leq \lim_{t \to 0^+}\dfrac{I_{\lambda}(u+t(v_{\varepsilon}-u)) - I_{\lambda}(u)}{t} \\
				&= \int_{\Omega}\langle \nabla_{\mathbb{G}}u, \nabla_{\mathbb{G}}(v_{\varepsilon}-u)\rangle_{\mathfrak{g}_{1}} - \lambda \int_{\Omega}\dfrac{v_{\varepsilon}-u}{u^{\gamma}} - \int_{\Omega}u^{2^{\star}_{Q}-1}(v_{\varepsilon}-u).
			\end{aligned}
		\end{equation}
		We omit the details concerning the second integral, we refer to \cite{Haitao} for the details.\\
		Set now
		\begin{equation*}
			E^{\varepsilon}:= \int_{\Omega} \langle \nabla_{\mathbb{G}}u, \nabla_{\mathbb{G}}\varphi^{\varepsilon}\rangle_{\mathfrak{g}_{1}} - \lambda \int_{\Omega}\dfrac{\varphi^{\varepsilon}}{u^{\gamma}}- \int_{\Omega}u^{2^{\star}_{Q}-1}\varphi^{\varepsilon},
		\end{equation*}
		\noindent and 
		\begin{equation*}
			E_{\varepsilon}:= \int_{\Omega} \langle \nabla_{\mathbb{G}}u, \nabla_{\mathbb{G}}\varphi_{\varepsilon}\rangle_{\mathfrak{g}_{1}}  - \lambda \int_{\Omega}\dfrac{\varphi_{\varepsilon}}{u^{\gamma}} - \int_{\Omega}u^{2^{\star}_{Q}-1}\varphi_{\varepsilon}.
		\end{equation*}
		With this notation at hand, we can write \eqref{eq:2.9Haitao} as 
		\begin{equation*}
			\int_{\Omega} \langle \nabla_{\mathbb{G}}u, \nabla_{\mathbb{G}}\varphi \rangle_{\mathfrak{g}_{1}}  - \lambda \int_{\Omega}\dfrac{\varphi}{u^{\gamma}} - \int_{\Omega}u^{2^{\star}_{Q}-1}\varphi \geq \dfrac{E^{\varepsilon} - E_{\varepsilon}}{\varepsilon}.
		\end{equation*}
		We want to show that
		$$\dfrac{E^{\varepsilon}}{\varepsilon} \geq o(1) \quad \textrm{and }  \quad \dfrac{E_{\varepsilon}}{\varepsilon} \leq o(1), \quad \textrm{as } \varepsilon \to 0^+.$$
		We will show only the first one, being the second very similar.
		Firstly, we define the sets
		\begin{align*}
			\Omega^{\varepsilon}&:= \left\{ g \in \Omega: u(g) +\varepsilon\varphi(g) \geq \overline{u}(g) > u(g)\right\},\\
			\mathcal{C}\Omega^{\varepsilon}&:= \{g \in \Omega: u(g)+\varepsilon\varphi(x)<\overline{u}(g)\}.
		\end{align*}
		\noindent
		 and notice that, being $u$ and $\overline{u}$ measurable functions,
		$$|\Omega^{\varepsilon}| \to 0 \quad  \textrm{as } \varepsilon \to 0^+.$$
				%%%%
		%%%%parte aggiunta mattia
		%%%%%
		Indeed, $\bigcap_{\varepsilon>0}\Omega^\varepsilon=\varnothing$, and 
		as $|\Omega|<+\infty$, this implies
		\[
		\lim_{\varepsilon\to0^+}|\Omega^\varepsilon|=\left|\bigcap_{\varepsilon>0}\Omega^\varepsilon\right|=0.
		\]
		%%%%%%
		%%%%%fine parte mattia
		%%%%%%

%		We also stress that
%		$$\mathbb{G}= (\mathbb{G}\setminus \Omega) \cup \mathcal{C}\Omega^{\varepsilon}\cup \Omega^{\varepsilon}.$$
		\noindent Therefore, following \cite{Haitao}, we have that
		\begin{equation}\label{eq:Pre2.12Haitao}
			\begin{aligned}
				\dfrac{E^{\varepsilon}}{\varepsilon} &= \dfrac{1}{\varepsilon}\left( 
				\int_{\Omega}\langle \nabla_{\mathbb{G}}(u-\overline{u}),\nabla_{\mathbb{G}}\varphi^{\varepsilon}\rangle_{\mathfrak{g}_{1}} +  
				\int_{\Omega}\langle\nabla_{\mathbb{G}}\overline{u},\nabla_{\mathbb{G}}\varphi^{\varepsilon}\rangle_{\mathfrak{g}_{1}} -  \int_{\Omega}(\lambda u^{-\gamma}+u^{2^{\star}_{Q}-1})\varphi^{\varepsilon} \right)\\
				&\geq \dfrac{1}{\varepsilon} \int_{\Omega} \langle\nabla_{\mathbb{G}}(u-\overline{u}),\nabla_{\mathbb{G}}\varphi^{\varepsilon}\rangle_{\mathfrak{g}_{1}} + \dfrac{1}{\varepsilon}\int_{\Omega^{\varepsilon}}(\lambda \overline{u}^{-\gamma} + \overline{u}^{2^{\star}_{Q}-1} - \lambda u^{-\gamma} - u^{2^{\star}_{Q}-1} )\varphi^{\varepsilon} \\
				&\geq \dfrac{1}{\varepsilon} \int_{\Omega} \langle\nabla_{\mathbb{G}}(u-\overline{u}),\nabla_{\mathbb{G}}\varphi^{\varepsilon}\rangle_{\mathfrak{g}_{1}} - \dfrac{\lambda}{\varepsilon} \int_{\Omega^{\varepsilon}}|\overline{u}^{-\gamma}-u^{-\gamma}||\varphi|.
			\end{aligned}
		\end{equation}
		Now, owing to \eqref{eq:Grad_and_level_set} and since $\varphi^{\varepsilon}=0$ on $\mathcal{C}\Omega^{\varepsilon}$, we have that
		\begin{equation}\label{eq:LocalHaitao}
			\begin{aligned}
				\dfrac{1}{\varepsilon} &\int_{\Omega} \langle\nabla_{\mathbb{G}}(u-\overline{u}),\nabla_{\mathbb{G}}\varphi^{\varepsilon}\rangle_{\mathfrak{g}_{1}} = \dfrac{1}{\varepsilon} \int_{\Omega^{\varepsilon}} \langle\nabla_{\mathbb{G}}(u-\overline{u}),\nabla_{\mathbb{G}}\varphi^{\varepsilon}\rangle_{\mathfrak{g}_{1}}\\
				&=\dfrac{1}{\varepsilon}\int_{\Omega^{\varepsilon}}|\nabla(u-\overline{u})|_{\mathbb{G}}^2 + \int_{\Omega^{\varepsilon}}\langle \nabla_{\mathbb{G}}(u-\overline{u}),\nabla_{\mathbb{G}}\varphi \rangle_{\mathfrak{g}_{1}} \\
				&\geq \int_{\Omega^{\varepsilon}}\langle \nabla_{\mathbb{G}}(u-\overline{u}),\nabla_{\mathbb{G}}\varphi \rangle_{\mathfrak{g}_{1}}  = o(1), \qquad \textrm{as } \varepsilon \to 0^{+}.
			\end{aligned}
		\end{equation}
		Combining \eqref{eq:Pre2.12Haitao} and  \eqref{eq:LocalHaitao}, we finally get
		\begin{equation}\label{eq:2.12Haitao}
			\dfrac{E^{\varepsilon}}{\varepsilon} \geq o(1), \quad \textrm{as } \varepsilon \to 0^+.
		\end{equation}
		Similarly
		\begin{equation}\label{eq:2.13Haitao}
			\dfrac{E_{\varepsilon}}{\varepsilon} \leq o(1), \quad \textrm{as } \varepsilon \to 0^+.
		\end{equation}
		Combining \eqref{eq:2.12Haitao} and \eqref{eq:2.13Haitao}, we get that
		$$\int_{\Omega}\langle \nabla_{\mathbb{G}}u, \nabla_{\mathbb{G}}\varphi\rangle_{\mathfrak{g}_{1}} - \int_{\Omega}(\lambda u^{-\gamma}- u^{2^{\star}_{Q}-1})\varphi \geq o(1) \quad \textrm{as } \varepsilon\to 0^+.$$
		Taking now $-\varphi$, and passing to the limit as $\varepsilon\to 0^+$, one closes the proof.
	\end{proof}
\end{lemma}

We are now ready to prove the existence of a weak solution of \eqref{eq:Main_Problem}$_\lambda$. 
It is an adaptation to our setting of \cite[Lemma 2.3]{Haitao}.
\begin{lemma}\label{lem:2.3Haitao}
	Problem \eqref{eq:Main_Problem}$_\lambda$ admits \emph{(}at least\emph{)}
	one weak solution $u_\lambda\in S^{1}_{0}(\Omega)$ for every $\lambda \in (0,\Lambda]$.
\end{lemma}
\begin{proof}
	The idea of the proof is rather standard: we want to construct both a weak subsolution and a weak supersolution and then apply Lemma \ref{lem:2.2Haitao}. 
	
	Let us start with the weak subsolution. 
	By Theorem \ref{thm:Singular_Problem}, we know that for every $\lambda \in (0,\Lambda)$ (and actually for all $\lambda >0$) there exists a unique solution $w_{\lambda}$ of \eqref{eq:Singular_Problem}, which is the Euler-Lagrange equation naturally associated with the functional $J_{\lambda}$ defined in \eqref{eq:def_Functional_Singular}.	The function $w_{\lambda}$ is a weak subsolution of \eqref{eq:Main_Problem}$_\lambda$. 
	
	Let us now look for a weak supersolution. By the very definition of $\Lambda$, we know that there necessarily exists $\lambda' \in (\lambda, \Lambda)$ such that \eqref{eq:Main_Problem}$_{\lambda'}$ admits a weak solution $u_{\lambda'}$, and this can be easily taken as a weak supersolution of \eqref{eq:Main_Problem}$_{\lambda}$.\\
	We now claim that
	\begin{equation}\label{eq:Claim}
		w_{\lambda}(g) \leq u_{\lambda'}(g), \quad \textrm{for a.e. } g\in \Omega.
	\end{equation}
	To this aim, we proceed essentially as in the proof
	of Theorem \ref{thm:Singular_Problem}, \textsc{Step V).} First of all, let us consider a smooth non-decreasing function $\theta: \mathbb{R}\to \mathbb{R}$ such that
	$$\theta(t) =1 \, \textrm{ for } t \geq 1 \quad \textrm{ and } \quad \theta(t) = 0 \, \textrm{ for } t \leq 0,$$
	\noindent and it is linked in a smooth way for $t \in (0,1)$. We further define the function
	$$\theta_{\varepsilon}(t):= \theta \left(\dfrac{t}{\varepsilon}\right), \quad \varepsilon>0, \, t\in \mathbb{R}.$$
	Due to its definition, we are entitled to use the function $\theta_{\varepsilon}(w_{\lambda}- u_{\lambda'})$ as a test function in both \eqref{eq:Main_Problem}$_{\lambda'}$ (solved by $u_{\lambda'}$) and \eqref{eq:Singular_Problem} (solved by $w_{\lambda}$). Thus, we have
	\begin{equation}\label{eq:solvedBywlambda}
			\int_{\Omega}\langle\nabla_{\mathbb{G}} w_{\lambda}, \nabla_{\mathbb{G}}(w_{\lambda}-u_{\lambda'})\rangle_{\mathfrak{g}_{1}} \, \theta'_{\varepsilon}(w_{\lambda}-u_{\lambda'})
			- \lambda \int_{\Omega}\dfrac{\theta_{\varepsilon}(w_{\lambda}-u_{\lambda'})}{w_{\lambda}^{\gamma}} = 0,
	\end{equation}
	\noindent and
	\begin{equation}\label{eq:solvedByuOverlinelambda}
		\begin{aligned}
			&\int_{\Omega}\langle\nabla_{\mathbb{G}} u_{\lambda'}, \nabla_{\mathbb{G}}(w_{\lambda}-u_{\lambda'})\rangle_{\mathfrak{g}_{1}}\, \theta'_{\varepsilon}(w_{\lambda}-u_{\lambda'})\\
			& \qquad - \lambda' \int_{\Omega}\dfrac{\theta_{\varepsilon}(w_{\lambda}-u_{\lambda'})}{u_{\lambda'}^{\gamma}} - \int_{\Omega}u_{\lambda'}^{2^{\star}_{Q}-1}\theta_{\varepsilon}(w_{\lambda}-u_{\lambda'}) = 0.
		\end{aligned}
	\end{equation}
	Subtracting \eqref{eq:solvedBywlambda} from  \eqref{eq:solvedByuOverlinelambda} we get
	\begin{equation}\label{eq:Subtraction}
		\begin{aligned}
			0 &\geq  - \int_{\Omega}|\nabla(u_{\lambda'}-w_{\lambda})|_{\mathbb{G}}^2 \, \theta'_{\varepsilon}(w_{\lambda}-u_{\lambda'})\\
			&= \int_{\Omega}\left( \dfrac{\lambda'}{u_{\lambda'}^{\gamma}} - \dfrac{\lambda}{w_{\lambda}^{\gamma}} + u_{\lambda'}^{2^{*}-1}\right) \theta_{\varepsilon}(w_{\lambda}-u_{\lambda'})\\
			&\geq \lambda \int_{\Omega}\left( \dfrac{1}{u_{\lambda'}^{\gamma}}- \dfrac{1}{w_{\lambda}^{\gamma}}\right)\theta_{\varepsilon}(w_{\lambda}-u_{\lambda'}).
		\end{aligned}
	\end{equation}
	Now, letting $\varepsilon \to 0^+$ we find that
	$$\int_{\{w_{\lambda}> u_{\lambda'}\}}\left( \dfrac{1}{u_{\lambda'}^{\gamma}}- \dfrac{1}{w_{\lambda}^{\gamma}}\right)\leq 0,$$
	\noindent and this implies that
	$$\left|\left\{ g \in \Omega: w_{\lambda}(g) > u_{\lambda'}(g)\right\}\right| = 0,$$
	\noindent as claimed in \eqref{eq:Claim}. \vspace*{0.1cm}
	
	With \eqref{eq:Claim} at hand, we are ready to complete the proof
	of the lemma: in fact, setting  $\overline{u}= u_{\lambda'}$ and $\underline{u} = w_{\lambda}$, 
	by \eqref{eq:Claim} and Theorem \ref{thm:Singular_Problem} we know that
	\vspace*{0.1cm}
	
	i)\,\,$\underline{u}$ is weak subsolution and $\overline{u}$ is a weak supersolution
	of problem \eqref{eq:Main_Problem}$_\lambda$;
	\vspace*{0.05cm}
	
	ii)\,\,$\underline{u}$ and $\overline{u}$ satisfy assumptions a)-b) in Lemma \ref{lem:2.2Haitao}.
	\vspace*{0.1cm}
	
	\noindent We can then apply Lemma \ref{lem:2.2Haitao}, which therefore proving that problem \eqref{eq:Main_Problem}$_{\lambda}$ admits a weak solution $u_{\lambda}$ for every $\lambda \in (0,\Lambda)$,
	further satisfying
	$$I_{\lambda}(u_\lambda) = \min\{u\in S_0^1(\Omega):\,w_{\lambda}\leq u\leq u_{\lambda'}\}
\leq I_{\lambda}(w_{\lambda}).$$
In particular,
by Theorem \ref{thm:Singular_Problem} we have
\begin{equation} \label{eq:Ilambdaulambdaneg}
 I_{\lambda}(u_{\lambda}) \leq I_{\lambda}(w_{\lambda}) \leq J_{\lambda}(w_{\lambda}) <0.
 \end{equation}
	We now turn to consider the `limit case' $\lambda = \Lambda$. The proof
	in this case is analogous to that in
	\cite[Lemma 2.3]{Haitao}, but we present it here for the sake of completeness.
	
	To begin with, we choose an increasing sequence $\{\lambda_k\}_k\subseteq(0,\Lambda)$ 
	which converges to $\Lambda$ as $k\to+\infty$; accordingly, we let
	$$u_k = u_{\lambda_k}\in S^{1}_{0}(\Omega)$$
	be the weak solution of problem \eqref{eq:Main_Problem}$_{\lambda_k}$ constructed
	above (via the Perron method). On account of 
	\eqref{eq:Ilambdaulambdaneg}, for every $k\geq 1$ we have
	\begin{equation} \label{eq:Ilambdakneg}
		I_{\lambda_k,}(u_{k}) 
		= \dfrac{1}{2} \int_{\Omega}|\nabla_{\mathbb{G}}u_k|^2 - \dfrac{\lambda_k}{1-\gamma}\int_{\Omega}|u_k|^{1-\gamma} - \dfrac{1}{2^{\star}_{Q}}\int_{\Omega}|u_k|^{2^{\star}_{Q}} <0.
	\end{equation}
	Moreover, by using $\varphi = u_k$ in \eqref{eq:weak_sub_super_sol} (recall that $u_k$
	solves \eqref{eq:Main_Problem}$_{\lambda_k}$), we get
	\begin{equation} \label{eq:testwithukzero}
		\int_{\Omega}|\nabla_{\mathbb{G}}u_k|^2 -\lambda_k\int_\Omega u_k^{1-\gamma}
		-\int_\Omega u_k^{2^{\star}_{Q}} = 0.
	\end{equation}
	By combining \eqref{eq:Ilambdakneg}\,-\,\eqref{eq:testwithukzero},
	it is then easy to recognize that the sequence $\{u_k\}_k$ is \emph{bounded in $S^{1}_{0}(\Omega)$};
	as a consequence, we can find a function
	$$u_{\Lambda}\in S^{1}_{0}(\Omega)$$ 
	such that
	(up to a subsequence and as $k\to+\infty$)
	\begin{itemize}
		\item[a)] $u_k\to u_{\Lambda}$ weakly in $S^{1}_{0}(\Omega)$ and strongly
		in $L^p(\Omega)$ for $1\leq p <2^{\star}_{Q}$;
		\item[b)] $u_k\to u_{\Lambda}$ a.e.\,in $\Omega$.
	\end{itemize}
	We now observe that, since $\lambda_k\geq \lambda_1$ for every $k\geq 1$
	(recall that the sequence $\{\lambda_k\}_k$ is increasing),
	by arguing as above we see that
	$u_{\lambda_k}\geq w_{\lambda_1}$, and thus
	$$u_{\Lambda} > 0\quad\text{a.e.\,in $\Omega$}.$$
	Moreover, since $u_k$ solves problem \eqref{eq:Main_Problem}$_{\lambda_k}$,
	we have
	$$\int_{\Omega}\langle \nabla_{\mathbb{G}}u_k,\nabla_{\mathbb{G}}\varphi\rangle_{\mathfrak{g}_{1}} -\lambda_k\int_\Omega u_k^{-\gamma}\varphi 
	-\int_\Omega u_k^{2^{\star}_{Q}-1}\varphi = 0\quad\text{for every $\varphi\in S^{1}_{0}(\Omega)$}.$$
	As a consequence, by letting $k\to+\infty$ in the above identity with the aid
	of the Lebesgue Dominated Convergence theorem
	(see, e.g., the proof of \cite[Lemma 2.2]{Haitao}
	and take into account Remark \ref{rem:ConvergenceS01}) we
	conclude that $u_{\Lambda}$ satisfies
	$$\int_{\Omega}\langle \nabla_{\mathbb{G}}u_{\Lambda},\nabla_{\mathbb{G}}\varphi\rangle_{\mathfrak{g}_{1}} -\Lambda\int_\Omega u_{\Lambda}^{-\gamma}\varphi
	-\int_\Omega u_{\Lambda}^{2^{\star}_{Q}-1}\varphi = 0\quad\text{for every $\varphi\in S^{1}_{0}(\Omega)$},$$
	and this proves that $u_{\Lambda}$ is a weak solution of
	problem \eqref{eq:Main_Problem}$_{\Lambda}$. 
	
	We explicitly
	point out that the convergence of the 
	$\int_{\Omega}\langle \nabla_{\mathbb{G}}u_k,\nabla_{\mathbb{G}}\varphi\rangle_{\mathfrak{g}_{1}}$ 
	to $\int_{\Omega}\langle \nabla_{\mathbb{G}}u_{\Lambda},\nabla_{\mathbb{G}}\varphi\rangle_{\mathfrak{g}_{1}}$ follows
	from the weak convergence of $u_k$ to $u_{\Lambda}$, since 
	$$v\mapsto \int_{\Omega}\langle \nabla_{\mathbb{G}}v,\nabla_{\mathbb{G}}\varphi\rangle_{\mathfrak{g}_{1}}$$
	is a linear and continuous functional on $S^{1}_{0}(\Omega)$. This closes the proof.
\end{proof}

\begin{lemma} \label{lem:25Haitao}
	Let $\underline{u},\overline{u},u_\lambda
	\in S^{1}_{0}(\Omega)$ be, respectively, the weak subsolution, 
	the we\-ak supersolution and the weak solution of problem \eqref{eq:Main_Problem}$_\lambda$
	obtained in Lemma \ref{lem:2.3Haitao}, and assume that $0<\lambda<\Lambda$.
	Then, $u_\lambda$ is a \emph{local minimizer} of $I_{\lambda}$
	in \eqref{eq:FunctionalIlambda}.
\end{lemma}
\begin{proof}
	By contradiction, suppose
	that $u_\lambda$ \emph{is not} a local minimizer for $I_{\lambda}$. Then,
	we can construct a sequence $\{u_j\}_j\subseteq S^{1}_{0}(\Omega)$
	satisfying the following properties:
	\vspace*{0.1cm}
	
	i)\,\,$u_j\to u_\lambda$ in $S^{1}_{0}(\Omega)$ as $j\to+\infty$;
	\vspace*{0.05cm}
	
	ii)\,\,$I_{\lambda}(u_j) < I_{\lambda}(u_\lambda)$ for every $j\in\mathbb{N}$.
	\vspace*{0.1cm}
	
	\noindent We explicitly observe that, by possibly replacing $u_j$ with $z_j = |u_j|$, we may assume that
	$u_j\geq 0$ a.e.\,in $\Omega$ for every $j\geq 1$. 
	In fact, since $u_j\to u_\lambda$ in $S^{1}_{0}(\Omega)$
	and since $u_\lambda > 0$ almost everywhere in $\Omega$, it is easy to recognize that
	$$\text{$|u_j|\to |u_\lambda| = u_\lambda$ in $S^{1}_{0}(\Omega)$ as $j\to+\infty$},$$
	and this shows that property i) is still satisfied by $\{z_j\}_j$. Moreover, we have
	$$I_{\lambda}(|u_j|) = I_{\lambda}(u_j) < I_{\lambda}(u_\lambda)\quad\text{for every $j\geq 1$},$$
	and this shows that also property ii) is still satisfied by the sequence $\{z_j\}_j$.
	Hence, from now on we tacitly understand that $\{u_j\}_j$ is a sequence of \emph{non-negative functions}
	satisfying properties i)-ii) above.
	Accordingly, we set 
	$$v_j := 
	\max\{\underline{u},\min\{\overline{u},u_j\}\}\in S^{1}_{0}(\Omega)$$ 
	and we define
	\begin{itemize}
		\item[$(\ast)$] $\overline{w}_j = (u_j-\overline{u})_+\in S^{1}_{0}(\Omega)_+(\Omega)$ 
		and $\overline{S}_j = \mathrm{supp}(\overline{w}_j)
		= \{u_j\geq \overline{u}\}$;
		\vspace*{0.1cm}
		
		\item[$(\ast)$] $\underline{w}_j = (u_j-\underline{u})_-\in S^{1}_{0}(\Omega)_+(\Omega)$ 
		and $\underline{S}_j = \mathrm{supp}(\underline{w}_j)
		= \{u_j\leq \underline{u}\}$.
	\end{itemize}
	We explicitly observe that, by definition, the following identities hold:
	\begin{equation} \label{eq:indetitiesujvjwj}
		\begin{split}
			\mathrm{a)}&\,\,v_j\in M = \{u\in S^{1}_{0}(\Omega):\,\underline{u}\leq u\leq\overline{u}\}; \\
			\mathrm{b)}&\,\,\text{$v_j \equiv \overline{u}$ on $\overline{S}_j$,
				$v_j \equiv \underline{u}$ on $\underline{S}_j$ and $v_j \equiv u_j$ on 
				$\{\underline{u}<u_j<\overline{u}\}$}; \\
			\mathrm{c)}&\,\,\text{$u_j = \overline{u}+\overline{w}_j$ on $\overline{S}_j$ and
				$u_j = \underline{u}-\underline{w}_j$ on $\underline{S}_j$}.
		\end{split}
	\end{equation}
	Following \cite{Haitao}, we now claim that
	\begin{equation} \label{eq:claimmeasureSn}
		\lim_{n\to+\infty}|\overline{S}_j| = \lim_{n\to+\infty}|\underline{S}_j| = 0.
	\end{equation}
	Indeed, let $\sigma > 0$ be arbitrary and let $\delta > 0$ be such that
	$|\Omega\setminus\Omega_\delta| < \frac{\sigma}{2}$, where we have set
	$\Omega_\delta = \{g\in\Omega:\,d_{\mathbb{G}}(g,\partial\Omega) > \delta\}\Subset\Omega$.
	Since, by construction, 
	$$\underline{u} = w_\lambda\in S^{1}_{0}(\Omega)$$ is the unique solution
	of problem \eqref{eq:Singular_Problem}, 
	we know from Theorem \ref{thm:Singular_Problem} that 
	\begin{equation} \label{eq:underustaccata}
		\text{$u_\lambda\geq \underline{u}\geq C > 0$ a.e.\,in $\Omega_\delta$},
	\end{equation}
	where $C = C(\delta,\underline{u}) > 0$ is a suitable constant 
	(recall that $\underline{u}\leq u_\lambda\leq \overline{u}$).
	\vspace*{0.05cm}
	
	On the other hand, since $u_\lambda$ is a weak solution of problem \eqref{eq:Main_Problem}$_\lambda$,
	and since
	$\overline{u} = u_{\lambda'}$ for some $\lambda<\lambda'<\Lambda$
	(see the proof of Lemma \ref{lem:2.3Haitao}),
	by \eqref{eq:underustaccata} we have 
	\begin{align*}
		-\Delta_{\mathbb{G}} (\overline{u}-u_\lambda) & = \lambda'\overline{u}^{-\gamma}-\lambda u_\lambda^{-\gamma}
		+(\overline{u}^{2^{\star}_{Q}-1}-u_\lambda^{2^{\star}_{Q}-1}) \\
		& (\text{since, by construction, $\underline{u}\leq u_\lambda\leq \overline{u}$
			and $\lambda < \lambda'$}) \\
		& \geq \lambda(\overline{u}^{-\gamma}-u_\lambda^{-\gamma}) \\
		& (\text{by the Mean Value Theorem, for some $\theta\in(0,1)$}) \\
		& = -\gamma\lambda(\theta\overline{u}+(1-\theta)u_\lambda)^{-\gamma-1}(\overline{u}-u_\lambda) \\
		& \geq -\gamma\lambda \underline{u}^{-\gamma-1}(\overline{u}-u_\lambda) \\
		& (\text{here we use \eqref{eq:underustaccata}}) \\
		& \geq -\gamma\lambda C^{-\gamma-1}(\overline{u}-u_\lambda),
	\end{align*}
	in the weak sense on $\Omega_\delta$; as a consequence, we see that $v := \overline{u}-u_\lambda
	\in S^{1}_{0}(\Omega)$ is a \emph{weak supersolution} (in the sense of
	Definition \ref{def:weak_sub_super_sol}) of equation \eqref{eq:PDEOrderzeroHarnack}, with
	$$c(x) = \gamma\lambda C(\delta,\overline{u})^{-\gamma-1} > 0.$$
	Since $v > 0$ a.e.\,on every ball $B\Subset\Omega$ (as $\overline{u} = u_{\lambda'}$ and 
	$\lambda\neq\lambda'$), we can apply again Co\-rol\-la\-ry~\ref{cor:BrezisNirenbergpernoi},
	ensuring the existence of $C_1 = C_1(\delta,u_\lambda,\overline{u}) > 0$ such that
	\begin{equation} \label{eq:vstaccatafinal}
		\text{$v = \overline{u}-u_\lambda \geq C_1 > 0$ a.e.\,in $\Omega_\delta$}.
	\end{equation}
	With \eqref{eq:vstaccatafinal} at hand, we can finally complete the proof
	of \eqref{eq:claimmeasureSn}: in fact, recalling that
	$u_j\to u_\lambda$ in $S^{1}_{0}(\Omega)\hookrightarrow L^2(\Omega)$
	as $j\to+\infty$, from \eqref{eq:vstaccatafinal} we obtain
	\begin{align*}
		|\overline{S}_j| & \leq |\Omega\setminus\Omega_\delta|+|\Omega_\delta\cap\overline{S}_j|
		< \frac{\sigma}{2}+\frac{1}{C_1^2}\int_{\Omega_\delta\cap\overline{S}_j}(\overline{u}-u_\lambda)^2 \\
		& (\text{since $0\leq \overline{u}-u_\lambda\leq u_j-u_\lambda$ a.e.\,in $\overline{S}_j$}) \\
		& < \frac{\sigma}{2}+\frac{1}{C_1^2}\|u_j-u_\lambda\|^2_{L^2(\Omega)} < \sigma,
	\end{align*}
	provided that $j$ is large enough, and this proves that $|\overline{S}_j|\to 0$ as $j\to+\infty$.
	In a very similar fashion, one can prove that $|\underline{S}_j|\to 0$ as $j\to+\infty$.
	\vspace*{0.1cm}
	
	Now we have established \eqref{eq:claimmeasureSn}, we can proceed toward the end of the demonstration
	of the lemma.
	To begin with, using identities b)-c) in \eqref{eq:indetitiesujvjwj} we write
	\begin{align*}
		I_{\lambda}(u_j) & = I_{\lambda}(v_j)
		+ \frac{1}{2}\left(\int_{\Omega}|\nabla_{\mathbb{G}}u_j|_{\mathbb{G}}^2-\int_{\Omega}|\nabla_{\mathbb{G}}v_j|_{\mathbb{G}}^2\right) \\
		&\qquad-\frac{\lambda}{1-\gamma}\int_\Omega(|u_j|^{1-\gamma}-|v_j|^{1-\gamma}) 
		-\frac{1}{2^{\star}_{Q}}\int_\Omega(|u_j|^{2^{\star}_{Q}}-|v_j|^{2^{\star}_{Q}}) \\
		& = I_{\lambda}(v_j)
		+ \frac{1}{2}\int_{\overline{S}_j\cup\underline{S}_j}(|\nabla_{\mathbb{G}} u_j|_{\mathbb{G}}^2-|\nabla_{\mathbb{G}} v_j|_{\mathbb{G}}^2)\\
		&\qquad-\frac{\lambda}{1-\gamma}\int_{\overline{S}_j\cup\underline{S}_j}
		(|u_j|^{1-\gamma}-|v_j|^{1-\gamma}) 
		-\frac{1}{2^{\star}_{Q}}\int_{\overline{S}_j\cup\underline{S}_j}(|u_j|^{2^{\star}_{Q}}-|v_j|^{2^{\star}_{Q}}) \\
		& = I_{\lambda}(v_j) +
		\mathcal{R}^{(1)}_j+\mathcal{R}^{(2)}_j = (\bigstar), \phantom{+\int_{\underline{S}_j\cup\overline{S}_j}}
	\end{align*}
	where we have introduced the shorthand notation
	\begin{align*}
		(\ast)&\,\,\mathcal{R}^{(1)}_j = \frac{1}{2}\int_{\overline{S}_j}
		\big(|\nabla_{\mathbb{G}} (\overline{u}+\overline{w}_j)|_{\mathbb{G}}^2-|\nabla_{\mathbb{G}} \overline{u}|_{\mathbb{G}}^2\big) \\
		&\qquad - \int_{\overline{S}_j}
		\Big\{\frac{\lambda}{1-\gamma}(|\overline{u}+\overline{w}_j|^{1-\gamma}-|\overline{u}|^{1-\gamma})
		+\frac{1}{2^{\star}_{Q}}(|\overline{u}+\overline{w}_j|^{2^{\star}_{Q}}-|\overline{u}|^{2^{\star}_{Q}})\Big\}; \\[0.15cm]
		(\ast)&\,\,\mathcal{R}^{(2)}_j = \frac{1}{2}\int_{\underline{S}_j}
		\big(|\nabla_{\mathbb{G}} (\underline{u}-\underline{w}_j)|_{\mathbb{G}}^2-|\nabla_{\mathbb{G}} \underline{u}|_{\mathbb{G}}^2\big) \\
		& \qquad - \int_{\underline{S}_j}
		\Big\{\frac{\lambda}{1-\gamma}(|\underline{u}-\underline{w}_j|^{1-\gamma}-|\underline{u}|^{1-\gamma})
		+\frac{1}{2^{\star}_{Q}}(|\underline{u}-\underline{w}_j|^{2^{\star}_{Q}}-|\underline{u}|^{2^{\star}_{Q}})\Big\},
	\end{align*}
	\noindent and then we obtain
	$$I_{\lambda}(u_j) = I_{\lambda}(v_j)+A_j+B_j,$$
	where we have used the notation
	\begin{align*}
		(\ast)&\,\,A_j = \frac{1}{2}\int_{\Omega}|\nabla_{\mathbb{G}}\overline{w}_j|_{\mathbb{G}}^2+
		\int_{\Omega}\langle \nabla_{\mathbb{G}}\overline{u},\nabla_{\mathbb{G}}\overline{w}_j\rangle_{\mathfrak{g}_{1}} 
		\\
		& \qquad - \int_{\overline{S}_j}
		\Big\{\frac{\lambda}{1-\gamma}(|\overline{u}+\overline{w}_j|^{1-\gamma}-|\overline{u}|^{1-\gamma})
		+\frac{1}{2^{\star}_{Q}}(|\overline{u}+\overline{w}_j|^{2^{\star}_{Q}}-|\overline{u}|^{2^{\star}_{Q}})\Big\};
		\\[0.15cm]
		(\ast)&\,\,B_j  = \frac{1}{2}\int_{\Omega}|\nabla_{\mathbb{G}}\underline{w}_j|_{\mathbb{G}}^2+
		\int_{\Omega}\langle \nabla_{\mathbb{G}}\underline{u},\nabla_{\mathbb{G}}\underline{w}_j\rangle_{\mathfrak{g}_{1}} 
		\\
		& \qquad - \int_{\underline{S}_j}
		\Big\{\frac{\lambda}{1-\gamma}(|\underline{u}-\underline{w}_j|^{1-\gamma}-|\underline{u}|^{1-\gamma})
		+\frac{1}{2^{\star}_{Q}}(|\underline{u}-\underline{w}_j|^{2^{\star}_{Q}}-|\underline{u}|^{2^{\star}_{Q}})\Big\}.
	\end{align*}
	Now, since we have already recognized that
	$v_j\in M$
	and since, by con\-stru\-ction, we know that $I_{\lambda}(u_\lambda) = \inf_M I_{\lambda}$ 
	(see the proof of Lemma \ref{lem:2.3Haitao}), we get
	\begin{equation} \label{eq:dovecontraddire}
		I_{\lambda}(u_j)\geq I_{\lambda}(u_\lambda)+A_j+B_j.
	\end{equation}
	On the other hand, since $\overline{u}=u_{\lambda'}$ is a \emph{weak supersolution}
	of \eqref{eq:Main_Problem}$_\lambda$, we have
	\begin{align*}
		A_j & = \frac{1}{2}\int_{\Omega}|\nabla_{\mathbb{G}}\overline{w}_j|_{\mathbb{G}}^2+
		\int_{\Omega}\langle \nabla_{\mathbb{G}}\overline{u},\nabla_{\mathbb{G}}\overline{w}_j\rangle_{\mathfrak{g}_{1}} 
		\\
		& \qquad - \int_{\overline{S}_j}
		\Big\{\frac{\lambda}{1-\gamma}(|\overline{u}+\overline{w}_j|^{1-\gamma}-|\overline{u}|^{1-\gamma})
		+\frac{1}{2^{\star}_{Q}}(|\overline{u}+\overline{w}_j|^{2^{\star}_{Q}}-|\overline{u}|^{2^{\star}_{Q}})\Big\} \\
		& (\text{by the Mean Value Theorem, for some $\theta\in (0,1)$}) \\
		& \geq \frac{1}{2}\int_{\Omega}|\nabla_{\mathbb{G}}\overline{w}_j|_{\mathbb{G}}^2
		+ \int_{\overline{S}_j}(\lambda \overline{u}^{-\gamma}+\overline{u}^{2^{\star}_{Q}-1})\overline{w}_j\,dx
		\\
		&\qquad -\int_{\overline{S}_j}\big\{\lambda(\overline{u}+\theta\overline{w}_j)^{-\gamma}
		\overline{w}_j+(\overline{u}+\theta\overline{w}_j)^{2^{\star}_{Q}-1}\overline{w}_j\}dx \\
		& \geq  \frac{1}{2}\int_{\Omega}|\nabla_{\mathbb{G}}\overline{w}_j|_{\mathbb{G}}^2-\int_{\overline{S}_j}
		\big((\overline{u}+\theta\overline{w}_j)^{2^{\star}_{Q}-1}-\overline{u}^{2^{\star}_{Q}-1}\big)\overline{w}_j\,dx
		\\
		& (\text{again by the Mean Value Theorem}) \\
		& \geq \frac{1}{2}\int_{\Omega}|\nabla_{\mathbb{G}}\overline{w}_j|_{\mathbb{G}}^2-C\int_{\overline{S}_j}(\overline{u}^{2^{\star}_{Q}-2}+\overline{w}_j^{2^{\star}_{Q}-2})
		\overline{w}_j^2,
	\end{align*}
	where $C > 0$ is a suitable constant only depending on the dimension $n$.
	From this, by exploiting H\"older's and Sobolev's inequalities, we obtain
	\begin{equation} \label{eq:todeduceAngeqzero}
		\begin{split}
			A_j & = \frac{1}{2}\int_{\Omega}|\nabla_{\mathbb{G}}\overline{w}_j|_{\mathbb{G}}^2 -C\int_{\overline{S}_j}(\overline{u}^{2^{\star}_{Q}-2}+\overline{w}_j^{2^{\star}_{Q}-2})
			\overline{w}_j^2 \\
			& \geq 
			\frac{1}{2}\int_{\Omega}|\nabla_{\mathbb{G}}\overline{w}_j|_{\mathbb{G}}^2 \left\{1-\hat{C}\Big(\int_{\overline{S}_j}\overline{u}^{2^{\star}_{Q}}\,dx
			\Big)^{\frac{2^{\star}_{Q}-2}{2^{\star}_{Q}}}-\hat{C}\left(\int_{\Omega}|\nabla_{\mathbb{G}}\overline{w}_j|^2_{\mathbb{G}}\right)^{(2^{\star}_{Q}-2)/2}\right\},
		\end{split}
	\end{equation}
	where $\hat{C} > 0$ is another constant depending on $Q$.
	
	With \eqref{eq:todeduceAngeqzero} at hand, we are finally ready to complete the proof.
	Indeed, taking into account the above \eqref{eq:claimmeasureSn}, we have
	\begin{align*}
		\lim_{n\to+\infty}\Big(\int_{\overline{S}_j}\overline{u}^{2^{\star}_{Q}}\,dx
		\Big)^{\frac{2^{\star}_{Q}-2}{2^{\star}_{Q}}} = 0; 
	\end{align*}
	moreover, since $u_j\to u_\lambda$ in $S^{1}_{0}(\Omega)$ as $j\to+\infty$, one also get
	\begin{align*}
		0& \leq \int_{\Omega}|\nabla_{\mathbb{G}}\overline{w}_j|_{\mathbb{G}}^2
		=\int_{\overline{S}_j}|\nabla_{\mathbb{G}} (u_j-
		\overline{u})|^2 \\
		& \leq 2 \|u_j-u_\lambda\|_{S^1_0(\Omega)}^2+2 \int_{\overline{S}_j}
		|\nabla_{\mathbb{G}} (u_\lambda-\overline{u})|^2  \to 0\qquad\text{as $j\to+\infty$}.
	\end{align*}
	Ga\-the\-ring these facts, we then infer the existence of some $j_0\geq 1$ such that
	$$A_j\geq \frac{1}{2}\int_{\Omega}|\nabla_{\mathbb{G}}\overline{w}_j|_{\mathbb{G}}^2 \left\{1-\hat{C}\Big(\int_{\overline{S}_j}\overline{u}^{2^{\star}_{Q}} 
	\Big)^{\frac{2^{\star}_{Q}-2}{2^{\star}_{Q}}}-\hat{C}\left(\int_{\Omega}|\nabla_{\mathbb{G}}\overline{w}_j|^2_{\mathbb{G}}\right)^{(2^{\star}_{Q}-2)/2}\right\} \geq 0\quad
	\forall\,\,j\geq j_0.$$
	By arguing in a very similar way, one can prove that $B_j\geq 0$ for every $j\geq j_0$
	(by possibly enlarging $j_0$ if needed); as a consequence, from 
	\eqref{eq:dovecontraddire} we get
	$$I_{\lambda}(u_j)\geq I_{\lambda}(u_\lambda) +A_j+B_j\geq I_{\lambda}(u_\lambda),$$
	but this is contradiction with property ii) of the sequence $\{u_j\}_j$.
\end{proof}

\medskip

\section{Existence of the second solution}\label{sec:Second_Solution}

In this section we are going to find 
a second solution to \eqref{eq:Main_Problem}$_\lambda$.
For this purpose, we will apply the Ekeland's variational principle
similarly to what done in \cite{Haitao}.
For any $\lambda<\Lambda$, the set where we develop the method is
\begin{equation}\label{def:tlambda}
H_\lambda:=\left\{u\in S^1_0(\Omega):\ u\geq u_\lambda\ \m{a.e.~in }\Omega\right\}
\end{equation}
where $u_\lambda$ is the (first, weak) solution to \eqref{eq:Main_Problem}$_\lambda$
determined in Lemma \ref{lem:2.3Haitao}. 

We know by Lemma \ref{lem:25Haitao}
that $r_0>0$ exists such that $r_0<\|u_\lambda\|_{S^1_0(\Omega)}$ and
$I_\lambda(u_\lambda)\leq I_\lambda(u)$ 
for any $u\in S^1_0(\Omega)$ such that $\|u-u_0\|_{S^1_0(\Omega)}\leq r_0$.
Therefore, we are in one of the following two cases
\begin{enumerate}
	\item for every $r\in (0,r_0)$,
	\[
	\inf_{\| u-u_\lambda\|_{S^1_0(\Omega)}=r}I_\lambda(u)=I_\lambda(u_\lambda)
	\]
	\item there exists $r_1\in (0,r_0)$ such that
	\[
	\inf_{\|u-u_\lambda\|_{S^1_0(\Omega)}=r_1}I_\lambda(u)>I_\lambda(u_\lambda).
	\]
\end{enumerate}
We treat the two cases separately.
\subsection{First case} \label{subsec:CaseI}
We are going to prove that for any $r\in (0,r_0)$ there exists
a solution $v_\lambda$ of \eqref{eq:Main_Problem}$_\lambda$ such that
$\|v_\lambda-u_\lambda\|_{S^1_0(\Omega)}=r$, therefore in particular $v_\lambda\not\equiv u_\lambda$.

By hypothesis,
we can find a sequence $\{u_k\}_k\subset H_{\lambda}$ satisfying the following properties:
\begin{itemize}
	\item $\|u_k-u_\lambda\|_{S^1_0(\Omega)} = r$ for every $k\geq 1$;	
	\item $I_\lambda(u_k)\to I_\lambda(u_\lambda) =:\mathbf{c}_\lambda$ as $k\to+\infty$.
\end{itemize}
We then choose $\bar r > 0$ so small that $r-\bar r > 0$ and $r+\bar r<r_0$ and,
accordingly, we consider the subset of $H_\lambda$ defined as follows:
$$X_\lambda = \{u\in H_\lambda:\,r-\bar r\leq\|u-u_\lambda\|_{S^1_0(\Omega)}\leq r+\bar r\}.$$
By construction $u_k\in X_\lambda$ for every $k\geq 1$.
Since it is \emph{closed}, this set $X_\lambda$
is a \emph{complete metric space} when endowed with the distance induced by $\|\cdot\|_{S^1_0(\Omega)}$;
moreover, since $I_\lambda$ is a \emph{real-valued and continuous functional} on $X_\lambda$, and since
$$\textstyle\inf_{X_\lambda}I_\lambda = I_\lambda(u_\lambda)$$
we are entitled to apply
 Ekeland's Variational Principle \cite{Ekeland} to the functional $I_\lambda$ on $X_\lambda$,
provi\-ding us with a sequence $\{v_k\}_k\subset X_\lambda$ such that
\begin{equation} \label{eq:EkelandCaseA}
	\begin{split}
		\mathrm{i)}&\,\, I_\lambda(v_k)\leq I_\lambda(u_k)\leq I_\lambda(u_\lambda)+1/k^2, \\
		\mathrm{ii)}&\,\,\|v_k-u_k\|_{S^1_0(\Omega)}\leq 1/k, \\
		\mathrm{iii)}&\,\, I_\lambda(v_k)\leq I_\lambda(u)+1/k\,
		\|v_k-u\|_{S^1_0(\Omega)}\quad\text{for every $u\in X$}.    
	\end{split}
\end{equation}
We now observe that, since $\{v_k\}_k\subset X_\lambda$ and since the set $X_\lambda$ is \emph{bounded}
in $S^1_0(\Omega)$, 
there exists  $v_\lambda\in S^1_0(\Omega)$ such that (as $k\to+\infty$ and 
up to a sub-sequence)
\begin{equation} \label{eq:limitvkCaseA}
	\begin{split}
		\mathrm{i)}&\,\,\text{$v_k\to v_\lambda$ weakly in $S^1_0(\Omega)$}; \\
		\mathrm{ii)}&\,\,\text{$v_k\to v_\lambda$ strongly in $L^p(\Omega)$ for every $1\leq p<2^\star_Q$}; \\
		\mathrm{iii)}&\,\,\text{$v_k\to v_\lambda$ pointwise
			a.e.\,in $\Omega$}.
	\end{split}
\end{equation}
where we have also used the \emph{compact embedding} 
$S^1_0(\Omega)\hookrightarrow L^2(\Omega)$ (see Section \ref{sec:fsandpdes}).

\begin{lemma}\label{lem:casoa1}
	The function $v_\lambda$ is a weak solution of \eqref{eq:Main_Problem}$_\lambda$.
\end{lemma}
\proof
We fix $w\in H_\lambda$ and choose $\varepsilon_0=\varepsilon_0(w,\lambda)$
sufficiently small that $v_k+\varepsilon(w-v_k)\in X_\lambda$ for every $0<\varepsilon<\varepsilon_0$.
We point out that such an $\varepsilon_0$ exists (for $k$ sufficiently big)
because by the properties above $\|u_k-u_\lambda\|_{S^1_0(\Omega)}=r$ and
\begin{align*}
 r-\frac{1}{k}\leq \|u_k-u_\lambda\|_{S^1_0(\Omega)}-\|v_k-u_k\|_{S^1_0(\Omega)}&\leq \|v_k-u_\lambda\|_{S^1_0(\Omega)}\\
& \leq \|u_k-u_\lambda\|_{S^1_0(\Omega)}+\|v_k-u_k\|_{S^1_0(\Omega)}\leq r+\frac{1}{k}.
\end{align*}
By setting $u=v_k+\varepsilon(w-v_k)$ in \eqref{eq:EkelandCaseA}
we get
\[
\frac{I_\lambda(v_k+\varepsilon(w-v_k))-I_\lambda(v_k)}{\varepsilon}\geq -\frac{1}{k}\, \|w-v_k\|_{S^1_0(\Omega)}.
\]
Taking the limit for $\varepsilon\to0^+$, we get
\[
\begin{split}
-\frac{1}{k}\,\|w-v_k\|_{S^1_0(\Omega)}\leq \int_\Omega\langle \nabla_\G v_k,&\nabla_\G(w-v_k)\rangle_{\mathfrak{g}_{1}}-\int_\Omega v_k^{2^\star_Q-1}(w-v_k)\\
&-\lambda\lim_{\varepsilon\to0^+} \int_\Omega (v_k+\theta \varepsilon(w-v_k))^{-\gamma}(w-v_k)
\end{split}
\]
with $0<\theta<1$. Observe that $v_k+\varepsilon(w-v_k)\geq u_\lambda$ a.e.~in $\Omega$,
and $w-v_k\in S^1_0(\Omega)$, 
therefore 
\[
\int_\Omega \left| (v_k+\theta\varepsilon(w-v_k))^{-\gamma}(w-v_k)\right|\leq \int_\Omega u_\lambda^{-\gamma}|w-v_k|<\infty.
\]
By the dominated convergence theorem, this implies
\[
\lim_{\varepsilon\to0^+}\int_\Omega(v_k+\theta\varepsilon(w-v_k))^{-\gamma}(w-v_k)=\int_\Omega v_k^{-\gamma}(w-v_k),
\]
and therefore we get the inequality
\begin{equation}\label{eq:ineq2.21haitao}
\begin{split}
	& -\frac{1}{k}\, \|w-v_k\|_{S^1_0(\Omega)}\\
	& \qquad\leq \int_\Omega\langle \nabla_\G v_k,\nabla_\G(w-v_k)\rangle_{\mathfrak{g}_{1}}-\int_\Omega v_k^{2^\star_Q-1}(w-v_k)
	-\lambda \int_\Omega v_k^{-\gamma}(w-v_k)\ \ \forall w\,\,\in H_\lambda.
	\end{split}
	\end{equation}
For any $\varphi\in S^1_0(\Omega)$ and any $\varepsilon>0$ we introduce
\begin{itemize}
	\item $\varphi_{k,\varepsilon}:=v_k+\varepsilon\varphi-u_\lambda$;
	\item $\varphi_{\varepsilon}:=v_\lambda+\varepsilon\varphi-u_\lambda$.
\end{itemize}
By construction $w:=v_k+\varepsilon\varphi+(\varphi_{k,\varepsilon})_-\in H_\lambda$, 
then, by \eqref{eq:ineq2.21haitao},
\begin{equation}\label{eq:ineq2.22haitao}
	\begin{split}
-\frac{1}{k}\, \|\varepsilon\varphi+(\varphi_{k,\varepsilon})_-\|_{S^1_0(\Omega)}\leq \int_\Omega\langle \nabla_\G v_k,&\nabla_\G(\varepsilon\varphi+(\varphi_{k,\varepsilon})_-)\rangle_{\mathfrak{g}_{1}}
-\int_\Omega v_k^{2^\star_Q-1}(\varepsilon\varphi+(\varphi_{k,\varepsilon})_-)\\
&-\lambda\int_\Omega v_k^{-\gamma}(\varepsilon\varphi+(\varphi_{k,\varepsilon})_-)
\end{split}
\end{equation}
We aim to pass to the limiti for $k\to+\infty$ and $\varepsilon\to0^+$.
We first observe that, because of~(\ref{eq:limitvkCaseA}-iii),
\[
(\varphi_{\varepsilon,k})_-\to(\varphi_\varepsilon)_-\ \m{pointwise a.e.~in }\Omega,\ \m{for }k\to+\infty.
\]
Since $v_k\geq u_\lambda>0$ a.e.~in $\Omega$ for any $k$, we get the following,
\begin{itemize}
	\vspace{7pt}
	\item $v_k^{2^\star_Q-1}\, (\varphi_{k,\varepsilon})_-=v_k^{2^\star_Q-1}\, (u_\lambda-\varepsilon\varphi-v_k)\cdot \mathbf{1}_{\{u_\lambda-\varepsilon\varphi-v_k\geq0\}}
		\leq(u_\lambda+\varepsilon|\varphi|)^{2^\star_Q}$
		\vspace{7pt}
		\item $v_k^{-\gamma}\, (\varphi_{k,\varepsilon})_-\leq 2\varepsilon \, |\varphi|\, u_\lambda^{-\gamma}$.
		\end{itemize}
By the results above and Remark	\ref{rem:defweaksolPb}-(2), we can use the Dominated
Convergence Theorem, obtaining
\begin{equation}\label{eq:gather1}
	\begin{split}
	\lim_{k\to+\infty}\biggl(\int_\Omega v_k^{2^\star_Q-1}(\varepsilon\varphi+(\varphi_{k,\varepsilon})_-)
	&+\lambda\int_\Omega v_k^{-\gamma}(\varepsilon\varphi+(\varphi_{k,\varepsilon})_-)\biggl)\\
	&=\int_\Omega v_\lambda^{2^\star_Q-1}(\varepsilon\varphi+(\varphi_\varepsilon)_-)+
	\lambda\int_\Omega v_\lambda^{-\gamma}(\varepsilon\varphi+(\varphi_\varepsilon)_-)
	\end{split}
\end{equation}
Focusing on the remaining term on the RHS of \eqref{eq:ineq2.22haitao},
with computations similar to the ones carried out in \cite[Lemma 3.4]{BadTar},
we get
\[
\int_\Omega \langle \nabla_\G v_k,\nabla_\G (\varphi_{k,\varepsilon})_-\rangle_{\mathfrak{g}_{1}}\leq
\int_\Omega\langle \nabla_\G{v_\lambda},\nabla_\G(\varphi_{\varepsilon})_-\rangle_{\mathfrak{g}_{1}} +o(1)\ \m{as}\ k\to+\infty
\]
and as a consequence of the weak convergence $v_k\to v_\lambda$ in $S_0^1(\Omega)$, we obtain
\begin{equation}\label{eq:gather2}
	\int_\Omega\langle \nabla_\G v_k,\nabla_\G(\varepsilon\varphi+(\varphi_{k,\varepsilon})_-)\rangle_{\mathfrak{g}_{1}}
	\leq \int_\Omega\langle \nabla_\G {v_\lambda},\nabla_\G(\varepsilon\varphi+(\varphi_{\varepsilon})_-)\rangle_{\mathfrak{g}_{1}} +o(1)\ \m{as}\ k\to+\infty.
\end{equation}
 Gathering \eqref{eq:gather1} and \eqref{eq:gather2}, and taking into account
that $\|(\varphi_{k,\varepsilon})_-\|_{S^{1}_{0}(\Omega)}$ is \emph{uniformly bo\-un\-ded} with respect to $k$ 
(as the same is true of $v_k$), we can finally
pass to the limit as $k\to+\infty$ in \eqref{eq:ineq2.22haitao}, obtaining
\begin{equation}
	\int_\Omega\langle \nabla_\G v_\lambda,\nabla_\G(\varepsilon\varphi+(\varphi_\varepsilon)_-)\rangle_{\mathfrak{g}_1}\geq
	\int_\Omega v_\lambda^{2^\star_Q-1}(\varepsilon\varphi+(\varphi_\varepsilon)_-)+
	\lambda\int_\Omega v_\lambda^{-\gamma}(\varepsilon\varphi+(\varphi_\varepsilon)_-).
\end{equation}
Exploiting the computations carried out in \cite[Lemma 2.6]{Haitao}, we get
\[
\begin{split}
	\int_\Omega\langle \nabla_\G v_\lambda,\nabla_\G \varphi\rangle_{\mathfrak{g}_{1}} 
	&- \lambda \int_\Omega v_\lambda^{-\gamma}\varphi-\int_\Omega v_\lambda^{2^\star_Q-1}\varphi\\
	&\geq-\frac{1}{\varepsilon} \left(\int_\Omega \langle \nabla_\G v_\lambda,\nabla_\G (\varphi_\varepsilon)_-\rangle_{\mathfrak{g}_{1}}
	-\lambda\int_\Omega v_\lambda^{-\gamma}(\varphi_\varepsilon)_--\int_\Omega v_\lambda^{2^\star_Q-1}(\varphi_\varepsilon)_-\right)\\
	&(\m{since }u_\lambda\m{ is a solution of \eqref{eq:Main_Problem}}_\lambda)\\
	&=-\frac{1}{\varepsilon} \biggl(\int_\Omega \langle \nabla_\G(v_\lambda-u_\lambda),\nabla_\G(\varphi_\varepsilon)_-\rangle_{\mathfrak{g}_{1}}
	-\lambda\int_\Omega(v_\lambda^{-\gamma}-u_\lambda^{-\gamma})(\varphi_\varepsilon)_-\\
	&\ \ \ 	-\int_\Omega(v_\lambda^{2^\star_Q-1}-u_\lambda^{2^\star_Q-1})(\varphi_\varepsilon)_-\biggl)\\
	\vspace{6pt}
	&(\m{since }v_\lambda\geq u_\lambda\m{ a.e.~and }v_\lambda=\lim_{k\to+\infty}v_k)\\
	&\geq -\frac{1}{\varepsilon} 
	\biggl(\int_{\{u_\lambda> v_\lambda+\varepsilon\varphi\}}\langle \nabla_\G(v_\lambda-u_\lambda),\nabla_\G (v_\lambda-u_\lambda+\varepsilon\varphi)\rangle_{\mathfrak{g}_{1}}\\
	&\ \ \ 	-\lambda\int_{\{u_\lambda> v_\lambda+\varepsilon\varphi\}}(v_\lambda^{-\gamma}-u_\lambda^{-\gamma})(v_\lambda-u_\lambda+\varepsilon\varphi)\biggl)\\
	&\geq o(1)\ \ \ \m{as}\ \varepsilon\to0^+,
\end{split}
\]
\noindent where we used \eqref{eq:Grad_and_level_set}
and, as before, that
\[
\bigcap_{\varepsilon>0}\{u_\lambda> v_\lambda+\varepsilon\varphi\}=\varnothing.
\]
This implies that $|\{u_\lambda>v_\lambda+\varepsilon\varphi\}|\to 0$
and therefore, by letting $\varepsilon\to0^+$ in the inequality above, we conclude that
\begin{equation}
	\int_\Omega\langle \nabla_\G v_\lambda,\nabla_\G\varphi\rangle_{\mathfrak{g}_{1}}-\lambda\int_\Omega v_\lambda^{-\gamma}\varphi
	-\int_\Omega v_\lambda^{2^\star_Q-1}\varphi\geq0,
\end{equation}
and as $\varphi \in S^1_0(\Omega)$ is arbitrary, this allows to conclude that $v_\lambda$
is a weak solution of \eqref{eq:Main_Problem}$_\lambda$ as claimed.\fine

\color{black}

It remains to show that $v_\lambda$ and $u_\lambda$ are different solutions.
In order to do that we use the following lemma.

\begin{lemma}\label{lem:casoa2}
	If $v_\lambda$ is the solution introduced above,
	then $\|v_\lambda-u_\lambda\|_{S^{1}_{0}(\Omega)}=r$.
\end{lemma}
\proof 
We want to prove that 
\begin{equation}\label{eq:stronconv}
	v_k\to v_\lambda\ \m{strongly in }S_0^1(\Omega)\ \m{as}\ k\to+\infty.
\end{equation}
Indeed, if this is the case, owing to the fact that $\|u_k-u_\lambda\|_{S^{1}_{0}(\Omega)}=r$ for any $k$,
we have
\[
r-\|v_k-u_k\|_{S^{1}_{0}(\Omega)} \leq \|v_k-u_\lambda\|_{S^{1}_{0}(\Omega)} \leq r+\|v_k-u_k\|_{S^{1}_{0}(\Omega)},
\]
which with the strong convergence and the fact that $\|v_k-u_k\|_{S^{1}_{0}(\Omega)}\leq \frac{1}{k}$
implies 
$$\|v_\lambda-u_\lambda\|_{S^{1}_{0}(\Omega)}=r.$$
We turn to prove \eqref{eq:stronconv}. 
As $v_k\to v_\lambda$ weakly in $S_0^1(\Omega)$ as $k\to+\infty$,
we proceed as in the proof of
Lemma \ref{lem:Lambdafinito} obtaining the following analogs of 
\eqref{eq:2.3Haitao}-\eqref{eq:2.5Haitao}
\begin{align}
	\int_\Omega v_k^{1-\gamma}&=\int_\Omega v_\lambda^{1-\gamma}+o(1)\label{eq:case1concl1}\\
	\vspace{9pt}
	\| v_k\|^{2^\star_Q}_{L^{2^\star_Q}}(\Omega)&=\|v_\lambda\|_{L^{2^\star_Q}(\Omega)}^{2^\star_Q}+
	\|v_k-v_\lambda\|_{L^{2^\star_Q}(\Omega)}^{2^\star_Q}+o(1)\\
	\|v_k\|_{S^{1}_{0}(\Omega)}^2&=\|v_\lambda\|_{S^{1}_{0}(\Omega)}^2+\|v_k-v_\lambda\|_{S^{1}_{0}(\Omega)}^2+o(1)\label{eq:case1concl2}
\end{align}
Also, because of (\ref{eq:limitvkCaseA})-(ii)
we get
\begin{equation}\label{eq:case1concl3}
	\lim_{k\to+\infty}\int_\Omega|v_k-v_\lambda|^{1-\gamma}=0.
\end{equation}
Therefore, choosing $w=v_\lambda$ in \eqref{eq:ineq2.21haitao},
we obtain
\[
\begin{split}
	\int_\Omega|\nabla_\G(v_k-v_\lambda)|^2+\lambda\int_\Omega v_k^{-\gamma}v_\lambda
	&\leq \lambda\int_\Omega v_k^{1-\gamma}+\int_\Omega v_k^{2^\star_Q-1}(v_k-v_\lambda)+o(1)\\
	&\leq \lambda\int_\Omega v_\lambda^{1-\gamma}+\|v_k-v_\lambda\|_{L^{2^\star_Q}(\Omega)}^{2^\star_Q}+o(1).	
\end{split}
\]
Observing that $0\leq v_k^{-\gamma}v_\lambda\leq u_\lambda^{-\gamma}v_\lambda\in L^1(\Omega)$
we can use again the Dominated Convergence Theorem,
\begin{equation}
	\int_\Omega v_k^{-\gamma}v_\lambda\to\int_\Omega v_\lambda^{1-\gamma}\ \ \m{as}\ k\to+\infty.
\end{equation}
To proceed further, we choose $w=2v_k\in H_\lambda$,
yielding
\[
\|v_k\|_{S^{1}_{0}(\Omega)}^2-\|v_k\|_{L^{2^\star_Q}(\Omega)}^{2^\star_Q}-\lambda\int_\Omega v_k^{1-\gamma}\geq-\frac{1}{k}\|v_k\|_{S^{1}_{0}(\Omega)}^2=o(1),
\]
combining the former equality
 with the follwing consequence of $v_\lambda$ being a solution of \eqref{eq:Main_Problem}$_\lambda$,
 \[
 \|v_\lambda\|_{S^{1}_{0}(\Omega)}^2-\|v_\lambda\|_{L^{2^\star_Q}(\Omega)}^{2^\star_Q}-\lambda\int_\Omega v_k^{1-\gamma}\geq o(1)\ \ \m{as}\ k\to+\infty
 \]
 we get
 \begin{equation}\label{eq:vlambda1}
 	\|v_k-v_\lambda\|_{S^{1}_{0}(\Omega)}^2\geq \|v_k-v_\lambda\|_{L^{2^\star_Q}(\Omega)}^{2^\star_Q}+o(1)\ \ \m{as}\ k\to+\infty.
 \end{equation}
 Assuming without loss of generalities that $I_\lambda(u_\lambda)\leq I_\lambda(v_\lambda)$,
 from \eqref{eq:EkelandCaseA} and \eqref{eq:case1concl1}-\eqref{eq:case1concl2}
 we obtain
 \[
 \begin{split}
 	I_\lambda(v_k-v_\lambda)&=I_\lambda(v_k)-I_\lambda(v_\lambda)+o(1)\\
 	&\leq I_\lambda(u_\lambda)-I_\lambda(v_\lambda)+\frac{1}{k^2}+o(1)\\
 	&=o(1)\ \ \m{as}\ k\to+\infty
 \end{split}
 \]
 which, together with \eqref{eq:case1concl3}, gives
 \begin{equation}\label{eq:vlambda2}
 	\frac{1}{2}\|v_k-v_\lambda\|_{S^{1}_{0}(\Omega)}^2-\frac{1}{2^\star_Q}\|v_k-v_\lambda\|^{2^\star_Q}_{L^{2^\star_Q}(\Omega)}=
 	I_\lambda(v_k-v_\lambda)+\frac{\lambda}{1-\gamma}\int_\Omega|v_k-v_\lambda|^{1-\gamma}\leq o(1).
 \end{equation}
 From \eqref{eq:vlambda1} and \eqref{eq:vlambda2}
 we finally conclude
 \[
 \lim_{k\to+\infty}\|v_k-v_\lambda\|_{L^{2^\star_Q}(\Omega)}^{2^\star_Q}=\lim_{k\to+\infty}\|v_k-v_\lambda\|_{S^{1}_{0}(\Omega)}^2=0,
 \]
 proving \eqref{eq:stronconv}.\fine
 
 We showed that for any $r\in (0,r_0)$ such that $\inf I_\lambda(u)=I_\lambda(u_\lambda)$
 (where the infimum is taken on the set of $u\in S^1_0(\Omega)$ such that $\|u-u_\lambda\|_{S^{1}_{0}(\Omega)}=r$)
 there exists a solution $v_\lambda$ to \eqref{eq:Main_Problem}$_\lambda$
 such that~$v_\lambda\not\equiv u_\lambda$.\newline

\subsection{Second case} \label{subsec:CaseII}

We are going to prove that if there exists $r_1\in(0,r_0)$ such that
\begin{equation}\label{eq:case2}
\inf_{\|u-u_\lambda\|_{S^{1}_{0}(\Omega)}=r_1}I_\lambda(u)>I_\lambda(u_\lambda)
\end{equation}
for some $r_1\in (0,r_0)$, then 
there exists a (second) solution $v_\lambda$ of the problem \eqref{eq:Main_Problem}$_\lambda$
such that $0<u_\lambda<v_\lambda$.

Consider the  space of continuous curves
$C([0,1],H_\lambda)$ endowed with the max distance
\begin{equation}\label{maxdistance}
d(\eta,\eta')=\max_{t\in[0,1]}\|\eta(t)-\eta'(t)\|_{S^1_0(\Omega)}.
\end{equation}
We recall that the space $H_\lambda$ was introduced in \eqref{def:tlambda}.
We define the following complete metric space,
\begin{equation}\label{def:gammal}
\Gamma_\lambda:=\left\{\eta\in C([0,1],H_\lambda):\begin{array}{l}
	%\diamond&
	 \eta(0)=u_\lambda,\\[0.1cm] 
	%\diamond& 
	\|\eta(1)-u_\lambda\|_{S^{1}_{0}(\Omega)}>r_1,\\[0.1cm]
	%\diamond& 
	I_\lambda(\eta(1))<I_\lambda(u_\lambda)
\end{array}
\right\}.
\end{equation}
In order to apply Ekeland's variational principle,
we need to show that $\Gamma_\lambda\neq \varnothing$
and to estimate the minimax level
\[
\ell_0:=\inf_{\eta\in\Gamma_\lambda}\max_{t\in[0,1]}I_\lambda(\eta(t)).
\]

We consider the functions $U_\varepsilon$ introduced
at \eqref{eq:def_u_varepsilon}.
For any $a\in \G$, we also consider the functions
\[
U_{\varepsilon,a}(g):=U_\varepsilon(a^{-1}\diamond g)
= {\varphi(a^{-1}\diamond g) T_{\varepsilon}(a^{-1}\diamond g)},
\]
where $\{T_\varepsilon\}$ is the one-parameter family 
of functions defined in \eqref{eq:Minimi_riscalati} starting from a fixed minimizer $T$
of the Sobolev Inequality \eqref{eq:Sobolev_Ineq}.
%where the dot stands for the multiplication in $\G$.

\begin{lemma}\label{lem:claim2.7haitao}
	There exists $\varepsilon_0>0$, $a\in\Omega$ and $R_0\geq1$ such that
\begin{align}
	 & I_\lambda(u_\lambda + R U_{\varepsilon,a}) < I_\lambda(u_\lambda) 
	&& \forall \varepsilon \in (0, \varepsilon_0),\ \forall R \geq R_0 \\
& I_\lambda(u_\lambda + t R_0 U_{\varepsilon,a}) < I_\lambda(u_\lambda) + \frac{1}{Q} S_\G^{Q/2} 
	&& \forall t \in [0, 1],\ \forall \varepsilon \in (0, \varepsilon_0)\label{eq:lemma2.7Haitao}
\end{align}
where $S_\G$ is the best Sobolev constant introduced at \eqref{eq:def_best_constant}.
\end{lemma}
\proof

We start by breaking down the term $I_\lambda(u+tRU_{\varepsilon,a})$. From now on
we will use the notation $p:=2^\star_Q-1=\frac{Q+2}{Q-2}$
for the sake of a simpler writing. We have

\begin{align*}
    I_\lambda(u+tRU_{\varepsilon,a})&=
    \frac{1}{2}\int|\nabla_{\G}u|^2-\frac{1}{p+1}\int u^{p+1}-\frac{\lambda}{1-\gamma}\int u^{1-\gamma}\\
    &\ \ \ \ +tR \left(\int\langle \nabla_{\G}u,\nabla_{\G}U_{\varepsilon,a}\rangle_{\mathfrak{g}_{1}}-\int u^pU_{\varepsilon,a}-\lambda\int u^{-\gamma}U_{\varepsilon,a}\right)\tag{C}\label{eq:C}\\
    &\ \ \ \ -\frac{1}{p+1}\left(\int(u+tRU_{\varepsilon,a})^{p+1}-\int u^{p+1}\right)+tR\, \int u^pU_{\varepsilon,a}\tag{A}\label{eq:A}\\
    &\ \ \ \ +\frac{t^2R^2}{2} \int|\nabla_{\G}U_{\varepsilon,a}|^2\tag{B}\label{eq:B}\\
    &\ \ \ \ -\frac{\lambda}{1-\gamma}\left(\int(u+tRU_{\varepsilon,a})^{1-\gamma}-\int u^{1-\gamma}\right)+\lambda\, tR\, \int u^{-\gamma}U_{\varepsilon,a}\tag{D}\label{eq:D}
\end{align*}
where all the integrals are over $\Omega\subset \G$.
We treat the conclusion of the calculation and after that we pass to the breakdown of each part of the above sum.
In particular, we are going to prove that, as $\varepsilon \to 0^+$, we have
\begin{align*}
    \eqref{eq:C}&\ =0\\
\eqref{eq:A}&\ =-\frac{A\, t^{p+1}R^{p+1}}{p+1}-t^pR^p\, K\, \varepsilon^{\frac{Q-2}{2}}+o\left(\varepsilon^{\frac{Q-2}{2}}\right)\\
    \eqref{eq:B}&\ =\frac{B\, t^2R^2}{2}+o\left(\varepsilon^{\frac{Q-2}{2}}\right)\\
    \eqref{eq:D}&\ =o\left(\varepsilon^{\frac{Q-2}{2}}\right)
\end{align*}
for an opportune choice of positive constants $A,B,K$.
Again in the spirit of a simpler notation,
we define $s:=tR$ and $S:=\left(\frac{B}{A}\right)^{\frac{1}{p-1}}$.\newline

We follow the approach of \cite{Tarantello} and introduce
\[
f_\varepsilon(s):=\frac{B\, s^2}{2}-\frac{A \, s^{p+1}}{p+1}-s^p\, K\, \varepsilon^n.
\]
We denote by $s_\varepsilon>0$ the point where $f_\varepsilon$ achieves its max. Observe that
\begin{equation}\label{eqboh}
	f'_\varepsilon(s)=Bs-As^p-ps^{p-1}K\varepsilon^n
\end{equation}
and also
\begin{equation}\label{bsas}
BS-AS^p=0.
\end{equation}
The last equations \eqref{eqboh} and \eqref{bsas} imply
\[
S>s_\varepsilon>0\ \m{ and } \ s_\varepsilon\to S\m{ as }\varepsilon\to0^+.
\]
Let's write
\[
s_\varepsilon=S(1-t_\varepsilon).
\]
We can develop for $t_\varepsilon$,
\[
t_\varepsilon\, (p-1)\left(\frac{B^p}{A}\right)^{\frac{1}{p-1}}=p\, \frac{B}{A}\, K\, \varepsilon^n+o(\varepsilon^n), \quad \textrm{ as } \varepsilon \to 0^+.
\]
and we obtain, as $\varepsilon \to 0^+$,
\begin{align*}
    I_\lambda(u+tRU_{\varepsilon,a})
    &<I_\lambda(u)+\frac{Bs_\varepsilon^2}{2}-\frac{As_\varepsilon^{p+1}}{p+1}-s_\varepsilon^{p}K\varepsilon^n+o(\varepsilon^n)\\
    &=I_\lambda(u)+\frac{BS^2}{2}-\frac{AS^{p+1}}{p+1}-BSt_\varepsilon+AS^pt_\varepsilon-S^{p}K\varepsilon^n+o(\varepsilon^n)\\
    &\ (\m{and therefore by Eq.~\eqref{bsas}})\\
    &=I_\lambda(u)+\frac{BS^2}{2}-\frac{AS^{p+1}}{p+1}-S^pK\varepsilon^n+o(\varepsilon^n)\\
    &=I_\lambda(u)+\left(\frac{1}{2}-\frac{1}{p+1}
    \right)\, \frac{B^{\frac{p+1}{p-1}}}
    {A^{\frac{2}{p-1}}}
    -S^pK\varepsilon^n+o(\varepsilon^n)
\end{align*}
which allows to conclude proving \eqref{eq:lemma2.7Haitao}, because
$\frac{1}{2}-\frac{1}{p+1}=\frac{1}{Q}$ and
moreover, we will prove below the following equalities
\[
A=\|T\|_{L^{2^\star_Q}(\G)},\ B=\||\nabla_\G T|\|_{L^2(\G)}^2,
\]
therefore, by definition
\[
 S_{\G}=\frac{B}{A^\frac{2}{p+1}}.
\]

\subsubsection*{Breaking down part \eqref{eq:C}}
This part
comes directly from the Definition \ref{def:weak_sub_super_sol}
of a weak solution,
by using $U_{\varepsilon,a}$ as test function.
Therefore, this term is null.

\subsubsection*{Breaking down part \eqref{eq:B}}

We set the constant $B$ to be
\[
B:=\||\nabla_{\G}T|\|_{L^2(\G)}^2=\||\nabla_{\G}T_{\varepsilon,a}|\|_{L^2(\G)}^2=S_\G^{Q\slash 2},
\]
then, as already stated in \eqref{eq:Stima_grad_u_varepsilon},
we know that
\[
\||\nabla_\G U_{\varepsilon,a}|\|_{L^2(\G)}^2=B+O(\varepsilon^{Q-2}).
\]
Therefore, if we set $\varphi$ and $a$ as to have $\left(a\cdot \supp( \varphi)\right)\Subset \Omega$,
then we can
 summarize part \eqref{eq:B} by 
\begin{equation}\label{partB}
    \frac{Bt^2R^2}{2}+t^2R^2\, O(\varepsilon^{Q-2}).
\end{equation}

\subsubsection*{Breaking down part \eqref{eq:A}}

This part requires a further breaking down,
\begin{align*}
    -\frac{1}{p+1}&\left(\int_\Omega (u+tRU_{\varepsilon,a})^{p+1}-\int_\Omega u ^{p+1}\right)+tR\, \int_\Omega u^pU_{\varepsilon,a} \\
&= -\frac{t^{p+1}R^{p+1}}{p+1}\int_\Omega U_{\varepsilon,a}^{p+1}\tag{A1}\\
& \quad-t^pR^p\int_\Omega u\, U_{\varepsilon,a}^p\tag{A2}\\
& \quad +X_\varepsilon + Y_\varepsilon\tag{A3}
\end{align*}

In particular, we will prove that for almost every $a\in \Omega$, the following
are true,
\begin{equation}\label{a1}\tag{A1}
    -\frac{t^{p+1}R^{p+1}}{p+1}\int_\Omega U_{\varepsilon,a}^{p+1}=-\frac{A\, t^{p+1}R^{p+1}}{p+1}+O(\varepsilon^Q)
\end{equation}
for an opportune positive constant $A$,
\begin{equation}\label{a2}\tag{A2}
-t^pR^p\int_\Omega u \, U_{\varepsilon,a}^p=-t^pR^p\, K\, \varepsilon^{\frac{Q-2}{2}}+o(\varepsilon^{\frac{Q-2}{2}})
\end{equation}
 for an opportune positive constant $K$,
\begin{equation}\label{a3}\tag{A3}
    X_\varepsilon + Y_\varepsilon=t^\be R^\be\, o(\varepsilon^{\frac{Q-2}{2}})
\end{equation}
for some $\beta$ such that $1< \beta< \frac{Q}{Q-2}$.\newline

We will use widely the
\cite[Lemma 4]{BrezisNirenberg}, that we state here precisely

\begin{lemma}\label{lemmetto}
    For any $x,y\in \R$ and $q\geq1$ we have
    \[
        \left||x+y|^q
        -|x|^q-|y|^q
        -qxy(|x|^{q-2}+|y|^{q-2})
        \right|
        =
        X+Y
    \]
    with $X,Y$ respecting the following,
    \begin{enumerate}
        \item if $q\leq 3$,the $X=0$ when $|x|<|y|$ and $Y=0$ when $|x|>|y|$,
        moreover there exists a constant $C=C(q)$ such that 
        \begin{align*}
            |X|&\leq C\, |x||y|^{q-1}\ \m{ if }|x|\geq |y|\\
            |Y|&\leq C\, |x|^{q-1}|y|\ \m{ if }|x|\leq |y|.
        \end{align*}
        \item if $q>3$, there exists a constanct $C=C(q)$ such that 
        \begin{align*}
            |X+Y|\leq C\, \left(
            |x|^{q-2}y^2+x^2|y|^{q-2}
            \right)
        \end{align*}
    \end{enumerate}
\end{lemma}
We apply this lemma after imposing
\begin{align*}
    q&=2^\star_Q=p+1\\
    x&=u\\
    y&=tR\, U_{\varepsilon,a}.
\end{align*}

\textit{Part \eqref{a1}}:
we write
\begin{equation}\label{eq:t+phi-1}
U_{\varepsilon,a}^{p+1}=T_{\varepsilon,a}^{p+1}+T_{\varepsilon,a}^{p+1}\, (\varphi_a^{p+1}\, \mathbf{1}_\Omega-1).
\end{equation}
where $T_{\varepsilon,a}(g):=T_{\varepsilon}(a^{-1}\diamond g)$
and $\varphi_a(g):=\varphi(a^{-1}\diamond g)$.
By definition there exists a neighborhood 
%$V\Subset \Omega$ 
$V$ of $a\in\Omega$
such that $\varphi_a|_V\equiv 1$. Moreover, 
as a consequence of \eqref{eq:Stima_Bonf_Ugu}, we have that
$$T_{\varepsilon,a}=\varepsilon^{\frac{Q-2}{2}}\, \widetilde T_{\varepsilon,a}$$
\noindent with $\widetilde T_{\varepsilon,a}$ 
uniformly bounded on $\Omega\backslash V$ for $\varepsilon\to0^+$. As $\varphi_a|_V\equiv 1$, 
the integral of the second term in \eqref{eq:t+phi-1}
is an $O\left(\varepsilon^{\frac{Q-2}{2}\, (p+1)}\right)=O(\varepsilon^Q)$.\newline

For the first term, we set the constant
\[
A:=\|T_{\varepsilon,a}\|_{L^{p+1}(\G)}^{p+1}=\|T\|_{L^{2^\star_Q}(\G)}^{2^\star_Q}=S_\G^{Q\slash 2}.
\]
Therefore, combining
these terms, the integral of $U_{\varepsilon,a}^{p+1}$ equals
\[
A+O(\varepsilon^{Q}) \quad \textrm{ as } \varepsilon \to 0^{+},
\]
as we wanted to prove.

\cor{Part \eqref{a2}}:
to treat this part we consider 
some known results about convolutions 
on homogeneous Lie groups, that we use over the
Carnot group $\G$.

Consider two measurable functions $f,h\colon \G\to \R$,
then their convolution is defined by
    \[
   	(f\ast h)(x):=\int_\G f(y)h(y^{-1}\diamond x)
   	=\int_\G f(x\diamond y^{-1})h(y).
    \]
Consider $\psi\colon \G\to \R$ measurable
such that $\int_\G\psi$ is finite. 
From this we get
\[
\psi_t:=\frac{1}{t^Q}\, \psi\circ\delta_{\frac{1}{t}}
\]
where $Q$ is the homogeneous dimension of $\G$ and
$\delta_{\varepsilon}$ the associated family of dilations.
Then,
by \cite[Proposition 1.20]{FollandStein},
we know that if $f\in L^q(\G)$, with $1\leq q\leq+\infty$, then
\[
\left\|f\ast\psi_t-f\, \int_\G\psi\right\|_{L^q(\G)}\to 0,\ \ \m{ as }\ t\to0^+.
\]
In particular this implies that
for a.e~$g\in\G$, $f\ast\psi_t(g)\to f(g)\, \int_\G\psi$ as $t\to0^+$.\newline

We consider 
\[
u\, U_{\varepsilon,a}^p=u\, \varphi_a^p\, T_{\varepsilon,a}^p=u T_{\varepsilon,a}^p+u(\varphi_a^p-1)T_{\varepsilon,a}^p,
\]
%considered as a function over the whole $\G$
%such that $u|_{\G\backslash\Omega}\equiv0$,
and we integrate separately the two terms.
In order to integrate the first one, we set
\[
\psi(g):=T^p(g^{-1})\quad \textrm{ for all } g\in\G,
\]
therefore, $\psi_\varepsilon(g)=\varepsilon^{\frac{Q-2}{2}}\, T^p_{\varepsilon}(g^{-1})$, and
\[
\int_\G u\, T_{\varepsilon,a}^p=\varepsilon^{\frac{Q-2}{2}}\, u\ast \psi_\varepsilon(a)\textrm{ for all } a\in \Omega.
\]
We point out that we are considering $u$ as a function defined
on the whole group $\G$ with $u|_{\G\backslash \Omega}\equiv0$.
As we said above, for a.e.~$a$ the convolution converges to $u(a)\, \int_\G \psi$.
Observe that in this case $\int_\G\psi=\int_\G T^p$ because the inversion
has Jacobian identically equal to $1$ in the case of Carnot groups. Therefore
for a.e.~$a\in\G$,
\[
u\ast \psi_\varepsilon(a)\to u(a)\, \int_\G T^p=:K\ \ \m{as}\ \varepsilon\to0^+.
\]

Regarding the second term, we observe as before that $\varphi_a|_V\equiv1$
in an opportune neighborhood $V$ of $a$, and 
$T_{\varepsilon,a}=\varepsilon^{\frac{Q-2}{2}}\, \widetilde T_{\varepsilon,a}$ with
$\widetilde T_{\varepsilon,a}$ uniformly bounded on $\Omega\backslash V$ as $\varepsilon\to0^+$.
Therefore, for an opportune positive constant $M$,
\[
\int_\Omega u\, (\varphi_a^p-1)\, T_{\varepsilon,a}^p\leq M\, \varepsilon^{\frac{Q-2}{2}\, p}\, \int_\Omega u=O\left(\varepsilon^{\frac{Q+2}{2}}\right).
\]

We may conclude that
\[
\int_\Omega u\, U_{\varepsilon,a}^p=K\, \varepsilon^{\frac{Q-2}{2}}+o\left(\varepsilon^{\frac{Q-2}{2}}\right)
\]
for a.e.~$a\in\Omega$,
and for an opportune positive constant $K$ (depending on $a$).

\cor{Part \eqref{a3}}: we follow again the idea of \cite{BrezisNirenberg}.
As suggested by Lemma \ref{lemmetto}, we have to consider two sub-cases. First when $q=p+1\leq 3$, which implies $Q\geq 5$,
 and secondly the case $q=4$, that gives $Q= 4$. We observe indeed that 
 there exists no non-trivial Carnot groups with $Q<4$
 (trivial here means with step $1$ and therefore a Euclidian space structure);
 {moreover, we have $Q = 4$ only in the case when 
 $\mathbb{G} = \mathbb{H}^1$ is the first Heisenberg group}.
\vspace{0.1cm}

\begin{enumerate}
\item \textsc{Case I): $Q\geq 5$}. We have
\begin{align*}
    |X_\varepsilon|&\leq Ct^pR^p\, \int_{\{u\geq tRU_{\varepsilon,a}\}}u\, U_{\varepsilon,a}^p\\
    |Y_\varepsilon|&\leq CtR\, \int_{\{u<tRU_{\varepsilon,a}\}}u^p\, U_{\varepsilon,a}
\end{align*}
\begin{itemize}
   \item Consider $X_\varepsilon$. Let $\alpha,\beta$ real numbers such that $\alpha+\beta=p$ and $0<\beta<\frac{Q}{Q-2}$.
   As a consequence $u\, t^pR^pU_{\varepsilon,a}^p\leq u^{1+\al}\, t^\beta R^\beta U_{\varepsilon,a}^\beta$ on the set $\{u\geq tRU_{\varepsilon,a}\}$.
    Moreover, by \eqref{eq:Stima_Bonf_Ugu} we have the estimate
    \[
    U_{\varepsilon,a}(g)\leq\frac{\varepsilon^{\frac{Q-2}{2}}}{{|a^{-1}\diamond g|_{\mathbb{G}}}^{Q-2}}\quad \textrm{ for all } g\in\G
    \]
    and
    \[
    |X_\varepsilon|\leq t^\beta R^\beta\int_\Omega u^{1+\al}\, U_{\varepsilon,a}^\beta
    \leq t^\beta R^\beta\int_\Omega u^{1+\al}\, \frac{\varepsilon^{\beta \,\frac{Q-2}{2}}}{{|a^{-1}\diamond g|_{\mathbb{G}}}^{\beta(Q-2) }}.
    \]
    
    By definition $1+\alpha<p+1=2^\star_Q$,
    this implies that $u^{1+\alpha}$ is in $L^1(\Omega)$ because
    $u\in S^1_0(\Omega)\Subset L^{2^\star_Q}(\Omega)$.
    Moreover, $\frac{1}{{|g|_{\mathbb{G}}}^{\beta (Q-2)}}$ is an $L^1$ function as well,
    because
     $\beta(Q-2) <Q$, the homogeneous dimension of $\G$.
    We can conclude by observing that 
    for almost every $a\in\G$,
    \[
    \int_\G u^{1+\alpha}\, \frac{1}{{|a^{-1}\diamond g|_{\mathbb{G}}}^{Q-2}}<\infty.
    \]
    Indeed, this is the convolution $(u^{1+\al}\star f)(a)$ where 
    $f(g):=\frac{1}{{|g^{-1}|_{\mathbb{G}}}^{Q-2}}$,
    and by \cite[Proposition 1.19]{Folland} it has finite
    value for a.e.~$a\in\G$.
  Therefore, we conclude that 
  $|X_\varepsilon|\leq t^{\beta} R^{\beta}\, O\left(\varepsilon^{\beta\,\frac{Q-2}{2}}\right)$ 
  for any~$\beta<\frac{Q}{Q-2}$, that gives
    \[
    |X_\varepsilon|=t^\beta R^\beta \, o(\varepsilon^{\frac{Q-2}{2}})\ \ \ \m{for }\beta:\ 1<\beta<\frac{Q}{Q-2}\ \m{and a.e. }a\in\G.
    \]\newline
    \item Consider $Y_\varepsilon$. Analougously to what we did for $X_\varepsilon$,
     we take $\al,\beta>0$ such that $\al+\beta=p=2^\star_Q-1$ and $\beta<\frac{2}{Q-2}$. 
     Therefore, $u^p\, U_{\varepsilon,a}\leq u^\al\, t^{1+\be}R^{1+\be}U_{\varepsilon,a}^{1+\be}$ 
     on the set $\{u<tRU_{\varepsilon,a}\}$.
    Moreover, from the definition of $Y_\varepsilon$ we get
    \[
    |Y_\varepsilon|\leq t^{1+\beta}R^{1+\beta}\,
    \int_\Omega u^\al\, \frac{\varepsilon^{\frac{Q-2}{2}\,(1+\be)}}{{|a^{-1}\diamond g|_{\mathbb{G}}}^{(Q-2)(1+\be)}}.
    \]
    Observe that $\al<p$, therefore $u^\al\in L^1(\Omega)$.
    Moreover, by definition we have the inequality $1<1+\be<\frac{Q}{Q-2}$, 
    therefore $\frac{1}{{|g|_{\mathbb{G}}}^{(Q-2)(1+\be)}}$ is a function $L^1(\Omega)$.
    As before we have the convergence of the integral,
    \[
    \int_\G u^\al\,\frac{1}{{|a^{-1}\diamond g|_{\mathbb{G}}}^{(Q-2)(1+\be)}}
    \]
    for almost every $a\in \G$,
    therefore $|Y_\varepsilon|\leq t^{1+\be}R^{1+\be}\, O\left(\varepsilon^{\frac{Q-2}{2}\cdot(1+\be)}\right)$ 
    with the constraint $1<1+\be<\frac{Q}{Q-2}$, and as a consequence
    \[
    |Y_\varepsilon|=t^{1+\be}R^{1+\be}\, o\left(\varepsilon^{\frac{Q-2}{2}}\right)\ \ \m{for }\be:\ 0<\be<\frac{2}{Q-2},\ \m{and for a.e. }a\in\G.
    \]
    This concludes the case $Q\geq5$.\newline
\end{itemize}

\item \textsc{Case II): $Q = 4$}. We continue to follow Lemma \ref{lemmetto}. As $p+1=4>3$ we have
\[
|X_\varepsilon+Y_\varepsilon|\leq Ct^2R^2\, \int u^2\, U_{\varepsilon,a}^2.
\]
%The only case of a non-trivial Carnot group with $Q=4$ is
%the Heiseneberg group $\H_1\cong \R^3$, this allows
%a finer estimate of $T_{\varepsilon,a}$. Indeed, there exists an
%explicit 
By the estimate \eqref{eq:Stima_Bonf_Ugu}, we know that
\[
T_{\varepsilon,a}(g)\leq\min\left\{\frac{1}{\varepsilon},\ \frac{\varepsilon}{{|a^{-1}\diamond g|_{\mathbb{G}}}^2}\right\}.
\]
Therefore, if we consider $\al,\be>0$ such that $\al+\be=2$,
we have 
\[
U_{\varepsilon,a}^2\leq \frac{1}{\varepsilon^{\al}}\, \frac{\varepsilon^\be}{{|a^{-1}\diamond g|_{\mathbb{G}}}^{2\be}}
\]
and as a consequence
\[
|X_\varepsilon+Y_\varepsilon|\leq Ct^2R^2\, \varepsilon^{2-2\al}\, \int_\G\frac{u^2}{{|a^{-1}\diamond g|_{\mathbb{G}}}^{2\beta}}.
\]
The integral is again
the convolution of two functions in $L^1(\G)$, therefore
it is finite for almost every $a\in \G$.

We conclude that $|X_\varepsilon+Y_\varepsilon|=t^2R^2\, O(\varepsilon^{2-2\al})$ for any $\al$ such that  $0<\al<2$ 
and almost every $a\in \G$. 
Therefore, by setting $\al<\frac{1}{2}$ we get 
\[
|X_\varepsilon+Y_\varepsilon|=t^2R^2\, o(\varepsilon) \quad \textrm{ as } \varepsilon \to 0^{+}.
\]
\end{enumerate}

\cor{Part \eqref{eq:D}}:
we finally need to estimate
\[
   -\frac{\lambda}{1-\gamma}\left(\int_\Omega u+tRU_{\varepsilon,a}^{1-\gamma}-\int_\Omega u^{1-\gamma}\right)+\lambda\, tR\, \int_\Omega u^{-\gamma}U_{\varepsilon,a}.
\]
We consider a constant $0<\tau<\frac{1}{4}$.
We divide the domain in two parts
\begin{align*}
\din_\varepsilon &:=\{g\in\Omega:\ |a^{-1}\diamond g|_{\mathbb{G}}\leq \varepsilon^{\tau}\}\\
\dout_\varepsilon&:=\{g\in \Omega:\ |a^{-1}\diamond g|_{\mathbb{G}}>\varepsilon^{\tau}\}.
\end{align*}
We start by focusing on the inner domain.
Observe that $u+tRU_{\varepsilon,a}>u>0$,
therefore 
\begin{equation}\label{in}\tag{D-in}
\begin{aligned}
    -\frac{\lambda}{1-\gamma}&\left(\int_{\din_\varepsilon}(u+tRU_{\varepsilon,a})^{1-\gamma}-\int_{\din_\varepsilon} u^{1-\gamma}\right) +\lambda\, tR\, \int_{\din_\varepsilon} u^{-\gamma}U_{\varepsilon,a}\\
    &\leq \lambda\, tR\, \int_{\din_\varepsilon} u^{-\gamma}U_{\varepsilon,a}\\
    &(\m{by Lemma \ref{lem:2.2Haitao}})\\
    &\leq CtR\, \int_{\din_\varepsilon}U_{\varepsilon,a}\\
    &(\m{by estimate \eqref{eq:Stima_Bonf_Ugu}})\\
    &\leq CtR\, \int_{\din_\varepsilon}\frac{\varepsilon^{\frac{Q-2}{2}}}{{|a^{-1}\diamond g|_{\mathbb{G}}}^{Q-2}}\\
    &(\m{by the co-area formula})\\
    &=CtR\, \varepsilon^{\frac{Q-2}{2}}\, \int_0^{\varepsilon^\tau}r\d r=o\left(\varepsilon^{\frac{Q-2}{2}}\right).
\end{aligned}
\end{equation}
Indeed, after observing that for $\varepsilon$ sufficiently small
$\din_\varepsilon\Subset \Omega$,
we can apply point b) of Lemma~\ref{lem:2.2Haitao}
and obtain that $u$ is bounded from below
by a positive constant in $\din_\varepsilon$.
Moreover, as $u$ is in $L^1(\Omega)$ we can use
the co-area formula in the last step.\newline

We treat now the outer domain.
Consider the Taylor expansion of $(u+\theta t R U_{\varepsilon,a})^{1-\gamma}$
with respect to $\theta$ and evaluate it at $\theta=1$ obtaining 
\[
\frac{1}{1-\gamma}\, (u+tRU_{\varepsilon,a})^{1-\gamma}=\frac{u^{1-\gamma}}{1-\gamma}+u^{-\gamma}tRU_{\varepsilon,a}-\gamma (u+\theta_\varepsilon tRU_{\varepsilon,a})^{-1-\gamma}\, t^2R^2U_{\varepsilon,a}^2
\]
with $0<\theta_\varepsilon<1$ a real number.

Therefore, if $\supp(\varphi_a)\Subset\Omega$, we have
\begin{equation}\label{out}\tag{D-out}
\begin{aligned}
    -\frac{\lambda}{1-\gamma}\left(\int_{\dout_\varepsilon}(u+tRU_{\varepsilon,a})^{1-\gamma}-\int_{\dout_\varepsilon} u^{1-\gamma}\right)& +\lambda\, tR\, \int_{\dout_\varepsilon} u^{-\gamma}U_{\varepsilon,a}\\
    &=t^2R^2\, \int_{\dout_\varepsilon}\frac{\gamma\cdot U_{\varepsilon,a}^2}{(u+\theta_\varepsilon tRU_{\varepsilon,a})^{1+\gamma}}\\
    &(\m{by estimate \eqref{eq:Stima_Bonf_Ugu} and Lemma \ref{lem:2.2Haitao}})\\
    &\leq Ct^2R^2\, \int_{\dout_\varepsilon}\frac{\varepsilon^{Q-2}}{\varepsilon^{2\tau (Q-2)}}\\
    &\leq Ct^2R^2\, |\Omega|\, \varepsilon^{(Q-2)(1-2\tau)}\\
    &=o\left(\varepsilon^{\frac{Q-2}{2}}\right).
\end{aligned}
\end{equation}
Here we used the definition of $\dout_\varepsilon$
and the fact that by Lemma \ref{lem:2.2Haitao},
 $u+\theta_\varepsilon tRU_{\varepsilon,a}$ is bounded
 below by a positive constant on any  open set $\mathcal O\Subset\Omega$.
Therefore, we choose the cut-off function $\varphi$
as to have $\supp(\varphi_a)\Subset\Omega$.
This concludes our proof of Lemma \ref{lem:claim2.7haitao}.\fine

\subsubsection{Existence of the second solution}
By Lemma \ref{lem:claim2.7haitao}, we  see that
\begin{equation}\label{eq:eta}
	\widetilde \eta(t) := u_\lambda+tR_0U_{\varepsilon,a}\in \Gamma_\lambda\quad\text{for all $\varepsilon\in (0,\varepsilon_0)$}
\end{equation}
(by enlarging $R_0$ if needed),
and thus $\Gamma_\lambda\neq \varnothing$, as claimed.
Now we observe that, since it is non-empty,
this set $\Gamma_\lambda$ is a complete metric space endowed
with the distance $d$ introduced in~\eqref{maxdistance}.

Now,
we follow the idea of \cite[Lemma 3.5]{BadTar}.
Given a continuous and locally Lipschitz functional
such as $I_\lambda\colon {H_\lambda}\to \R$
we consider its generalized directional derivative
defined in the Appendix at \eqref{def:gdd}.
Observe that
\[
I_\lambda^0(x,y)=\int_\Omega\langle \nabla_\G x,\nabla_\G y\rangle_{\mathfrak{g}_1}-\lambda\int_\Omega x^{-\gamma}y-\int_\Omega x^{2^\star_Q-1}y.
\]
We then apply Lemma \ref{lem:appe}.
The $\Phi$ functional becomes
\[
\Phi(\eta):= \max_{t\in [0,1]}I_\lambda(\eta(t)).
\]
and we denote the minimax level as
\[
\ell_0:=\inf_{\eta\in\Gamma_\lambda}\Phi(\eta).
\]
From the application of the Ekeland's Variational Principle
and the lemma, we get
a sequence $\{\eta_k\}_k\in \Gamma_\lambda$
verifying
\begin{itemize}
	\item $\Phi(\eta_k)\leq \ell_0+\frac{1}{k}$
	\item $\Phi(\eta_k)\leq \Phi(\eta)+\frac{1}{k}d(\eta_k,\eta)$
\end{itemize}
and
we find another sequence $\{t_k\}_k\subset[0,1]$
such that
\begin{itemize}
	\item $\eta_k(t_k)\in H_\lambda$
	\item $I_\lambda(v_k)\to\ell_0$ as $k\to+\infty$
	\item there exists $C>0$ such that for every $w\in H_\lambda$, the following is verified,
	\begin{equation}\label{eq:badtarfin}
		\int_\Omega\langle \nabla_\G v_k,\nabla_\G(w-v_k)\rangle_{\mathfrak{g}_{1}} -\lambda\int_\Omega v_k^{-\gamma}(w-v_k)-\int_\Omega v_k^{2^\star_Q-1}(w-v_k)\geq
		-\frac{C}{k}(1+\|w\|_{S^{1}_{0}(\Omega)}).
	\end{equation}
\end{itemize}
In particular, if we choose $w=2v_k$ in \eqref{eq:gddeta},
we get
\[
\| v_k\|^2_{{S_0^1(\Omega)}}-\|v_k\|^{2^\star_Q}_{L^{2^\star_Q}(\Omega)}-\lambda\int v_k^{1-\gamma}\geq-\frac{1}{k}\max(1,\|v_k\|_{{S_0^1(\Omega)}}).
\]
Summing up this to the fact that $I_\lambda(v_k)\to \ell_0$ as $k\to+\infty$,
we obtain
\begin{equation}\label{eq:gamma0}
	\ell_0+o(1)\geq \left(\frac{1}{2}-\frac{1}{2^\star_Q}\right)\|v_k\|_{{S_0^1(\Omega)}}^2-\lambda\left(\frac{1}{1-\gamma}-\frac{1}{2^\star_Q}\right)\int v_k^{1-\gamma}.
\end{equation}

% the equation above, we get
%\begin{equation}\label{eq:last}
%	\|v_k\|_{S^{1}_{0}(\Omega)}^2-\lambda\int_\Omega v_k^{1-\gamma}-\int_\Omega v_k^{2^\star_Q}\geq
%	-\frac{C}{k}(1+2\|v_k\|_{S^{1}_{0}(\Omega)}).
%\end{equation}
%Using H\"older's and Sobolev's inequality,
%and the fact that $I_\lambda(v_k)\to \ell_0$, we get
%\begin{equation}\label{eq:gamma0}
%	\begin{split}
%		\ell_0+o(1)
%		&= \frac{1}{2}\|v_k\|_{S^{1}_{0}(\Omega)}^2-\frac{\lambda}{1-\gamma}\int_\Omega v_k^{1-\gamma}
%		-\frac{1}{2^\star_Q}\int_\Omega v_k^{2^\star_Q}\\
%		&\geq\left(\frac{1}{2}-\frac{1}{2^\star_Q}\right)\|v_k\|_{S^{1}_{0}(\Omega)}^2-\lambda\left(\frac{1}{1-\gamma}-\frac{1}{2^\star_Q}\right)\int_\Omega v_k^{1-\gamma}
%		-\frac{C}{2^\star_Q\, k}(1+2\|w\|)\\
%		&\geq \left(\frac{1}{2}-\frac{1}{2^\star_Q}\right)\|v_k\|_{S^{1}_{0}(\Omega)}^2-C(\|v_k\|_{S^{1}_{0}(\Omega)}^{1-\gamma}-2\|v_k\|_{S^{1}_{0}(\Omega)}-1).
%	\end{split}
%\end{equation}
%We recall that the constant $C$ depends on $Q$ and $|\Omega|$.

Since $\frac{1}{2}-\frac{1}{2^\star_Q}>0$, 
then if $v_k$ is unbounded in $S^1_0(\Omega)$
we would have that (up to a subsequence) $\|v_k\|_{S^{1}_{0}(\Omega)}\to+\infty$
and this contradicts \eqref{eq:gamma0}.
Therefore $v_k$ is bounded.
From \eqref{eq:gddeta} we get
\begin{equation}\label{eq:last}
	\|v_k\|_{S^{1}_{0}(\Omega)}^2-\lambda\int_\Omega v_k^{1-\gamma}-\int_\Omega v_k^{2^\star_Q}\geq
	-\frac{C}{k}(1+2\|v_k\|_{S^{1}_{0}(\Omega)}).
\end{equation}
\color{black}

Now, we can proceed as in 
the proofs of Lemmas \ref{lem:casoa1} and \ref{lem:casoa2},
to prove that
$v_k$ weakly converges (up to a subsequence) to a weak solution $v_\lambda$
of \eqref{eq:Main_Problem}$_\lambda$ and moreover
\begin{equation}\label{eq:convergence}
	\|v_k-v_\lambda\|_{S^{1}_{0}(\Omega)}^2-\|v_k-v_\lambda\|_{L^{2^\star_Q}}^{2^\star_Q}=o(1)\ \ \m{as}\ k\to+\infty.
\end{equation}

To complete this case, it remains to show that $v_\lambda\not \equiv u_\lambda$.
Observe that for any $\eta\in \Gamma_\lambda$,
\[
\|\eta(0)-u_\lambda\|_{S^{1}_{0}(\Omega)}=0\ \ \m{and}\ \ \|\eta(1)-u_\lambda\|_{S^{1}_{0}(\Omega)}>r_1,
\]
therefore there exists $t_\eta\in [0,1]$
such that $\|\eta(t_\eta)-u_\lambda\|_{S^{1}_{0}(\Omega)}=r_1$.
Since we are assuming \eqref{eq:case2},
we have
\[
\ell_0=\inf_{\Gamma_\lambda}\Phi(\eta)\geq\inf_{\Gamma_\lambda}I_\lambda(\eta(t_\eta))
\geq\inf_{\|u-u_\lambda\|_{S^{1}_{0}(\Omega)}=r_1} I_\lambda(u)>I_\lambda(u_\lambda)
\]
where the last infimimum is taken among the $u\in H_\lambda$
such that $\|u-u_\lambda\|_{S^{1}_{0}(\Omega)}=r_1$. At the same time,
if we consider $\widetilde\eta$ defined in \eqref{eq:eta},
then by \eqref{eq:lemma2.7Haitao}
we get,
\[
\ell_0\leq \Phi(\widetilde \eta)=\max_{t\in [0,1]}I_\lambda(\widetilde \eta(t))<
I_\lambda(u_\lambda)+\frac{1}{Q}S_\G^{Q\slash2}.
\]
Summing up,
\begin{equation}\label{eq:sumup}
	I_\lambda(u_\lambda)<\ell_0<I_\lambda(u_\lambda)+\frac{1}{Q}S_\G^{Q\slash2}.
\end{equation}
Observe that, since $v_k\to v_\lambda$ weakly in $S^1_0(\Omega)$, then
the equalities \eqref{eq:case1concl1}-\eqref{eq:case1concl2} hold also in this context.
Therefore, considering also \eqref{eq:sumup} and the fact that $I_\lambda(v_k)\to \ell_0$,
we get (for $k$ sufficiently large)
\begin{equation}\label{eq:sumup2}
	\begin{split}
		\frac{1}{2}&\|v_k-v_\lambda\|_{S^{1}_{0}(\Omega)}^2-\frac{1}{2^\star_Q}\|v_k-v_\lambda\|_{L^{2^\star_Q}(\Omega)}^{2^\star_Q}\\
		&=\frac{1}{2}(\|v_k\|_{S^{1}_{0}(\Omega)}^2-\|v_\lambda\|_{S^{1}_{0}(\Omega)}^2)-\frac{1}{2^\star_Q}(\|v_k\|^{2^\star_Q}_{L^{2^\star_Q}(\Omega)})+o(1)\\
		&=I_\lambda(v_k)-I_\lambda(u_\lambda)+o(1)\\
		&=\ell_0-I_\lambda(u_\lambda)+o(1)\\
		&<\frac{1}{Q}S_\G^{Q\slash2}-\delta_0
	\end{split}
\end{equation}
for some $\delta_0>0$ such that the last term is positive.
From \eqref{eq:convergence}, \eqref{eq:sumup} and \eqref{eq:sumup2},
by reasoning as in \cite[Proposition 3.1]{Tarantello},
we get that $v_k\to v_\lambda$ strongly in $S^1_0(\Omega)$.
This, together with \eqref{eq:last} and \eqref{eq:gamma0}, gives
\[
I_\lambda(u_\lambda)<\gamma_0=\lim_{k\to+\infty}I_\lambda(v_k)=I_\lambda(v_\lambda),
\]
thus implying that $u_\lambda\not\equiv v_\lambda$.\newline
\medskip

Gathering all the results established so far, we can finally provide the
\begin{proof}[Proof of Theorem \ref{thm:main}]
	Let $\Lambda$ be as in \eqref{eq:DefinitionLambda}, that is, 
    $$\Lambda := \sup \{ \lambda >0: \eqref{eq:Main_Problem}_\lambda \textrm{ admits a weak solution}\}.$$
    Taking into account Lemma \ref{lem:Lambdafinito}, we know that $\Lambda\in (0,+\infty)$; moreover,
    by combining Lemma \ref{lem:2.3Haitao} with the 
    computations in Sections \ref{subsec:CaseI}-\ref{subsec:CaseII}, we know that
    \vspace*{0.1cm}
    
    i)\,\,there exist \emph{at least two distinct weak solutions} $u_\lambda,v_\lambda$ of
    \eqref{eq:Main_Problem}$_\lambda$ for $\lambda\in(0,\Lambda)$;
    
    ii)\,\,there exists \emph{at least one weak solution} $u_\lambda$ of
    \eqref{eq:Main_Problem}$_\Lambda$.
    \vspace*{0.1cm}
    
    \noindent Hence, assertions a)-b) in the statement of the theorem are established. 
    
    Finally,
    by the very definition of $\Lambda$ we derive that \eqref{eq:Main_Problem}$_\lambda$
    \emph{does not admit} weak solutions when $\lambda > \Lambda$; this establishes also
    assertion c), and the proof is complete.
\end{proof}

\appendix
\section{Generalized directional derivative}\label{app:gdd}

We introduce the generalized directional derivative. Consider
the {convex cone $H_\lambda\subseteq S_0^1(\Omega)$} defined
at \eqref{def:tlambda}, and
consider any continuous and locally Lipschitz functional $I\colon {H_\lambda}\to \R$.
Moreover, to simplify the notation, in what follows we set
$$\|u\| = \|u\|_{S_0^1(\Omega)}\quad \text{for every $u\in H_\lambda$}.$$
The \cor{generalized directional derivative} at 
{$x,v\in H_\lambda$} is defined as
\begin{equation}\label{def:gdd}
	I^0(x,y):=\limsup_{\|h\|\to0,\ \rho\to0^+}\frac{I(x+h+\rho y)-I(x+h)}{\rho},
\end{equation}
\begin{lemma}\label{lem:subadd}
	For fixed $x$, the map $I\mapsto I^0(x,v)$ is subadditive with respect to $v$.
\end{lemma}
\proof
We use the notation $I^*_{x,h,v}(\rho):=I(x+h+\rho v)-I(x+h)$.
Therefore,
\[
I^*_{x,h,v+w}(\rho)=I^*_{x,h+\rho v,w}(\rho)+I^*_{x,h,v}(\rho).
\]
Then, the lemma follows by observing that 
\[
\limsup_{\|h\|\to0,\ \rho\to0^+}I^*_{x,h+\rho v,w}(\rho)=\limsup_{\|h\|\to0,\ \rho\to0^+}I^*_{x,h,w}(\rho)
\]
because any $h'$ sufficiently small can be represented
as $h+\rho v$ for $h$ and $\rho$ sufficiently small.\fine

Moreover, we suppose that $I$ respects property \eqref{eq:case2}
and we apply the Ekeland's principle to the continuous functional 
\[
\Phi(\eta):=\max_{t\in[0,1]}I(\eta(t)),
\]
setting $ \ell:=\inf_{\eta\in \Gamma_\lambda}\Phi(\eta)$.
As a direct consequence, there exists
a sequence $\{\eta_k\}_k\subset \Gamma_\lambda$
verifying
\begin{align}
	\Phi(\eta_k)&\leq \ell+\frac{1}{k}\label{Phi1}\\
	\Phi(\eta_k)&\leq \Phi(\eta)+\frac{1}{k}d(\eta_k,\eta)\ \ \ \forall\eta\in \Gamma_\lambda\label{Phi2}
\end{align}
\begin{lemma}\label{lem:appe}
	For every $k$ consider $\Lambda_k:=\{t\in (0,1):\ I(\eta_k(t))=\max_{s\in[0,1]}I(\eta_k(s))\}$,
	then for every $k$ there exists $t_k\in \Lambda_k$ such that, $v_k:=\eta_k(t_k)$,
	we have
	\begin{equation}\label{eq:gddeta}
		I^0\left(v_k;\ \frac{w-v_k}{\max(1, \|w-v_k\|)}\right)\geq-\frac{1}{k}\ \ \ \forall w\in H_\lambda.
	\end{equation}
\end{lemma}
\proof
Observe that $\Lambda_k$ is always a compact subset of $(0,1)$.
We star by considering the case when $\Lambda_k$ is a single point $t_k$.

We argue by contradiction supposing that there exists $\omega\in H_\lambda$
such that 
\begin{equation}\label{assumption}
	I^0\left(\eta_k(t_k);\ \frac{\omega-\eta_k(t_k)}{\max(1,\|\omega-\eta_k(t_k)\|}\right)<-\frac{1}{k}.
\end{equation}
Consider a continuous cut-off function $g\colon[0,1]\to [0,1]$ such that $g\equiv 1$ in a neighborhood
of $t_k$ and $g(0)=g(1)=0$. We also define $m_k(t):=\max(1,\|\omega-\eta_k(t)\|)$
and we introduce
\[
\eta_{k,\varepsilon}(t):=\eta_k(t)+\frac{g(t)\varepsilon}{m_k(t)}(\omega-\eta_k(t)).
\]
By construction $\eta_{k,\varepsilon}(t)\in \Gamma_\lambda$ for every $t\in [0,1]$.
Therefore by \eqref{Phi2} we obtain
\begin{equation}\label{ineq:badtar}
	\max_{t\in [0,1]} I(\eta_k(t))\leq \max_{t\in [0,1]}I(\eta_{k,\varepsilon}(t))+\frac{\varepsilon}{k} \max_{t\in[0,1]}\frac{g(t)\|\omega-\eta_k(t)\|}{m_k(t)}.
\end{equation}
If $t_{k,\varepsilon}$ is a point where $I(\eta_{k,\varepsilon})$ reaches its max,
then there exists $\varepsilon_n\to0$ sequence such that $t_{k,\varepsilon_n}\to t_k$.
Therefore $g(t_{k,\varepsilon_n})=1$ for $n$ sufficiently big. We denote by $v_{k,n}:=\eta_k(t_{k,\varepsilon_n})$
and $m_{k,n}:=\max(1,\|\omega-v_{k,n}\|)$ (observe that this is a scalar value),
then
\begin{equation}\label{ineq:badtar2}
	I(v_{k,n})\leq I(v_k)\leq I(\eta_{k,\varepsilon_n}(t_{k,\varepsilon_n}))+\frac{\varepsilon_n}{k}
	=I\left(v_{k,n}+\frac{\varepsilon_n}{m_{k,n}}(\omega-v_{k,n})\right)+\frac{\varepsilon_n}{k}.
\end{equation}
where the first inequality comes from the fact that $v_k=\eta_k(t_k)$
is the max for the $I$ value for any $\eta_k(t)$,
and the second inequality comes from \eqref{ineq:badtar}.\newline

We set $m_k^*:=\max(1,\|\omega-v_k\|)$, then
by the continuity of $\eta_k$,  we get
$v_{k,n}\to v_k$ in $S^1_0(\Omega)$ and $m_{k,n}\to m_k^*$
as $n\to\infty$.
Therefore, we can re-write \eqref{ineq:badtar2} obtaining
\[
\frac{1}{\varepsilon_n}\left(I\left(v_k+h_n^*+\varepsilon_n \frac{\omega-v_k}{m_k^*}\right)-I(v_k+h_n)\right)\geq-\frac{1}{k}
\]
with
\begin{align*}
	h_n&=v_{k,n}-v_k\\
	h_n^*&=v_{k,n}-v_k+\frac{\varepsilon_n}{m_{k,n}}(\omega-v_{k,n})-\frac{\varepsilon_n}{m_k^*}(\omega-v_k).
\end{align*}
By what we proved, we derive that $h_n^*-h_n=o(\varepsilon_n)$ as $n\to\infty$ and
this in turn implies
\[
I^0\left(v_k,\frac{\omega-v_k}{m_k^*}\right)\geq-\frac{1}{k},
\]
thus contradicting the assumption \eqref{assumption}.\newline

If $\Lambda_k$ has more than one element, we have to correct the above proof.
For any $t_k\in \Lambda_k$, consider the $\omega\in H_\lambda$
that satisfies \eqref{assumption} for $t_k$.\newline

\cor{Claim.} 
For any $t_k\in \Lambda_k$ there exists a neighborhood
$J$ of $t_k$ such that
\begin{equation}\label{assumption2}
	I^0\left(\eta_k(s),\frac{\omega-\eta_k(s)}{\max(1,\|\omega-\eta_k(s)\|)}\right)<-\frac{1}{k}\ \ \ \ \forall s\in J.
\end{equation}\newline

Let's suppose that the claim above is correct and conclude.
By compactness of $\Lambda_k$, there exist $t_k^{(1)},\dots,t_k^{(r)}$ elements of $\Lambda_k$
and their associated neighborhoods $J_1,\dots,J_r$, such that
\eqref{assumption} is verified on $J_i$ by some $\omega_i\in H_\lambda$ 
and $\Lambda_k\subset \cup_i J_i\subset [0,1]$.
For an opportune partition of the unit $g_1,\dots,g_r$ associated to the $J_i$s,
we can update $\omega$ to a continuous function $\omega\colon J:=\cup_iJ_i\to H_\lambda$,
\[
\omega(s):=\sum_{i=1}^rg_i(s)\omega_i.
\]
As a consequence of the subadditivity proved in Lemma \ref{lem:subadd},
we have
\[
I^0(\eta_k(s),\omega(s)-\eta_k(s))\leq -\frac{1}{k} \sum_{i = 1}^r g_i(s) \max(1,\|\omega_i-\eta_k(s)\|).
\]
Moreover,
\[
\sum_{i = 1}^r g_i(s)\|\omega_i-\eta_k(s)\|\geq \left\|\sum_{i = 1}^r g_i(s)(\omega_i-\eta_k(s))\right\|=\|\omega(s)-\eta_k(s)\|,
\]
and therefore we get
\begin{equation}\label{assumption3}
	I^0\left(\eta_k(s),\frac{\omega(s)-\eta_k(s)}{\max(1,\|\omega(s)-\eta_k(s)\|)}\right)<-\frac{1}{k}\ \ \ \forall s\in J.
\end{equation}

From here the proof can continue as before if we replace every appearance of $\omega$ with 
$\omega(t)$ (with the opportune value of $t$)
and if we take the cut-off function $g$ to be $1$ on $\Lambda_k$ and $0$ outside $J$.

In this case $t_{k,\varepsilon_n}\to t_k$ for \cor{some} $t_k\in \Lambda_k$,
$v_k=\eta_k(t_k)$
and therefore we derive \eqref{ineq:badtar}
with $\omega(t_{k,\varepsilon_n})$ in place of $\omega$.
We also have to update the definition of $h_n^*$ using
$\omega(t_{k,\varepsilon_n})$ when necessary.\newline

It remains to prove the claim above.
Consider the function
\[
{\bf{w}}(s):=\frac{\omega-\eta_k(s)}{\max(1,\|\omega-\eta_k(s)\|)}
\]
which by definition is continuous.
By \eqref{assumption}, there exists $\widetilde \varepsilon, \widetilde \delta,\widetilde \rho$ all positive
such that % for any $h\in H_\lambda$ with $\|h\|<\widetilde \delta$ and for any $\rho<\widetilde \rho$,
\begin{equation}\label{eq:barbatrucco}
	\sup_{\|h\|<\widetilde \delta,\ \rho<\widetilde \rho}\frac{I(\eta_k(t_k)+h+\rho{\bf w}(t_k))-I(\eta_k(t_k)+h)}{\rho}<-\frac{1}{k}-\widetilde \varepsilon.
\end{equation}
We give a first definition of $J$ by imposing $\|\eta_k(s)-\eta_k(t_k)\|<\frac{\widetilde \delta}{2}$.
Then, for any $h\in H_\lambda$ with $\|h\|<\frac{\widetilde \delta}{2}$ we have
\[
\eta_k(s)+h=\eta_k(t_k)+h',\ \ \m{with}\ \|h'\|<\widetilde \delta.
\]
Therefore we have,
\begin{align*}
	I^0(\eta_k(s),{\bf w}(s))&\leq \sup_{\|h\|<\frac{\widetilde \delta}{2},\ \rho<\widetilde \rho}
	\frac{I(\eta_k(s)+h+\rho {\bf w}(s))-I(\eta_k+h)}{\rho}\\
	&\leq \sup_{\|h'\|<\widetilde \delta,\ \rho<\widetilde \rho}
	\frac{I(\eta_k(t_k)+h'+\rho {\bf w}(s))-I(\eta_k(t_k)+h')}{\rho}\\
	&=\sup_{\|h'\|<\widetilde \delta,\ \rho<\widetilde \rho}\Biggl(
	\frac{I(\eta_k(t_k)+h'+\rho{\bf w}(s))-I(\eta_k(t_k)+h'+\rho{\bf w}(t_k))}{\rho}\\
	\vspace{12pt}
	&\quad \quad \quad \quad \quad \quad \quad
	+\frac{I(\eta_k(t_k)+h'+\rho{\bf w}(t_k))-I(\eta_k(t_k)+h')}{\rho}\Biggl)
\end{align*}
Observe that the first therm (in the last sum) can be made arbitrarily small (up to stretching~$J$) by using that 
$I$ is locally Lispchitz, while the second term respects \eqref{eq:barbatrucco}. Therefore the claim is proved
and this concludes our proof.\fine

 \end{document}